\documentclass[11 pt]{amsart}
\newtheorem{theorem}{Theorem}[section]
\newtheorem{lemma}[theorem]{Lemma}

\theoremstyle{definition}

\newtheorem{proposition}[theorem]{Proposition}

\theoremstyle{remark}
\newtheorem{remark}[theorem]{Remark}
\newtheorem{corollary}[theorem]{Corollary}

\numberwithin{equation}{section}

\def\epi{\varepsilon}

\newcommand{\BR}{\mathbb{R}}
\newcommand{\p}{\partial}
\newcommand{\wt}{\widetilde}

\newcommand{\CT}{\mathcal{T}}

\newcommand{\CX}{\mathcal{X}}
\newcommand{\CM}{\mathcal{M}}
\newcommand{\CH}{\mathcal{H}}
\newcommand{\CW}{\mathcal{W}}
\newcommand{\CG}{\mathcal{G}}
\newcommand{\CF}{\mathcal{F}}
\newcommand{\CA}{\mathcal{A}}
\newcommand{\CL}{\mathcal{L}}

\usepackage{amsmath,amsthm,verbatim,amssymb,amsfonts,amscd, graphicx,txfonts,paralist}
\usepackage{graphics}
\usepackage{fourier}

\begin{document}
\title[Invariant Manifold]{Invariant Manifolds of traveling waves of the 3D Gross-Pitaevskii equation in the energy space }
\author{Jiayin Jin} 
\address{School of Mathematics, Georgia Institute of Technology}
\email{jin@math.gatech.edu}
\author{Zhiwu Lin${}^\dagger$}
\address{School of Mathematics, Georgia Institute of Technology}
\email{zlin@math.gatech.edu}
\thanks{${}^\ddagger$Work supported in part by NSF-DMS 1411803 and NSF-DMS 1715201}
 \author{Chongchun Zeng${}^\ddagger$}
 \address{School of Mathematics, Georgia Institute of Technology}
\email{chongchun.zeng@math.gatech.edu}
\thanks{${}^\ddagger$Work supported in part by NSF-DMS 1362507}

\begin{abstract}
We study the local dynamics near general unstable traveling waves of the 3D Gross-Pitaevskii equation in the energy space by constructing smooth local invariant  center-stable, center-unstable and center  manifolds. We also prove that (i) the center-unstable manifold attracts nearby orbits exponentially before they get away from the traveling waves along the center directions and (ii) if an initial data is not on the center-stable manifolds, then the forward flow will be ejected away from traveling waves exponentially fast. Furthermore, under a non-degenerate assumption, we show the orbital stability of the traveling waves on the center manifolds, which also implies the local uniqueness of the local invariant manifolds. Our approach based on a geometric bundle coordinates should work for a general class of Hamiltonian PDEs. 
\end{abstract}
\maketitle

\section{Introduction}
Consider the Gross-Pitaevskii (GP) equation 
\begin{equation}\label{GP}
 iu_t+\Delta u +(1-\vert u\vert ^{2})u=0,\hspace{0.5in} u: {\mathbb{R}\times\mathbb{R}^{3}} \to \mathbb{C},
\end{equation}
where $u$ satisfies the boundary condition $\vert u\vert \rightarrow 1$ as $\vert x\vert \rightarrow \infty$. The Gross-Pitaevskii equation arises in various physical problems such as superconductivity, superfluidity in Helium II, and Bose-Einstein condensate (for example \cite{AHMNPTB,BR2}). Formally, the Gross-Pitaevskii equation is a Hamiltonian PDE associated to the energy
\begin{equation} \label{E:energy} 
E(u)=\frac{1}{2}\int_{\mathbb{R}^{3}}\vert \nabla u\vert ^{2}dx + \frac{1}{4}\int_{\mathbb{R}^{3}}(1-\vert u\vert ^{2})^{2}dx,
\end{equation}
and the energy space is 
\begin{equation} \label{E:energy-s}
X_{0}=\{u\in H^{1}_{loc}(\mathbb{R}^{3}): \nabla u \in L^{2}(\mathbb{R}^{3}), 1-\vert u\vert ^{2}\in L^{2}(\mathbb{R}^{3})\}.
\end{equation}
The well-posedness of (GP) in $X_{0}$ was proved by G{\'e}rard \cite{G1}.  From the definition of $X_{0}$,  it is clear that the real part and imaginary part of a function in $X_{0}$ may have different spatial decaying rates, which makes the analysis of this equation quite different from the classical NLS. 

Related to the translation invariance of (GP), the momentum 
\begin{equation} \label{E:momentum} 
P(u)=\frac{1}{2}\int_{\mathbb{R}^{3}}\langle i\nabla u, u-1\rangle dx
\end{equation}
is also formally conserved. We denote each component of $P(u)$ as 
\begin{equation}
P_{j}(u)=\frac{1}{2}\int_{\mathbb{R}^{3}}\langle i\partial_{x_{j}} u, u-1\rangle dx=-\int_{\mathbb{R}^{3}}\langle u_{1}-1, \partial_{x_{j}}u_{2}\rangle dx
\end{equation}
for $j=1,2,3$. The corresponding relative equilibria are traveling wave solutions to (GP) which are  solutions in the form $u(t,x)= U_{c}(x-ct)$ where $c\in \mathbb{R}^{3}$ and $U_{c}$ satisfies the equation 
\begin{equation}\label{twGP}
-ic\cdot \nabla U_{c}+\Delta U_{c} + (1-|U_{c}|^{2})U_{c}=0.
\end{equation}
Due to the rotational invariance of (GP), for existence and stability/instability of traveling waves we only need to consider those traveling waves traveling in $x_{1}$ direction, i.e., $u(t,x)=U_{ae_{1}}(x-ae_{1}t)$, where $e_{1}=(1,0,0)^T$.

Traveling waves with finite energy play a very important role in the dynamics of (GP).  In a series of papers \cite{BM,BGM,GR,JPR, JR}, the existence, some qualitative properties, and the stability of traveling waves have been studied formally. A rigorous mathematical study 
was initiated by B{\'e}thuel and Saut in \cite{BS}, in which they proved the existence of traveling waves for 2D (GP), followed by \cite{BGS1,BGS2, BOS,C, Gra1,Gra2,M,M2}. Especially, Mari{\c{s}} \cite{M} constructed full branch of subsonic traveling waves for 3D (GP) by minimizing the energy-momentum functional subject to a Pohozaev type constraint. 

The stability and local dynamics near traveling waves have also been studied. On the one hand, 
these include the orbital stability of some traveling waves based on the variational structure \cite{CM}. On the other hand, 
a Grillakis-Shatah-Strauss type stability criterion was formulated based on numerics and heuristic arguments  
\cite{BR, JPR} and then later rigorously proved \cite{LWZ} for the traveling waves constructed by Mari{\c{s}} \cite{M}. Furthermore, Lin, Wang, and Zeng also proved that the nonlinear orbital stability of traveling waves on the lower (stable) branch by a Lyapunov functional argument, and  nonlinear instability of traveling waves on the upper (unstable) branch by constructing their unstable manifolds which consist of some rather smooth functions in $X_0$. 

Given any traveling wave solution $U_c = (u_c, v_c)$ of \eqref{GP} with traveling speed $c\in \BR^3$, $|c| \in (0, \sqrt{2})$, its spatial translations form a 3D manifold 
\[
\CM=\{U_{c}(\cdot+y): y\in \mathbb{R}^{3} \}
\]
of  traveling waves. To analyze the dynamics near $\CM$, we write (GP) in the traveling frame $u(t,x)=U(t,x-ct)$, where $U(t,x)$ satisfies
\begin{equation}\label{tfeq1}
i\partial_{t}U-ic \cdot \nabla U+\Delta U+(1-|U|^{2})U=0.
\end{equation}
It is clear that $U_{c}$ is a steady state of \eqref{tfeq1}. Linearizing \eqref{tfeq1} at $U_c$, one has 
\begin{equation}\label{le}
\partial_{t}U= JL_{c}U, \qquad J=\begin{pmatrix}0&1\\-1&0\end{pmatrix}, \quad L_c = (E+ c\cdot P)''(U_c).
\end{equation}
A more explicit expression of $L_c$ can be found in \eqref{lcy}. Under spatial decay assumption \eqref{E:TW-decay} of $U_c$, it is straight forward to verify that the tangent space of the energy space $X_0$ at $U_c$ is $X_{1}=H^{1}\times \dot{H}^{1}$, where naturally the linearized equation \eqref{le} should be considered. The linearized energy quadratic form  $L_c:X_1 \to X_1^*$ is bounded, symmetric, and uniformly positive except in finite many directions.  
As here the symplectic operator $J^{-1} = J^*=-J: X_1^* \to X_1$ is not bounded and thus the general framework of Grillakis-Shatah-Strauss \cite{GSS1, GSS2} does not apply to \eqref{le}, nevertheless the recent results in \cite{LZ17} are applicable (see Section \ref{S:coord}). Consequently, \eqref{le} satisfies the following exponential trichotomy property, which was derived  in \cite{LWZ} for those traveling waves obtained in \cite{M}  under some additional assumptions. 

\begin{lemma} 
There exist $C, \lambda, d_1 >0$ and closed subspaces $X^c, X^u, X^s$ of $X_1$ invariant under $e^{JL_c}$ such that $X_1 = X^u \oplus X^c \oplus X^s$, $d=\dim X^u = \dim X^s <\infty$, and 
\begin{align*}
&\Vert e^{JL_c}|_{X^s} \Vert \le C e^{\lambda t}, \;\; \forall t\ge 0, \;\quad \Vert e^{JL_c}|_{X^u} \Vert \le C e^{\lambda t}; \; \forall t\le 0,\\
& \Vert e^{JL_c}|_{X^c} \Vert \le C(1+|t|^{d_1}) , \;\; \forall t \in \BR. 
\end{align*}
\end{lemma} 

The exponential trichotomy both describes the linear dynamics near the traveling waves and provides a framework to analyze the local nonlinear dynamics. 

If $X^{u, s} =\{0\}$, $U_c$ is spectrally stable. Under  some additional assumptions the nonlinear orbital stability of $\CM$ can be obtained from the spectral stability (see, for example, \cite{LWZ}). If $U_c$ is spectrally unstable with $d>0$, conceptually one expects the existence of locally invariant manifolds which can be viewed as the deformation from the invariant subspaces under the perturbation of small nonlinear terms.  Here the local invariance of a submanifold manifold $\mathcal{N}$ means that, for any initial value $U(0)$ in the interior of $\mathcal{N}$, there exists $T>0$ such that the solution $U(t) \in \mathcal{N}$ for $t\in (-T, T)$, and thus solutions can exit $\mathcal{N}$ only through its boundary. Here the locally invariant submanifolds related to the exponential trichotomy are the unstable and stable manifolds of $U_c$ and the center-unstable, center-stable, and center manifolds of $\CM$. The former two manifolds contain $U_c$ and are tangent to $X^u$ and $X^s$ at $U_c$, respectively, while the latter three ones should contain $\CM$ and be tangent to $X^u \oplus X^c$, $X^s \oplus X^c$, and $X^c$. Here are some comments on their dynamic significance.   \\
1.) Firstly, to some extent, the nonlinear dynamics in these invariant manifolds are reflected, or even exactly conjugate in the case of the unstable and stable manifolds, by the corresponding linear ones. For example, the unstable manifolds can be characterized as the set of all solutions near $U_c$ which converge to $U_c$ backward in $t$ and grow out of a prefixed neighborhood of $U_c$ along the direction of $X^u$ as $t$ increases at least at certain exponential growth rate. This immediately provides a stronger result than mere nonlinear instability. \\
2.) Secondly, these invariant manifolds provide some framework to organize the local dynamics. One may imagine that for a typical initial value near $\CM$, its trajectory would first approach the center-unstable manifold along the direction of $X^s$ and then exit the neighborhood of $\CM$ along the $X^u$ direction, constituting a saddle type dynamics. \\
3.) Thirdly, there is numerical evidence \cite{BR} showing that after leaving a neighborhood of unstable traveling waves (upper branch), the flows either scatter to stable traveling waves (lower branch) or scatter to constant states. Under non-degeneracy conditions {\bf (H1-2)} given in Section \ref{S:non-deg},  there is orbital stability in the center-stable manifolds which provides a third type of dynamics, where orbits stay closed to traveling wave manifolds not easily observed in numerics. In the case the center-stable manifold is co-dim 1, there is a good chance that it serves as the boundaries between the two different types of asymptotic dynamics. Similar scenario has been proved for some models such as the Klein-Gordon and NLS when combined with other tools such as the viral identity \cite{NS3, NS4}.

Under some additional conditions, the 1-dim unstable and stable manifolds in $X_0 \cap H^3$ of an unstable traveling wave $U_c$ were constructed in \cite{LWZ}. As these two invariant curves in the phase space represent very low dimensional special structures, it is indeed more desirable for them to have extra properties such as higher $H^k$ regularity, $k>1$. 

The main goal of this paper is to obtain the existence, smoothness, and some dynamic properties of the center-stable, center-unstable, and the center manifolds. In contrast to the finite dimensional stable and unstable directions, as the center subspace contains all but finitely many dimensions, it is more preferable for the these invariant manifolds to be constructed as finite co-dimensional submanifolds in the energy space $X_0$. Moreover, on a center manifold where the topology is induced by the same one of the energy space, the energy conservation provides a crucial control on the nonlinear dynamics. 

\begin{theorem} 
Let $U_c$ be a traveling wave of \eqref{GP} so that spatial decay condition \eqref{E:TW-decay} is satisfied and $d=\dim X^u = \dim X^s>0$, then 
\begin{enumerate}
\item There exist locally invariant co-dim $d$ center-unstable manifold $\CW^{cu}$ and center-stable manifold $\CW^{cs}$, both containing $\CM$.
\item $\CW^{cu}$ and $\CW^{cs}$ are translation invariant, i.e. if $U \in \CW^{cu, cs}$, then $U(\cdot +y) \in \CW^{cu, cs}$ for any $y \in \BR^3$.
\item Let $\psi$ be defined in \eqref{E:psi}, then $\psi^{-1} (\CW^{cu})$ and $\psi^{-1}(\CW^{cs})$ are smooth co-dim $d$ submanifolds of $X_1$. Moreover, the tangent spaces of \; $\psi^{-1} (\CW^{cu, cs})$ at $\psi^{-1}(U_c)$ are equal to $(D\psi^{-1}) (U_c) X^{cu, cs}$. 
\item Orbits in a small neighborhood of $\CM$ are exponentially attracted to $\CW^{cu}$ as $t$ increases (before they possibly exit the neighborhood in the center directions), while repelled exponentially by $\CW^{cs}$. 
\end{enumerate}
\end{theorem}

\begin{corollary} 
$\CW^{cu}$ and $\CW^{cs}$ intersect transversally along $\CM$ and the intersection $\CW^c = \CW^{cu} \cap \CW^{cs}$ gives a locally invariant smooth center manifold tangent to $X^c$ along $\CM$, which is also invariant under spatially translation. Moreover, $\CW^c$ exponentially attracts nearby orbits in $\CW^{cs}$ and exponentially repels those in $\CW^{cu}$ (before they possibly exit a neighborhood of $\CM$ in the center directions).
\end{corollary}

More detailed statements of the results can be found in Section \ref{S:InMa} and \ref{S:smooth}. 

As $X_0$ is not a flat space, we identify $X_0$ with $X_1$ through a coordinate map $\psi: X_1 \to X_0$ given in \eqref{E:psi}, which is borrowed from \cite{G2}. 
So the above invariant manifolds are smooth manifolds in the sense that their images under $\psi^{-1}$ are smooth submanifolds in $X_1$. Statement (4) in the above theorem implies that $\CM$ is orbitally unstable. 

It is well-known that center-unstable manifolds, et. al. are not unique even for ODEs and the dynamics,  including the stability, on the center manifold is rather subtle even though all spectral instability has already been excluded. However, under non-degeneracy conditions {\bf (H1-2)} given in Section \ref{S:non-deg} which ensure the uniform positivity of $L_c$ on $X^c$ and include the cases considered in \cite{LWZ}, we have 

\begin{theorem} 
Assume {\bf (H1-2)} in addition, then there exist $C, \delta >0$ such that 
\begin{enumerate}
\item Given any initial value $U(0)$ such that $ \psi^{-1}\big(U(0)\big)$ is in a $C^{-2} \delta$-neighborhood of \;$\psi^{-1}(\CM)$, then $U(0) \in \CW^{cu, cs}$ if and only if $\psi^{-1}\big(U(t)\big)$ is in a $\delta$-neighborhood of $\psi^{-1}(\CM)$ for all $\mp t \ge 0$, and $U(0) \in \CW^{c}$ if and only if $ \psi^{-1}\big(U(t)\big)$ is in a $\delta$-neighborhood of $\psi^{-1}(\CM)$ for all $t \in \BR$.
\item $\CM$ is orbitally stable in $\CW^c$. 
\end{enumerate}
\end{theorem}

Due to the above characterizations, we obtain

\begin{corollary}
Assuming {\bf (H1-2)}, $\CW^{cu, cs, c}$ are locally unique. 
\end{corollary}

More detailed statements of the results can be found in Section \ref{S:non-deg}.

Among previous results on local invariant manifolds of relative equilibria of dispersive PDEs, Bates and Jones \cite{BJ} proved a general theorem for the existence of Lipschitz locally invariant manifolds of equilibria for semilinear PDEs by an energy argument, and then applied it to the Klein-Gordon equation. In \cite{Sc}, Schlag constructed a co-dimension $1$ center-stable manifold of the manifold of ground states for 3D cubic NLS in $W^{1,1}(\mathbb{R}^{3})\bigcap W^{1,2}(\mathbb{R}^{3})$ under an assumption that the linearization of NLS at each ground has no embedded eigenvalue in essential spectrum and proved scattering on the center-stable manifold.  Later, this result was improved by Beceanu \cite{Be1, Be2} who constructed  center-stable manifolds in $W^{1,2}(\mathbb{R}^{3})\bigcap |x|^{-1}L^{2}(\mathbb{R}^{3})$ and in critical space $\dot{H}^{1/2}(\mathbb{R}^{3})$ . Similar results were obtained in Krieger and Schlag \cite{KS} for the supercritical 1D NLS. Nakanishi and Schlag \cite{NS4} constructed a center-stable manifold of ground states for 3D cubic NLS in the energy space with a radial assumption by using the framework in Bates and Jones \cite{BJ}. Nakanishi and Schlag \cite{NS1} constructed center-stable manifolds of ground states for nonlinear Klein-Gordon equation without radial assumption, following a graph transform appraoch. Also, see \cite{NS3, KNS1,KNS2,KNS3, GJLS} for related results. 

At a rough conceptual level, our proof follows the framework as in \cite{Ca81, CLS92}. However,  instead of near a steady state, our construction is around the 3-dim invariant manifold $\CM$, which is qualitatively comparable to \cite{CLY00}, a more general result in finite dimensions.  A rather naive initial attempt may be to construct the local invariant  manifolds near $U_c(\cdot +y)$ for each $y\in \BR^3$ and then patch them together to obtain $\CW^{cu, cs, c}$. While the local construction for each $y$ may follow from the standard procedure combined with some space-time estimates,  it is highly questionable whether such `patch-up' is possible as these local invariant manifolds of each $U_c(\cdot +y)$ are not unique in the first place. Therefore we construct $\CW^{cu, cs, c}$ as the center-unstable, center-stable, and center manifolds of $\CM$ instead of the ones of $U_c(\cdot +y)$. This requires a coordinate system in a neighborhood of the whole $\CM$. A very natural option would be something like 
\[
U = \psi\big( \psi^{-1}(U_c) (\cdot +y) + w^u (\cdot + y) + w^s (\cdot + y) + w^c (\cdot +y)\big)
\] 
where $D \psi(U_c) w^{u, s} \in X^{u, s}$ and $D \psi(U_c) w^c$ in a fixed subspace of $X^c$ transversal to $span\{\p_{x_j} U_c: j=1,2,3\}$. However, this is {\it not} a smooth coordinate system of the phase space as the differentiation in $y$ would cause a loss of regularity in $D_y w^c(\cdot +y)$. This issue also appeared in previous works such as \cite{BLZ08} and \cite{NS1}, in the latter of which it was handled by a rather analytically oriented nonlinear `mobile distance'. Here instead we adopt a more geometric bundle coordinate system, which was also used in \cite{BLZ08}, based on the observation that, while the above parametrization by the spatial translation of $y$ is not smooth with respect to $y$, the vector bundles $\{ w \in X_1 : w(\cdot -y) \in X^{u,s,c}\}$ are smooth in $y$. In such a framework based on vector bundle coordinates, some second fundamental form type quantities are to be carefully treated in the rather technical but intuitive analysis. See Remark \ref{R:coord}. Our estimates are based on the exponential trichotomy and the energy conservation and involve minimal amount of dispersive estimates in Section \ref{S:linear}. In particular, no spectral assumptions such as the nonexistence of embedded eigenvalues or resonance is needed. Based on this coordinate system, we decompose equation \eqref{GP} and, as in the standard procedure of the construction of local invariant manifolds, we cut off the nonlinear terms (except the ones corresponding to the second fundamental form of the center bundle) outside a small neighborhood of $\CM$. At this stage, with the estimates of non-homogeneous linear equations in Section \ref{S:linear}, one may apply the usual Lyapunov-Perron integral equation method or the graph transform of Hadamard to obtain invariant manifolds for the modified system which coincides with the original one. While we obtained the invariant manifolds by conceptually following the procedure in \cite{Ca81, CLS92} which is more in line with the Lyapunov-Perron method, one could also choose to estimate the time-$T$ map of the modified equation and then apply the graph transform method as in \cite{CLY00, BLZ08}. This framework is rather general and it can be  adapted with minimal modifications to yield local invariant manifolds of unstable relative equilibria involving finite dimensional symmetry groups, including ground states and excited states, of a large class of dispersive PDEs such as the NLS, nonlinear Klein-Gordon equations, etc. (See  Section \ref{S:InMa} and \ref{S:smooth}.) Generally, the main necessary assumption would just be that the Hessian of the modified energy functional at the relative equilibria has only finitely many negative directions, so that the linear analysis in \cite{LZ17} is applicable. In fact, in a forthcoming paper, we construct local invariant manifolds of traveling waves of supercritical gKdV equation and analyze the nearby dynamics. 

 The paper is organized as the follows. In Section \ref{S:coord} we set up the basic framework for the construction of local invariant manifolds of $\CM$. Section \ref{S:linear} is on the estimate of non-homogeneous linear equations. The existence of Lipschitz local invariant manifolds and some of their  properties related to the local dynamics are obtained in Section \ref{S:InMa}, while the smoothness in Section \ref{S:smooth}. The non-degenerate case under assumption {\bf (H1-2)} is analyzed in Section \ref{S:non-deg}. Finally, some tedious technical details are left in the Appendix.
 
\noindent {\bf Notations.} Throughout the paper, we follow the following notations: 
\begin{itemize} 
\item $\dot H^s$: the homogeneous Sobolev space $\{u \mid |D|^s u \in L^2\}$. 
\item $X_0$: the energy space defined in \eqref{E:energy-s}.
\item $X_1 = H^1 \times \dot H^1$. 
\item $\langle \cdot, \cdot \rangle$: Euclidean or $L^2$ duality pair unless specified otherwise. 
\item The generic constant $C$ serving as an upper bound always independent of $y \in \BR^3$.  
\item Differentiations are usually not with respect to $c$, unless specified. 
\end{itemize}

\section{A coordinate system near traveling waves} \label{S:coord}

In this section, we rewrite equation \eqref{GP} in an appropriate local coordinate system near traveling waves.

\subsection{Structure of $X_{0}$ and a generalization of the momentum} \label{SS:X_0}

Any $u\in X_{0}$ can be written as $u=\alpha(1+v)$ where $\alpha\in S^{1}$ and $v\in \dot{H}^{1}(\mathbb{R}^{3})$ satisfying $$\vert 1+v\vert ^{2}-1=2\text{Re}(v)+\vert v\vert ^{2}\in L^{2}(\mathbb{R}^{3}).$$ The distance on the energy space is introduced as following. Given $u=\alpha(1+v)$ and $\tilde{u}=\tilde{\alpha}(1+\tilde{v})$ in $X_{0}$, we define the distance $d$ by 
\begin{equation}
d(u,\tilde{u})=\vert \alpha-\tilde{\alpha}\vert + \Vert\nabla v- \nabla\tilde{v}\Vert_{L^{2}(\mathbb{R}^{3})}+ \big\Vert\vert 1+v\vert ^{2}-\vert 1+ \tilde{v}\vert ^{2}\big\Vert_{L^{2}(\mathbb{R}^{3})}.
\end{equation}

Select $\chi(\xi)\in C^{\infty}_{0}(\mathbb{R})$ such that $\chi(\xi)=1$ near $\xi=0$ and define the Fourier multiplier $\chi(D)$ as 
$$\widehat{\chi(D)u}(\xi)=\chi(|\xi|)\hat{u}(\xi).$$

\begin{lemma}\label{MSEP}(\cite{G2}), 
The mapping
\begin{equation} \label{E:psi}
\begin{split}
\psi_{\chi}:& S^{1}\times(H^{1}(\mathbb{R}^{3})+i\dot{H}^{1}(\mathbb{R}^{3}))\rightarrow X_{0}\\
&(\alpha,w)\mapsto \alpha\left(1+w-\chi(D)\left(\frac{(Im(w))^{2}}{2}\right)\right)
\end{split}
\end{equation}
is a homeomorphism.
\end{lemma}

\begin{remark}
Note that $X_{0}$ is not a linear space, but this homeomorphism between $H^{1}\times \dot{H}^{1}$ and $X_{0}$ allows us to work in the linear space $H^{1}\times \dot{H}^{1}$. Since $\dot H^1 (\mathbb{R}^3) \subset L^6(\mathbb{R}^3)$, the above $\alpha \in S^1$ is invariant for any solution of \eqref{GP}, which can be fixed to be $1$ due to the phase invariance of \eqref{GP}. Also, this structure of $X_{0}$ does not depend on the choice of the cut-off function $\chi$. To simplify the notation, we  will fix $\alpha =1$ and wrote $\psi (w)$ for $\psi_{\chi} (1, w)$. Apparently, $\psi^{-1}$ is given by 
\begin{equation} \label{E:psi-inv}
\psi^{-1}(u +iv)= \big(u-1+ \frac 12 \chi(D) (v^2) , v\big). 
\end{equation}
\end{remark}

The coordinate mapping $\psi$ commutes with the translation and $SO(3)$ action. Namely, let $y \in \mathbb{R}^3$ and $Q_{3\times 3}$ be an orthogonal matrix with $\det Q=1$, then 
\begin{equation} \label{E:psi-sym}
\psi \big(\alpha, w(Q\cdot) \big) = \psi(\alpha, w) (Q\cdot), \quad \psi \big(\alpha, w(\cdot - y) \big) = \psi(\alpha, w) (\cdot- y).
\end{equation}
This is useful as \eqref{GP} is invariant under the translation and $SO(3)$ action.  

Let $X_{1}=H^{1}\times \dot{H}^{1}$. By Lemma \ref{MSEP}, for any $u \in X_{0}$, there exits a unique $w=w_1 +i w_2 \in X_{1}$ such that $u=\psi(w)$. As in \cite{LWZ}, extend momentum as
\begin{equation}
\widetilde{P}(w)= -\int_{\mathbb{R}^{3}}\Big[w_{1}+(1-\chi(D))\frac{w_{2}^{2}}{2}\Big]\nabla w_{2}dx, \quad \widetilde P \in C^\infty(X_1, \BR). 
\end{equation}
One can see that $\widetilde{P}(w)=P(u)$ when $u=\psi(w)\in 1+H^{1}(\mathbb{R}^{3})$.

\subsection{A local form of the GP equation near a traveling wave manifold}

Consider a smooth and bounded traveling wave solution $U_c = u_c + i v_c$ of \eqref{GP} with the traveling velocity $c\in \BR^3$ satisfying $|c|\in (0, \sqrt{2})$, we first rewrite the equation in the traveling frame in a neighborhood of $U_c$. 
Assume 
\begin{equation} \label{E:TW-decay}
\lim_{|x| \to \infty} \left( |x|^2 \big( |\text{Re} \, U_c (x)-1| + |\nabla \text{Re}\, U_c|\big) +|x|| \text{Im}\, U_c (x)|\right) =0. 
\end{equation}
Such traveling waves exist as proved \cite{M, Gr04}. In terms of the coordinate mapping $\psi$ given in \eqref{E:psi}, let   
\[
w_{c}:= \psi^{-1} (U_c) = (w_{1c},w_{2c}) \in X_{1}.
\] 
The traveling wave manifold $\{ U_c(\cdot + y)\mid y\in \BR^3\}$ with wave velocity $c$ generated by $U_c$ is invariant under \eqref{GP}. To study the nearby dynamics, we rewrite solutions in the traveling frame $u(t,x)=U(t,x-ct)$ and then $U(t,x)$ satisfies
\begin{equation}\label{tfeq}
i\partial_{t}U-ic\cdot\nabla U+\Delta U+(1-\vert U\vert ^{2})U=0
\end{equation}
or in the abstract form 
\begin{equation} \label{E:tfeq} 
\p_t U = J (E+ c\cdot P)' (U), \quad J=\begin{bmatrix}0&1\\-1&0\end{bmatrix}
\end{equation}
where we recall $E$ and $P$ are the energy and momentum defined in \eqref{E:energy} and \eqref{E:momentum}, respectively, and $J$ is the matrix representation of $-i$. The traveling wave $U_c$ generate a manifold of equilibria of \eqref{E:tfeq}:
\begin{equation} \label{E:TW-mani}
\mathcal{M}=\{U_{c}(\cdot+y): y\in \mathbb{R}^{3}\}.
\end{equation}
Our main goal is to construct local invariant manifolds of $\mathcal{M}$. For any $y \in \BR^3$, let 
\begin{equation}\label{KY}
K_{c,y}\begin{pmatrix}w_{1}\\w_{2}\end{pmatrix}= D\psi\big( w_c (\cdot + y) \big) \begin{pmatrix}w_{1}\\w_{2}\end{pmatrix}= \begin{pmatrix}w_{1}-\chi(D)\left(v_{c}(\cdot+y)w_{2}\right)\\w_{2}\end{pmatrix}.
\end{equation}
and 
\begin{equation}\label{lcy}
\begin{split}
L_{c,y}= & (E + c\cdot P)''( U_c (\cdot + y)\big) \\
=& \begin{bmatrix}-\Delta - 1+(3u^{2}_{c}+v_{c}^{2}) (\cdot+y)&
-c\cdot\nabla+2(u_{c}v_{c}) (\cdot+y)\\c\cdot \nabla+(2u_{c}v_{c})(\cdot+y)& -\Delta -1+(u_{c}^{2}+3v_{c}^2)(\cdot+y)\end{bmatrix}. 
\end{split} \end{equation}
Both $K_{c,y}$ and $L_{c,y}$ are conjugate to $K_{c,0}$ and $L_{c, 0}$ through translation
\begin{equation} \label{E:trans-inv-O}
K_{c,y} w =\big( K_{c,0} w(\cdot -y)\big) (\cdot +y), \quad L_{c,y} U =\big( L_{c,0} U(\cdot -y)\big) (\cdot +y). 
\end{equation}
To simply the notation, we denote
\[
K_y = K_{0,y}, \quad L_{c} = L_{c, 0}.
\]
From \eqref{E:TW-decay} and the Hardy's inequality, we have that $K_{c, y}$ is an isomorphism on $X_1$ with 
\begin{equation} \label{E:K-inv}
K_{c, y}^{-1} w=  \begin{pmatrix}w_{1}+\chi(D)\left(v_{c}(\cdot+y)w_{2}\right)\\w_{2}\end{pmatrix},
\end{equation}
and $L_{c, y}$ induces a real valued symmetric bounded bilinear form on $X_1$, namely, 
\[
K_{c, y}, \, K_{c, y}^{-1} \in L(X_1), \quad L_{c, y} \in L(X_1, X_1^*), \quad L_{c, y}^*= L_{c, y}. 
\]
Moreover, using the translation invariance \eqref{E:trans-inv-O} and the Hardy's inequality, there exists $C>0$ such that 
\begin{equation} \label{E:KY}
\Vert K_{c, y} \Vert_{L(X_1)} + \Vert K_{c, y}^{-1} \Vert_{L(X_1)} \le C, \quad \forall y \in \BR^3.
\end{equation}
Consequently, $J$ is viewed as a closed operator 
\[
J: X_1^* \supset D(J)  \to X_1 \; \text{ satisfying } J^* =-J
\] 
where 
\[\begin{split}
D(J)=& \{ w= (w_1, w_2) \mid  w_1 \in H^{-1} \cap \dot H^1 \text{ and } w_2 \in \dot H^{-1} \cap H^1\}\\
=& \{ w= (w_1, w_2) \mid  w_1 \in H^1 \text{ and } |\xi| \hat w_2, \, |\xi|^{-1}\hat w_2 \in L^2\} . 
\end{split}\]

Suppose for some $y(t) \in \BR^3$ and $w(t) = \big(w_1(t), w_2(t)\big) \in X_1$ smooth in $t$, 
\begin{equation}\label{UnearM} \begin{split}
U(t) = &\psi\left(w_{c}\left(\cdot +y(t)\right)+w(t)\right) \\
= & U_{c}\left(\cdot +y\right)+ K_{c,y} w - \left( \frac 12 \chi(D) (w_2^2), 0 \right)^T
\end{split}\end{equation} 
is a solution to \eqref{tfeq}. 

Here \eqref{E:psi-sym} and the definition of $K_{c,y}$ are used.  Substituting (\ref{UnearM}) into (\ref{E:tfeq}) and using the definition of $L_{c,y}$ and that $U_c$ is an equilibrium of \eqref{E:tfeq}, we obtain
\begin{align*}
&\p_t y \cdot \nabla U_c(\cdot+y) +\partial_{t}(K_{c,y}w)  - \left(\chi(D) (w_2 \p_t w_2), 0 \right)^T \\
=&JL_{c,y}K_{c,y}w +J \big( (E+c \cdot P)' (U)  -L_{c, y} K_{c,y}w  \big).
\end{align*}
The above equation of $w= (w_1, w_2)^T$ can be written as a system of 2 equation of $\p_t w_1$ and $\p_t w_2$, which can be solved easily due to the upper triangular structure of $K_{c,y}$. In particular we have the equation for $w_2$,
\begin{equation} \label{E:w2} 
\p_t w_2 =  - \big(L_{c,y}K_{c,y}w\big)_1 - \p_t y \cdot \nabla v_c(\cdot+y)  +G_2(c,y,w) 
\end{equation}
where $G_2$ is given in the below. The above system of evolution equations for $w$ can be written in a compact form 
\begin{equation}\label{w1w2}
\p_t y \cdot \nabla U_c(\cdot+y) +\partial_{t}(K_{c,y}w)= JL_{c,y}K_{c,y}w+G(c,y, \p_t y,w),
\end{equation}
where $G= \big(G_1 (c, y, \p_t y, w), G_2(c, y, w)\big)^T$ are 
\begin{align*}
G = &\left( |U|^2 - |U_c(\cdot +y)|^2 - 2 U_c (\cdot +y) \cdot (K_{c, y} w) \right) JU_c (\cdot+ y) \\
& + \left( |U|^2 - |U_c(\cdot +y)|^2 \right) JK_{c,y} w + \frac 12 \begin{pmatrix} 2\chi(D)(w_2 \p_t w_2 -w_{2}\nabla w_{2}\cdot c) \\ -(1-|U|^2) \chi(D) (w_2^2) - \Delta \chi(D) (w_2^2)   \end{pmatrix}    
\end{align*}
and $\p_t w_2$ in $G_1$ should be substituted by \eqref{E:w2}. This results in the dependence of $G_1$ on $\p_t y$. The nonlinearity $G$ is affine in $\p_t y$ and contains terms of $w$ of algebraic degree between 2 and 6. Like $K_{c, y}$ and $L_{c,y}$, $G(c, y, \tilde y, w)$ is translation invariant in the sense of \eqref{E:trans-inv-O}, namely,  
\begin{equation} \label{E:trans-inv-G-0}
G(c, y+x, \tilde y, w) = G\big(c, y, \tilde y, w(\cdot -x)\big) (\cdot + x). 
\end{equation}
A more detailed form of and some basic estimates on $G$ are straight forward but tedious and we leave them in the Appendix.

\subsection{Decomposition of $X_1$ and local coordinates near traveling waves}

The free choices of $y\in \BR^3$ and $w \in X_1$ are clearly redundant in the representation of the above $U$. We shall impose appropriate restrictions on $w$ by analyzing the linearization of \eqref{GP} near $U_c$. We will focus on unstable traveling waves. Namely, we assume 
\begin{enumerate} 
\item [({\bf H})] The spectrum $\sigma (JL_c) \nsubseteq i\BR$.  
\end{enumerate}
Since $|c| \in (0, \sqrt{2})$, \eqref{E:TW-decay} and the explicit form \eqref{lcy} of $L_c$ imply that $L_c: X_1 \to X_1^*$ is a compact perturbation to 
\[
L_{c, \infty}= \begin{bmatrix}-\Delta +2& -c\cdot\nabla\\c\cdot \nabla& -\Delta \end{bmatrix}: X_1 \to X_1^*
\]
which is an isomorphism as it induces a uniformly positive quadratic form on $X_1$. This follows from a proof similar to the one in \cite{LWZ} and we skip the details. Therefore $\dim \ker L_c < \infty$ and it is uniformly positive on some finite co-dimensional subspace of $X_1$. Let $n^-(L_c)$ be the Morse index of $L_c$, namely, 
\begin{equation} \label{E:n^-} 
n^-(L_c)= \max \{\dim Y \mid L_c \text{ is negative on the subspace } Y \subset X_1\}. 
\end{equation} 
According to the index formula and the structural decomposition of linear Hamiltonian systems \cite{LZ17}, it holds that $n^-(L_c)>0$ for any unstable traveling wave. We first cite Theorem 2.1 in \cite{LZ17} whose hypotheses are easily satisfied due to $\dim \ker L_c <\infty$ and Remark 2.2 in \cite{LZ17}. \\

\noindent {\bf Theorem 2.1 in \cite{LZ17}.} 
There exist closed subspaces $Y_{j}$, $j=1, \ldots, 6$, and $Y_{0}= \ker L_c$ such that
\begin{enumerate}
\item $X_1 = \oplus_{j=0}^{6} Y_{j}$, $Y_{j} \subset \cap_{k=1}^\infty D\big((JL_c)^k\big)$, $j\ne3$, and
\[
\dim Y_{1} = \dim Y_{4}, \; \dim Y_{5} = \dim Y_{6}, \; \dim Y_{1} + \dim Y_{2} + \dim Y_{5} = n^{-}(L_c);
\]
\item $JL_c$ and $L_c$ take the following forms in this decomposition
\begin{equation}
JL_c\longleftrightarrow%
\begin{pmatrix}
0 & A_{01} & A_{02} & A_{03} & A_{04} & 0 & 0\\
0 & A_{1} & A_{12} & A_{13} & A_{14} & 0 & 0\\
0 & 0 & A_{2} & 0 & A_{24} & 0 & 0\\
0 & 0 & 0 & A_{3} & A_{34} & 0 & 0\\
0 & 0 & 0 & 0 & A_{4} & 0 & 0\\
0 & 0 & 0 & 0 & 0 & A_{5} & 0\\
0 & 0 & 0 & 0 & 0 & 0 & A_{6}%
\end{pmatrix}
, \label{block-JL}%
\end{equation}%
\begin{equation}
L_c\longleftrightarrow%
\begin{pmatrix}
0 & 0 & 0 & 0 & 0 & 0 & 0\\
0 & 0 & 0 & 0 & B_{14} & 0 & 0\\
0 & 0 & L_{Y_{2}} & 0 & 0 & 0 & 0\\
0 & 0 & 0 & L_{Y_{3}} & 0 & 0 & 0\\
0 & B_{14}^{\ast} & 0 & 0 & 0 & 0 & 0\\
0 & 0 & 0 & 0 & 0 & 0 & B_{56}\\
0 & 0 & 0 & 0 & 0 & B_{56}^{\ast} & 0
\end{pmatrix}
. \label{block-L}%
\end{equation}
\item $B_{14}: Y_{4} \to Y_{1}^{*}$ and $B_{56}: Y_{6} \to Y_{5}^{*}$ are
isomorphisms and there exists $\epi>0$ satisfying $\mp\langle L_{Y_{2,3}} u,
u\rangle\ge\epi\Vert u\Vert^{2}$, for all $u \in Y_{2,3}$;
\item all blocks of $JL_c$ are bounded operators except $A_{3}$, where $A_{03}$
and $A_{13}$ are understood as their natural extensions defined on $Y_{3}$;
\item $A_{2,3}$ are anti-self-adjoint with respect to the equivalent inner product $\mp\langle L_{Y_{2,3}} \cdot, \cdot\rangle$ on $Y_{2,3}$;
\item the spectra $\sigma(A_{j})\subset i\mathbf{R}$, $j=1,2,3,4$, $\pm\operatorname{Re}\lambda>0$ for all $\lambda\in\sigma(A_{5,6})$, and $\sigma(A_{5})=-\sigma(A_{6})$; and
\item $n^{-}(L|_{Y_{5}\oplus Y_{6}})=\dim Y_{5}$ and $n^{-}(L|_{Y_{1}\oplus Y_{4}})=\dim Y_{1}.$
\item $(u, v)_{X_1} =0$ for all $u \in Y_1 \oplus Y_2 \oplus Y_3\oplus Y_4 $ and $v \in \ker L_c$. 
\end{enumerate}

We modify this decomposition of $X_1$ slightly for this paper. Let 
\[
X_c^T = span\{ \p_{x_j} U_c \mid j=1,2,3\}, \quad \tilde Y_0 = \{ w \in \ker L_c \mid (w, \tilde w) =0, \, \forall \tilde w \in X_c^T\}
\]
and 
\[
 X_c^{d1} =  \tilde Y_0 \oplus Y_1 \oplus Y_2, \quad  X_c^{e, d2, +, -} = Y_{3, 4, 5, 6}.
\]
For any $y \in \BR^3$ and $\alpha \in \{T, d1, e, d2, +, -\}$, define 
\[
 X_{c, y}^\alpha = \{ w\in X_1 \mid w(\cdot -y) \in  X_c^\alpha\}. 
\]
Recall the traveling wave manifold $\mathcal{M}$ defined in \eqref{E:TW-mani}. Clearly 
\begin{equation} \label{E:decom1} 
X_1 =  X_{c,y}^T \oplus  X_{c,y}^{d1} \oplus  X_{c, y}^e \oplus  X_{c,y}^{d2} \oplus   X_{c,y}^+ \oplus  X_{c,y}^-, \quad  X_{c, y}^T = T_{U_c(\cdot +y)} \mathcal{M},
\end{equation}
with associated projection $ \Pi_{c,y}^{T, d1, e, d2, +, -}$. Let 
\[
\quad 0 < \lambda < \min \{\text{Re} \mu \mid \mu \in \sigma(A_5)\}.   
\]

\begin{lemma} \label{decomp2}
Assume \eqref{E:TW-decay} and ({\bf H}), then there exists $C>0$, such that, for any $y \in \BR^3$, 
\begin{enumerate} 
\item $ X_{c, y}^{+, -, T, d1, d2} \subset \dot H^k \cap X_1$ for any $k\ge 1$;
\item $0< \dim  X_{c, y}^{\pm} = d \le n^-(L_c)$, $d_1= \dim  X_{c, y}^{d1} = n^-(L_c) + \dim \ker L_c  -3-d$, and $d_2= \dim  X_{c, y}^{d2} \le n^- (L_c)-d \le d_1$;
\item there exist bases $ V_{c, j}^\pm$, $j=1, \ldots, d$ of $ X_{c, 0}^\pm$, $ V_{c, j}^{d1}$, $j=1, \ldots, d_1$, of $ X_{c, 0}^{d1}$, $ V_{c, j}^{d2}$, $j=1, \ldots, d_2$ of $ X_{c, 0}^{d2}$, and $ V_{c, j}^{T} = \p_{x_j} U_c$, $j=1,2,3$, of $ X_{c, 0}^T$, along with $ \zeta_{c,j}^\pm$, $j=1, \ldots, d$, $ \zeta_{c,j}^{d1}$, $j=1, \ldots, d_1$, $ \zeta_{c,j}^{d2}$, $j=1, \ldots, d_2$, and $ \zeta_{c, j}^T$, $j=1,2,3$, belonging to $D\big((JL_{c, y})^*\big) \cap H^k= H^k \times (\dot H^{-1} \cap H^k)$ for any $k\ge 1$, such that 
\[
 \Pi_{c, y}^\alpha w = \sum_{j=1}^{\dim  X_{c, 0}^\alpha} \langle  \zeta_{c, j}^\alpha(\cdot +y), w  \rangle  V_{c, j}^\alpha (\cdot +y), \quad \alpha\in \{T, d1, d2, +, -\},  
\]
consequently, projections $ \Pi_{c,y}^{T, d1, e, d2, +, -}$ are smooth in $y$ with derivatives bounded uniformly in $y\in \BR^3$; 

\item In the decomposition $X_1 =\oplus_{\alpha \in \{T, d1, e, d2, +, -\}}  X_{c, y}^\alpha$, $JL_{c,y}$ and the quadratic form $L_{c, y}$ take the form 
\[
L_{c, y} \longleftrightarrow \begin{bmatrix} 0 & 0 & 0 &0 & 0 & 0 \\ 0 & L_{11} (y) & 0 & L_{12} (y) & 0 &0  \\ 0 & 0 & L^e (y) & 0& 0 & 0 \\  0 & L_{12}(y)^* & 0 & 0 & 0 & 0 \\ 0 & 0 & 0 & 0 & 0 & L_{+-}(y) \\ 0 & 0 & 0 & 0 & L_{+-}(y)^* & 0\end{bmatrix}, 
\]
\[
JL_{c, y} \longleftrightarrow \begin{bmatrix} 0 &  A_{T1} (y) &  A_{Te}(y) &  A_{T2} (y)  &0 &0\\ 0 &  A_{1} (y) &  A_{1e} (y) &  A_{12} (y) & 0 &0 \\ 0 & 0 &  A_{e} (y) & A_{e2}(y) & 0 & 0\\ 0 & 0 & 0 &  A_2 (y) & 0 &0\\ 0 & 0 &0 &0 &  A_+ (y)& 0 \\ 0 & 0 &0 &0 & 0 &  A_-(y) \end{bmatrix}
\]
where 
\begin{enumerate} 
\item all above blocks are translation invariant in the sense of \eqref{E:trans-inv-O};
\item all above blocks are bounded except $A_e(y)$;
\item $\Vert e^{t A_1(y)} \Vert + \Vert e^{t A_2(y)} \Vert \le C (1 + |t|)^{d_1}$;

\item the quadratic form $\langle L^e(y)  V^e,  V^e \rangle \ge \frac 1C \Vert  V^e \Vert^2_{X_1}$ for any $ V^e \in  X_{c, y}^e$ and $ A_e(y)$ is anti-self-adjoint with respect to $L^e(y)= L_{c, y}|_{ X_{c, y}^e}$. 
\item $\sigma\big( A_+(y) \big)= -\sigma\big(( A_-(y)\big)$, $\Vert e^{t A_\pm(y)}|_{ X_{c, y}^\pm} \Vert \le C e^{\lambda t}$, for all $\mp t\ge 0$;
\end{enumerate} \end{enumerate}
\end{lemma}

\begin{proof} 
All the conclusions directly follow from Theorem 2.1 in \cite{LZ17} except those on the dual basis $\zeta_{c, j}^\alpha$ and the smoothness of $ \Pi_{c, y}^\alpha$ in $y$. In particular $ X_{c, y}^{+, -, T, d1, d2} \subset \dot H^k \cap X_1$ is due to $D\big((JL_c)^k\big) =\dot H^{1+2k} \cap X_1$. To complete the proof, we only need to show the smoothness of $ \Pi_{c, y}^\alpha$ in $y$. Due the translation invariance, we have  
\begin{equation} \label{E:trans-inv}
 \Pi_{c, y+ y'}^\alpha w =\left(  \Pi_{c, y'}^\alpha \left( w(\cdot -y)\right) \right)(\cdot + y), \quad \alpha \in  \{T, d1, e, d2, +, -\}
\end{equation}
for any $y \in \BR^3$, which also implies 
\[
D_y^n  \Pi_{c, y} w=\left( D_y^n  \Pi_{c, 0} \left( w(\cdot -y)\right) \right)(\cdot + y).
\]
Let 
\[
w^1, \ldots, w^{d_0}, w^{d_0+1}, \ldots, w^{d'} \; \text{ be a basis of }\;  X_c^T \oplus  X_c^{d1} \oplus  X_c^{d2} \oplus   X_c^+ \oplus  X_c^-
\]
formed by bases of $ X_c^{T, d1, d2, +, -}$ such that $w^j = \p_{x_j} U_c$, $j=1,2,3$ and $w^1, \ldots, w^{d_0}$ is a basis of $\ker L_c$ where $d_0  = \dim \ker L_c$. 

Let 
\[
\wt w^j = \begin{cases} \begin{bmatrix} 1-\Delta & 0 \\ 0& -\Delta \end{bmatrix} w^j, \quad & j=1, \ldots, d_0; \\ 
L_c w^j, & j=d_0+1, \ldots, d'.  
\end{cases}\]
Clearly $\wt w^j \in X_1^* \cap H^k= H^k \times (\dot H^{-1} \cap H^k)$ for any $k$ and $j=1, \ldots, d'$ and $\{\wt w^1, \ldots, \wt w^{d_0}\}$ and $\{\wt w^{d_0+1}, \ldots, \wt w^{d'}\}$ are both linearly independent. Moreover $\{\wt w^1, \ldots, \wt w^{d'}\}$ are also linearly independent. In fact, assume  
\[
\wt w= a_1 \wt w^1 + \ldots, a_{d_0} \wt w^{d_0} = a_{d_0+1} \wt w^{d_0+1} + \ldots, a_{d'} \wt w^{d'}
\]
for some $a_1, \ldots, a_{d'}$. Since $w^j \in \ker L_c$ for $j=1, \ldots, d_0$, we have  
\[
0=  \langle L_c \sum_{j=1}^{d_0} a_j w^j ,  \sum_{j=d_0+1}^{d'} a_j w^j \rangle = \langle \wt w, \sum_{j=1}^{d_0} a_j w^j \rangle = (\wt w, \wt w)_{X_1^*}. 
\]
Therefore $\wt w=0$ and we obtain the linear independence of  $\{\wt w^1, \ldots, \wt w^{d'}\}$. 

For any $w \in  X_c^e$, on the one hand, from the above $L_c$-orthogonality between $ X_c^e$ and $ X_c^T \oplus  X_c^{d1} \oplus  X_c^{d2} \oplus   X_c^+ \oplus  X_c^-$, we have $\langle \wt w^j, w\rangle = \langle L_c w_j, w \rangle =0$, for $j=d_0+1, \ldots, d'$. On the other hand, due to the orthogonality between $ X_c^e$ and $\ker L_c$ with respect to the $(\cdot, \cdot)_{X_1}$, for any $j=1, \ldots, d_0$, we have $\langle \wt w^j, w\rangle = (w^j, w)_{X_1} =0$. Counting the dimensions, we obtain that 
\[
\wt w^1, \ldots, \wt w^{d'} \; \text{ form a basis of } \; \ker i_{ X_c^e}^* = \{ f \in X_1^* \mid \langle f, w \rangle =0, \; \forall w \in  X_{c}^e\}.
\]  

As $\ker i_{ X_c^e}^*$ is isomorphic to $( X_c^T \oplus  X_c^{d1} \oplus  X_c^{d2} \oplus   X_c^+ \oplus  X_c^-)^*$, let $\gamma_1, \ldots, \gamma_{d'} \in \ker i_{ X_c^e}^*$ be the dual basis of $w^1, \ldots, w^{d'}$. Since $\gamma_j$ can be written as a linear combinations of $\wt w^1, \ldots, \wt w^{d'}$, we have $\gamma_j \in  H^k \times (\dot H^{-1} \cap H^k)$ for any $k$ and $j=1, \ldots, d'$. From \eqref{E:trans-inv} and the definition of $\gamma_j$, it is easy to verify that, for any $\alpha \in  \{T, d1, d2, +, -\}$, $w\in X_1$, and $y\in \BR^3$, 
\[
 \Pi_{c,y}^\alpha w = \sum_{w^j \in  X_c^\alpha} \langle \gamma_j (\cdot +y), w\rangle w^j(\cdot +y). 
\]
The smoothness of $ \Pi_{c, y}^\alpha$ in $y$ follows from the regularity $\gamma_j \in H^k \times (\dot H^{-1} \cap H^k)$ and $w^j \in \dot H^k \cap X_1$ for any $k$ and $j=1, \ldots, d'$, which also implies the smoothness of $ \Pi_{c, y}^e = I -\sum_{\alpha = T, d1, d2, +, -}  \Pi_{c, y}^\alpha$. Divide $\{w^j, \, j=1, \ldots, d'\}$ and $\{\gamma_j, \; j=1, \ldots, d'\}$ according to $\alpha\in \{T, d1, d2, +, -\}$, we obtain $ V_{c, j}^\alpha$ and $ \zeta_{c, j}^\alpha$ and complete the proof of the lemma.   
\end{proof}

\begin{remark}
Under the following additional non-degeneracy assumptions 
\begin{equation} \label{E:non-deg-1}
ker(L_{c}) = span\{\nabla U_{c}\}, \quad n^-(L_c) = d= \dim  X^+, 
\end{equation}
we have $ X_{c,y}^{d1, d2}=\{0\}$ and the decomposition may be simplified. We shall discuss this case carefully in Section \ref{S:non-deg}. 
\end{remark}

With respect to the bases $\{ V_{c, j}^\alpha\}$, $\alpha \in \{d1, d2, +, -\}$, operators 
\[
 A_{T1}(y), \;  A_1(y), \;  A_{T2} (y),  \;  A_{12} (y),  \;  A_{2} (y),  \;  A_{+} (y),  \;  A_{-} (y)
\]
representing $ \Pi_{c, y}^\beta JL_{c,y}|_{ X_{c,y}^\alpha} :  X_{c, y}^\alpha \to  X_{c,y}^\beta$ in the above block decomposition of $JL_{c, y}$ can be represented by matrices 
\[
M_{T1}, \;  M_1, \; M_{T2} ,  \; M_{12} ,  \; M_{2} ,  \; M_{+} ,  \; M_{-} 
\]
which are independent of $y$ due to the translation invariance\eqref{E:trans-inv-O} of $K_{c, y}$ and $L_{c,y}$. Namely, 
\begin{equation} \label{E:M_ab}
 A_{T1}(y)  \big( ( V_{c, 1}^\alpha, \ldots,  V_{c, d_1}^\alpha) a\big) = ( V_{c, 1}^\beta, \ldots,  V_{c, 3}^\beta)  (M_{T1} a), \quad \forall a\in \BR^{d_1}, \ldots
\end{equation}
From Lemma \ref{decomp2},  
\begin{equation} \label{E:exp-M}
\Vert e^{tM_1} \Vert + \Vert e^{tM_2} \Vert \le C(1+ |t|)^{d_1}, \; \forall t \in \BR; \quad \Vert e^{tM _\pm}\Vert \le C e^{\pm \lambda t}, \forall\, \mp t\ge 0.
\end{equation}
A similar representation through translation in $x \to x+y$ of $A_e(y)$ would cause loss of regularity when $\p_y$ is carried out. Instead we will keep working with $A_e(y) = \Pi_{c,y}^e JL_{c,y} \Pi_{c, y}^e$. When viewed as an (unbounded) operator  from $X_1$ to $X_1$, it is a uniformly (in $y$) bounded perturbation to a constant coefficient operator and its derivatives of all orders are bounded operators. In fact, separating the terms in \eqref{lcy} with constant coefficients from those  with spatially decaying variable coefficients implies 
\begin{equation} \label{E:homo-1}
JL_{c, y} = JL_{c, \infty} + \wt Q(y), \quad L_{c, \infty} =    \begin{bmatrix} 2-\Delta & -c\cdot \nabla \\
c\cdot \nabla & -\Delta \end{bmatrix}, 
\end{equation}
and 
\begin{equation} \label{E:tQ}
\wt Q(y) = \begin{bmatrix} 2u_{c}v_{c} & u_{c}^{2}-1+3v_{c}^2 \\  3(1-u^{2}_{c}) -v_{c}^{2} & - 2u_{c}v_{c}\end{bmatrix} (\cdot+y).
\end{equation}

\begin{lemma} \label{L:A_e}
Fix $c \in (0, \sqrt{2})$. For any integer $k\ge 0$, there exists $C_k>0$ such that for any $y \in \BR^3$, it holds 
\[
\Vert D_y^k \big(A_e(y) - JL_{c, \infty}\big) \Vert_{L\big( (\otimes^k (\BR^3)) \otimes X_1, X_1)} \le C_k.
\]
\end{lemma}

\begin{proof}
Clearly $\wt Q(0) \in L^\infty$ and $\nabla^k \wt Q(0) \in L^\infty \cap L^2$ for any $k \ge 1$, along with \eqref{E:TW-decay} and Hardy's inequality, it is straight forward  to prove, for all $y \in \BR^3$ and $V \in X_1$, 
\begin{equation} \label{E:tQ-1}
\Vert \wt Q V \Vert_{X_1} + \Vert D_y^k \wt Q  V \Vert_{X_1} + \Vert J \wt Q V \Vert_{X_1^*} + \Vert J D^k_y \wt Q V \Vert_{X_1^*} \le C_k \Vert V \Vert_{X_1}, 
\end{equation}
for some $C_k>0$ independent of $y$. Write 
\[
A_e(y) - JL_{c, \infty} = A_e(y)- JL_{c, y} + \wt Q(y).
\]
Therefore, to complete the proof of the lemma, we only need to show the boundedness of $D_y^k \big(JL_{c, y}- A_e(y)\big)$, which, according to Lemma \ref{decomp2}, has the same blockwise decomposition except the $A_e(y)$ component replaced by $0$. The uniform boundedness (in $y$) of $A_e(y)- JL_{c, y}$ follows from the boundedness of those blocks, where the uniformity in $y$ is due to their translation invariance in the sense of \eqref{E:trans-inv-O}. The uniform upper bounds of $D_y^k \big(JL_{c, y}- A_e(y)\big)$ also follow the translation invariance of these blocks and the extra regularity of $\zeta_{c, j}^\alpha$ and $V_{c,j}^\alpha$, $\alpha \in \{ T, d1, d2, +, -\}$. 
\end{proof}

Following from that $X_{c,y}^T$ is the tangent space $T_{U_c (\cdot +y)} \mathcal{M}$ and $K_{c, y} = D\psi\big(w_c(\cdot +y)\big)$ is an isomorphism, $K_{c,y}^{-1} X_{c, y}^T$ is the tangent space of 
\[
\psi^{-1} (\mathcal{M}) = \{ w_c (\cdot + y) \mid y\in \BR^3\}. 
\]
Based on the Implicit Function Theorem, it is straight forward to prove that, for small $\delta$, 
\begin{align*}
\big\{ w_c (\cdot +y) + w^{d1} &+ w^e + w^{d2} + w^+ + w^-  \mid \\
& w^\alpha \in K_{c, y}^{-1} X_{c, y}^\alpha, \; \Vert w^\alpha \Vert <  \delta, \;    \alpha \in \{d1, e, d2, +, -\}\big\}
\end{align*}
is a neighborhood of $\psi^{-1} (\mathcal{M}) \subset X_1$, where each point has a unique representation in the above form.

\subsection{A local bundle coordinate system.} \label{SS:bundle}

Accordingly, we shall set up the bundle coordinates near $\psi^{-1} (\mathcal{M})$ precisely.  
\begin{equation} \label{E:CX}
\CX^e = \{(y, V^e) \mid y\in \BR^3, \; V^e \in X_{c, y}^e \},  
\end{equation} 
and balls on this bundle
\begin{equation} \label{E:CX-delta}
\CX^e (\delta) = \{ (y, V) \in \CX^e  \mid \Vert V \Vert_{X_1} < \delta
\}. 
\end{equation}
Let $y_\# \in \BR^3$ and $B^3(\delta)$ be the open ball on $\BR^3$ centered at $y_\#$ with radius $\delta$. For $\delta \ll1$, a smooth (due to the smoothness of $\Pi_{c, y}^e$ with respect to $y$) local trivialization from $B^3 (\delta) \times X_{c, y_\#}^e$ to $\CX^e$, thus a local coordinate system, of $\CX^e$ is given by $(y, \Pi_{c, y}^e V)$, $V \in X_{c, y_\#}^e$. There is a natural translation on $\CX^{e}$
\[
(z, y, V^e) \longrightarrow \big(y+z, V^e(\cdot, +z) \big). 
\]
Along with other subspaces $X_{c, y}^{T, d1, d2, +, -}$, we will often consider bundles $\BR^k \oplus \CX^e$ over $\BR^3$ with fibers $\BR^k \oplus X_{c, y}^e$, as well as their balls 
\begin{equation} \label{E:Bball} 
B^k (\delta_1) \oplus \CX^e (\delta_2) = \{ (y, a, V^e) \mid  a \in \BR^k, \; |a| < \delta_1, \;  (y, V^e) \in \CX^e(\delta_2) \}. 
\end{equation} 
For any fixed $y_\#$, the smoothness of $\Pi_{c, y}^e$ with respect to $y$ allows it serve to as a local trivialization of the fibers $X_{c, y}^e$ for $y$ near $y_\#$. 

Define an embedding 
\[
Em: \BR^{3+ d_1+d_2+ 2d} \oplus \CX^e \to X_1 
\]
as 
\begin{equation} \label{E:em} \begin{split}
& Em (y, a^T, a^{d1}, a^{d2}, a^+, a^-, V^e) \\
=&   \sum_{j=1}^3 a_j^T \p_{x_j} w_c (\cdot +y)   + K_{c,y}^{-1} \Big(\sum_{j=1}^{d_1} a_j^{d1} V_{c,j}^{d1} (\cdot + y) \\
&+ \sum_{j=1}^{d_2} a_j^{d2} V_{c,j}^{d2} (\cdot + y) + \sum_{j=1}^{d} a_j^+ V_{c,j}^+ (\cdot + y) + \sum_{j=1}^d a_j^- V_{c,j}^- (\cdot + y) + V^e \Big)\\
:= &  K_{c,y}^{-1} \Big(\big(a^T V_{c}^T + a^{d1} V_{c}^{d1} + a^{d2} V_{c}^{d2} + a^+ V_{c}^+ + a^- V_{c}^-\big) (\cdot + y) + V^e \Big).
\end{split}\end{equation} 
The embedding $Em^\perp : \BR^{d_1+d_2+2d} \oplus \CX^e$ defined on the transversal bundle will be used for the bundle coordinates near $\psi^{-1} (\mathcal{M})$
\begin{equation} \label{E:em-perp}
Em^\perp (y, a^{d1}, a^{d2}, a^+, a^-, V^e) = Em (y,0, a^{d1}, a^{d2}, a^+, a^-, V^e). 
\end{equation}
Clearly $Em^\perp$ is translation invariant in the sense
\begin{equation} \label{E:trans-inv-Em} \begin{split}
&Em^\perp \big(y+ \tilde y, a^{d1}, a^{d2}, a^+, a^-, V^e(\cdot +\tilde y)\big) \\
& \qquad \qquad = Em^\perp (y, a^{d1}, a^{d2}, a^+, a^-, V^e)(\cdot +\tilde y), \quad \forall \tilde y \in \BR^3. 
\end{split}\end{equation}

On the one hand, according to the above trivialization, given any Banach space $Z$, a mapping $f:Z\to \CX^e$ is said to be smooth near some $z_0 \in Z$ if $y(z)$ and $V^e(z) \in X_{c, y(z_0)}^e$ are smooth in $z$ near $z_0$, where $f(z) = \big(y(z), \Pi^e_{c, y(z)} V^e(z)\big)$. Due to the smoothness of $\Pi_{c, y}^e$, in fact this is equivalent to the smoothness of $y(z)$ and $V(z) \in X_1$ where $f(z) = \big(y(z), V(z)\big)$.

On the other hand, for any Banach space $Y$, a mapping $g:\CX^e \to Y$ is said to be smooth near some $(y_\#, V_\#)$ if 
\[
\tilde g(y, V)= g(y, \Pi_{c, y}^e V), \quad y\in \BR^3, \; V \in X_{c, y_\#}^e
\]
is smooth in $(y, V) \in \BR^3 \times X_{c, y_\#}$ near $(y_\#, V_\#)$. It is straight forward to verify 
\begin{itemize} 
\item $g$ is smooth if and only if locally $g(y, \Pi_{c, y}^e V)$, $y\in \BR^3$, $V \in X_1$, is smooth on $\BR^3 \times X_1$.\item $g$ is smooth if and only if locally it is the restriction to $\CX^e$ of a smooth mapping defined on $\BR^3 \times X_1$;
\item $g$ is smooth if and only if $g\circ f$ is smooth for any smooth $f: Z \to \CX^e$ defined on any Banach space $Z$; 
\item $Em$ is smooth with respect to $(y, V^e)$,  due to the smoothness of $K_{c,y}^{-1}$ and the basis $V_{c, j}^\alpha$, $\alpha \in \{T, d1, d2, +, -\}$.
\end{itemize}
We shall often work with $g\big(y, Em(y, a, V^e)\big)$ with $g$ smooth on $\BR^3 \times X_1$.

Near the 3-dim manifold  $\mathcal{M}$ of traveling waves, we will work through the mapping $\Phi$ defined on $\BR^{d_1 +d_2+2d} \oplus \CX^e$ which is diffeomorphic on  $\BR^{d_1 +d_2+2d} (\delta) \oplus \CX^e (\delta)$
\begin{equation} \label{E:coord-1} \begin{split}
U =  &\Phi (y, a, V^e) 
=  \psi\big( w_c( \cdot +y) + Em^\perp (y, a, V^e) \big).
\end{split}\end{equation}

This is a smooth vector bundle coordinate system in a neighborhood of $\mathcal{M} \subset X_0$ for sufficiently small $\delta>0$. From \eqref{E:em} and \eqref{E:em-perp}, $\Phi$ can be naturally extended into a smooth mapping on $\BR^{3+d_1+d_2+2d} \oplus X_1$. 

\begin{remark} \label{R:coord} 
As the subspaces $X_{c, y}^{T, d1, e, d2, +, -}$ are obtained as the translations of $X_{c}^{T, d1, e, d2, +, -}$, it is tempting to use the coordinate system 
\[
U = \psi\big ((w_c + w^{d1} + w^{d2} + w^+ +w^- + w^e) (\cdot +y)\big)
\]
where 
$w^\alpha \in X_c^\alpha$ and $y \in \BR^3$. 

However, such translation parametrization is not  smooth in $X_1$ because the differentiation in $y$ causes a loss of one order regularity in $D_y w^e(\cdot +y)$. This is one of the main issues in Nakanishi and Schlag \cite{NS1}, where the authors constructed the center-stable manifolds of the manifold of ground states for the Klein-Gordon equation. They introduced a  nonlinear ``mobile distance'' to overcome that difficulty. Instead, the above bundle coordinate system \eqref{E:coord-1}, where $V^e \in X_{c, y}^e$ is not directly parametrized by a translation in $y$, represents a different framework based on the observation that, while the parametrization by the spatial translation of $y$ is not smooth in $X_1$ with respect to $y$,  the vector bundles $X_{c, y}^{T, d1,e, d2, +, -}$ over $\mathcal{M}$ are smooth in $y$ as given in Lemma \ref{decomp2}. This approach was used also in \cite{BLZ08}. While it avoids the loss of regularity when differentiating in $y$, it will involve more geometric calculation. 
\end{remark}

\subsection{An equivalent form of the GP equation near traveling waves} \label{csntw}

Let $U(t,x)$ be any solution to \eqref{tfeq}. If $U(t,x)$ stays in a small neighborhood of $\mathcal{M}$, then we can express $U$ in the coordinate system \eqref{E:coord-1} 
\begin{equation}\label{coordw}
U(t)= \Phi\big( y(t), a(t), V^e (t) \big), 
\quad ( y, a, V^e) (t)  \in B^{d_1+d_2+2d}(\delta) \oplus  \CX^e (\delta)
\end{equation}
for some $\delta>0$. 
Substituting $U(t)$ and \eqref{coordw} into \eqref{w1w2} and using \eqref{UnearM}, we obtain
\begin{equation}\label{eqdec}
\begin{split}
&\p_t y \cdot \nabla U_{c}(\cdot+y) + \p_t  V^e + \big( (\p_t a^{d1})  V_c^{d1} +   (\p_t a^{d2})  V_c^{d2} + (\p_t a^{+})  V_c^{+} \\
&+(\p_t a^{-})  V_c^{-}\big) (\cdot +y) + \big( a^{d1} \p_t y \cdot (\nabla  V_c^{d1}) +  a^{d2} \p_t y \cdot (\nabla  V_c^{d2}) \\
&+ a^{+} \p_t y \cdot (\nabla  V_c^{+}) +a^{-} \p_t y \cdot (\nabla  V_c^{-}) \big) (\cdot +y)  \\
=& JL_{c,y}K_{c,y} Em^\perp (y, a, V^e)+G\big(c,y, \p_t y, Em^\perp (y, a, V^e)\big).
\end{split}
\end{equation}

Starting with $\p_t y$, we identify the evolution equation of each coordinate component. Applying $ \Pi_{c, y}^T$ and using Lemma \ref{decomp2} and \eqref{E:M_ab}, we have 
\begin{align*} 
&\p_t y  + \langle  \zeta_c^T (\cdot +y), \p_t  V^e\rangle  + \langle  \zeta_c^T, \; \p_t y \cdot \nabla \big( a^{d1}   V_c^{d1} +  a^{d2}  V_c^{d2} + a^{+}   V_c^{+}  +a^{-}   V_c^{-}\big) \rangle   \\
=& M_{T1} a^{d1} + M_{T2} a^{d2}- \langle L_{c, y} J  \zeta_c^T (\cdot +y),  V^e\rangle \\
&+ \langle  \zeta_c^T (\cdot +y), G\big(c,y, \p_t y,Em^\perp (y, a, V^e)\big)\rangle.
\end{align*} 
Since $ V^e \in  X_{c,y}^e$ implies $\langle  \zeta_c^\alpha (\cdot +y),  V^e\rangle=0$  all $t$, we have 
\begin{equation} \label{E:p_tVe-0}
\langle  \zeta_c^\alpha (\cdot +y), \p_t  V^e\rangle =- \langle (\p_t y \cdot \nabla  \zeta_c^\alpha) (\cdot +y),  V^e\rangle, \quad \alpha \in \{T, d1, d2, +, -\}. 
\end{equation}

Therefore $\tilde y= \p_t y$ satisfies the following equation 
\begin{align*}
&\tilde y - \langle (\tilde y \cdot \nabla  \zeta_c^T) (\cdot +y),  \Pi_{c, y}^e K_{c, y} w \rangle  + \langle  \zeta_c^T, \; \tilde y \cdot \nabla \big( (I - \Pi_{c, y}^T -  \Pi_{c, y}^e) K_{c, y} w \big) \\
=& M_{T1} a^{d1} + M_{T2} a^{d2}- \langle L_{c, y} J  \zeta_c^T (\cdot +y),  V^e \rangle + \langle  \zeta_c^T (\cdot +y), G\big(c,y, \tilde y, w\big)\rangle, 
\end{align*}
where $w = Em^\perp (y, a, V^e)$. We actually note that the above equation is well-defined for any small $w\in X_1$. From Lemma \ref{ESG} and the regularity of $ V_c^\alpha$ and $ \zeta_c^\alpha$, when $\Vert w\Vert_{X_1}$ is sufficiently small, one may solve for $\tilde y = \p_t y$ and obtain 
\begin{equation} \label{eqy'} 
\p_t y = M_{T1} a^{d1} + M_{T2} a^{d2}- \langle L_{c, y} J  \zeta_c^T (\cdot +y),  V^e \rangle+ G^T (c, y, w), 
\end{equation} 
where 
\[
w= Em^\perp (y, a, V^e).
\]
According to Lemma \ref{ESG} and the regularity of $ \zeta_c^{T, d1, d2, +, -}$, $G^T(c, y, w)$ is smooth in $y$ and $w \in X_1$ when $\Vert w\Vert_{X_1}\ll1$. As we did not prove $G \in X_1$ in Lemma \ref{ESG}, we used the extra regularity of $ \zeta_c^\alpha \in H^k \times (\dot H^{-1} \cap H^k)$. Furthermore, there exists $C>0$ such that, for any $y \in \BR^3$ and small $w \in X_1$, 
\begin{equation} \label{E:G^T-1}
|D_w^k D_y^l G^T(c, y, w)| \le C_{l, k} \Vert w \Vert_{X_1}^{\max \{2 -k, 0\}}.
\end{equation}

Applying $ \Pi_{c, y}^\alpha$, $\alpha \in \{d1, d2, +, -,e\}$, to \eqref{eqdec} and using the basis $ V_c^\alpha$, Lemma \ref{decomp2}, \eqref{E:M_ab}, \eqref{E:p_tVe-0} and \eqref{eqy'}, we obtain
\begin{equation} \label{E:a^pm}
\p_t a^\pm   = M_\pm a^\pm  + G^\pm (c, y,  w),
\end{equation}
\begin{equation} \label{E:a^d1}
\p_t a^{d1}   = M_{1} a^{d1} + M_{12} a^{d2} - \langle L_{c, y} J  \zeta_c^{d1} (\cdot +y),  V^e \rangle + G^{d1} (c, y,  w), 
\end{equation}
\begin{equation} \label{E:a^d2}
\p_t a^{d2}   = M_{2} a^{d2}  + G^{d2} (c, y,  w), 
\end{equation}
\begin{equation} \label{E:Ve}
\Pi_{c,y}^e \p_t V^{e}   = A_e(y) V^e + a^{d2} A_{e2}(y) V_c^{d2} (\cdot+y) + G^{e} (c, y,  w), \; V^e \in X_{c, y}^e,
\end{equation}
where $A_{e2}(y) = \Pi_{c, y}^e JL_{c, y} \Pi_{c, y}^{d2}$ is smooth in $y$ and 
\begin{equation}  \label{E:em-perp-G}
w= Em^\perp (y, a, V^e). 
\end{equation}
Much as in the derivation of $G^T$, $G^\alpha$ is also well-defined for any small $w\in X_1$, $\alpha \in \{d1, d2, +, -, e\}$. Like $K_{c,y}$ and $L_{c,y}$, $G^\alpha$ is translation invariant, 
\begin{equation} \label{E:trans-inv-G} \begin{split} 
&G^\alpha (c, y+z, w (\cdot + z) \big) = G^\alpha (c, y, w), \quad \alpha \in \{T, d1, d2, +, -\} \\
&G^e (c, y+z, w (\cdot + z) \big) = G^e (c, y, w) (\cdot + z)
\end{split}\end{equation}
for all $z\in \BR^3$. For $\Vert w \Vert_{X_1} \ll1$, $G^\alpha$, $\alpha \in \{T, d1, d2, e, +, -\}$, are quadratic in $w$. From Lemma \ref{ESG} and the regularity of $\zeta_c^{T, d1, d2, +, -}$, they are smooth in $y$ and $w$ and satisfy 
\begin{equation} \label{E:G^pm-0}
|D_w^k D_y^l G^\alpha (c, y, w)| \le C_{l, k} \Vert w \Vert_{X_1}^{\max \{2 -k, 0\}}, \quad \alpha \in \{T, d1, d2, +, -\}.
\end{equation}
The multi-linear terms in $G^e$prevent it from belonging to $X_1$ (see Lemma \ref{ESG}). However, due to the extra regularity of $\zeta_c^\alpha$, projections $\Pi_{c, y}^\alpha$, $\alpha \in \{d1, d2, +, -\}$, act on a larger class of functions than $X_1$, from Lemma \ref{ESG}, we have 
\begin{equation} \label{E:G^pm-1}
(I-\Pi_{c, y}^e) G^e (c, y, w) =0, \quad G^e (c, y,w) \in X_1 +  W^{1, \frac 32} + (L^{\frac 32} \cap \dot W^{1, \frac 65})
\end{equation} 
and

\begin{equation} \label{E:G^pm-2}
|D_w^k D_y^l G^e |_{X_1 +  W^{1, \frac 32} + L^{\frac 32} \cap \dot W^{1, \frac 65}} \le C_{l, k} \Vert w \Vert_{X_1}^{\max \{2 -k, 0\}}, 
\end{equation}

\noindent {\bf Transforming the $V^e$ equation.} Before we end this section, we transform \eqref{E:Ve} to an equivalent form. In fact, since $(I-\Pi_{c, y}^e) V^e \equiv 0$, we have 
\begin{equation} \label{E:cov-deri}
(I-\Pi_{c, y}^e) \p_t V^e =  D_y \Pi_{c, y}^e (\p_t y) V^e.
\end{equation}
Therefore, \eqref{E:Ve} implies 
\begin{equation} \label{E:VE} \begin{split} 
\p_t V^e =& 
A_e(y) V^e + \mathcal{F} (c, y) (\p_t y, V^e) \\
& + a^{d2} A_{e2}(y) V_c^{d2} (\cdot+y) + G^{e} (c, y,  w)
\end{split} \end{equation}
where 
\[
A_e(y) =  \Pi_{c, y}^e J L_{c, y} \Pi_{c, y}^e
\]
was given in Lemma \ref{decomp2} and the bounded bilinear operator $\mathcal{F}(c,y): \mathbb{R}^{3}\otimes X_{1}\rightarrow X_1 
$ 
 is given by
\begin{equation} \label{E:CF-def}
\mathcal{F}(c,y)(z,V)= D_{y}\Pi_{c,y}^e (z) \big(\Pi^{e}_{c,y}V - (I-\Pi^{e}_{c,y})V\big).
\end{equation}
Here we can take the last part of $\mathcal{F}$ in the form of $V - b (I-\Pi^{e}_{c,y})V$ for any $b$, which would not change the validity of \eqref{E:CF-def} for $V \in X_{c, y}^e$. The above choice of $\mathcal{F}$ would bring certain convenience in some calculation later. Using the smoothness of $\Pi_{c, y}^\alpha$ in $y$ given in Lemma \ref{decomp2}, we obtain 
\begin{equation} \label{E:CF}
\Vert \mathcal{F}(c, y)(z, V) \Vert_{X_1} \le C |z| \Vert V \Vert_{X_1}
\end{equation}
for some $C>0$ independent of $y$. The bilinear operator $\mathcal{F}$ is a modification of the second fundamental form of the bundle $X_{c, y}^e$ over $\BR^3$ as a sub-bundle of $X_1 = X_{c, y}^e \oplus \big((I-\Pi_{c, y}^e) X_1\big)$ over $\BR^3$.  

While \eqref{E:VE} is deduced from \eqref{E:Ve}, actually the opposite also holds if $V(s) \in X_{c, y(s)}^e$ for some $s$. To see this, applying $I-\Pi_{c,y}^e$ to \eqref{E:VE} we obtain
\[
 \partial_{t}\big((I - \Pi_{c,y}^e)V\big)= (I-\Pi_{c,y}^e) D_{y}\Pi_{c,y}^e (\partial_{t}y)  \big(\Pi^{e}_{c,y}V - (I-\Pi^{e}_{c,y})V\big)   - D_{y}\Pi_{c,y}^e(\partial_{t}y)V.
\]
Differentiating $\Pi_{c,y}^e \Pi_{c,y}^e=\Pi_{c,y}^e$ with respect to $y$ we have 
\begin{equation}\label{E:Pi-e-temp-0}
D_{y}\Pi_{c,y}^e(\cdot) \Pi_{c,y}^e+ \Pi_{c,y}^e D_{y}\Pi_{c,y}^e(\cdot) =D_{y}\Pi_{c,y}^e(\cdot).
\end{equation}
It follows that 
\begin{equation}\label{E:Pi-e-temp-1}
 \partial_{t}\big((I - \Pi_{c,y}^e)V\big) =-D_{y}\Pi_{c,y}^e(\partial_{t}y) (I - \Pi_{c,y}^e) V.
\end{equation}
Since this is a well-posed homogeneous linear equation of $(I - \Pi_{c,y}^e)V$, which is finite dimensional, the solution has to vanish if we assume $V(s) \in X_{c, y(s)}^e$. Therefore 
\begin{equation} \label{E:V-e}
V(t)\in X_{c, y(t)}^e, \;\;  \forall t, \; \text{ if } V(s) \in X_{c, y(s)}^e \; \text{ and } V(t) \; \text{ solves \eqref{E:VE}}. 
\end{equation} 
Finally \eqref{E:Ve} follows from applying $\Pi_{c,y}^e$ to \eqref{E:VE}. 

Compared to \eqref{E:Ve}, equation \eqref{E:VE} is more convenient as the latter may be posed on the whole space $X_1$. Along with the boundedness of $\mathcal{F}$, it makes it easier to prove the local well-posedness and obtain estimates of \eqref{E:VE} and thus we will mainly work with \eqref{E:VE}. 

In summary, in a neighborhood of $\mathcal{M} \subset X_0$, equation \eqref{GP} written in the bundle coordinates $(y, a^{d1}, a^{d2}, a^+, a^-, V^e) \in B^{d_1+d_2+2d} (\delta) \oplus \CX^e (\delta)$ is equivalent to the system consisting of  \eqref{eqy'}, \eqref{E:a^pm}, \eqref{E:a^d1}, \eqref{E:a^d2}, and \eqref{E:VE}, along with \eqref{E:em-perp-G}.

\section{Linear analysis} \label{S:linear}

We first analyze the linear part of \eqref{E:VE} whose unknown is valued in a vector bundle $X_{c, y}^e$ over $\BR^3$. However it is observed that \eqref{E:VE} is well-posed with $V^e \in X_1$, we will consider this general situation as well as the case $V \in X_{c, y}^e$. Relaxing the restriction $V\in X_{c, y}^e$ would provide a little convenience in some estimates later. Moreover since $G^e$ does not necessarily belong to $X_1$, we give a space-time estimate which will be used to close the nonlinear estimates in later sections. 
Consider the following more general form of \eqref{E:VE}
\begin{equation} \label{lv}
\p_t V = \Pi_{c, y}^e J L_{c, y} \Pi_{c, y}^e V + \mathcal{F} (c, y) (\p_t y, V) + f(t).
\end{equation}
Here we assume $y(t)$, $-\infty \le t_0 \le t \le t_1 \le \infty$ satisfy
\begin{equation} \label{E:linear-1}
\sigma := |\p_t y|_{L^\infty\big((t_0, t_1), \BR^3\big)} \le 1, \quad 
\forall t\in (t_0, t_1).
\end{equation}
For the non-homogeneous term $f= \big(f_1(t, x), f_2(t, x)\big)$, we need the norm 
\begin{equation} \label{E:STnorm} 
\Vert f\Vert_{\hat X_{(t_0, t_1)}^{p, q}} \triangleq \Vert f_1 \Vert_{L^{p}_{(t_{0},t_1)}{B}^1_{q,2}} +  \Vert  f_2 \Vert_{L^p_{(t_{0},t_1)}\dot{B}^{1}_{q,2}}
\end{equation}
along with the associated spaces $\hat X_{(t_0, t_1)}^{p, q}$ and $\hat X_{(t_0, t_1), loc}^{p, q}$, where $B_{p, r}^s$ and $\dot B_{p, r}^q$denote the standard Besov space as well as the homogeneous Besov space, respectively. In the standard terminology, an admissible Stritchartz pair $(p, q)$ and conjugate exponent $p'$ of $p \in [1, \infty]$ are those satisfying 
\begin{equation} \label{E:S-pairs}
p,q\in [2,\infty], \quad 2/p+3/q=3/2;  \qquad 1/p'+1/p=1.
\end{equation}
Our main goal in this section is to prove the following proposition. 

\begin{proposition} \label{P:linear-0}
Suppose \eqref{E:linear-1} holds, $(p, q)$ is a Stritchartz pair, and $f \in \hat X_{(t_0, t_1), loc}^{\tilde p', q'}$ where  $\tilde p \in [1, p]$. Then for any $s\in (t_0, t_1)$ and initial value $V(s) \in X_1$, \eqref{lv} has a unique solution $V(t) \in X_1$. Moreover, there exists $C>0$ independent of $t_0, t_1, s, \sigma, y(\cdot)$, and $f(\cdot)$, such that for any $t \in (t_0, t_1)$, and $\eta > C\sigma$, we have  
\begin{align*}
\langle  L_{c,  y(t)} V^e(t), V^e(t) \rangle^{\frac 12}  
\le & e^{ C\sigma |t-s|} \langle L_{c, y(s)} V^e(s), V^e(s) \rangle^{\frac 12} + C\eta^{-\frac 1{\tilde p}} \Vert e^{\eta |t- \cdot|} f^e \Vert_{\hat X_{(s, t)}^{\tilde p', q'}}\\
|V^\perp(t)| \le &  e^{C\sigma |t-s|} | V^\perp(s)| + C\eta^{-\frac 1{\tilde p}} \Vert e^{\eta |t- \cdot|} f^\perp \Vert_{L_{(s, t)}^{\tilde p'}},
\end{align*}
where 
\begin{equation} \label{E:temp-1} \begin{split}
& V^e (t) = \Pi_{c, y(t)}^e V(t), \quad V^\perp (t) = (I-\Pi_{c, y(t)}^e) V(t), \\ 
&f^e(t)= \Pi_{c, y(t)}^e f(t), \quad  f^\perp (t)  = (I-\Pi_{c, y(t)}^e) f(t),
\end{split}\end{equation} 
satisfying 
\begin{equation} \label{E:V-perp} \begin{split}
&\p_t V^e = \Pi_{c, y}^e J L_{c, y} \Pi_{c, y}^e V^e + \mathcal{F} (c, y) (\p_t y, V^e) + f^e(t), \\
& \partial_{t}V^\perp  =-D_{y}\Pi_{c,y}^e(\partial_{t}y) V^\perp  + f^\perp. 
\end{split} \end{equation}
\end{proposition}

Here one keeps in mind that $I- \Pi_{c,y}^e$ may be applied to a larger class of functions than $X_1$ and its range is finite dimensional. The above decoupling of $V^\perp$ and $V^e$ is due to the choice \eqref{E:CF-def} of $\mathcal{F}$. From the positivity of of $L_{c, y}$ on $X_{c, y}^e$ (Lemma \ref{decomp2}), we have 

\begin{corollary} \label{C:linear-1}
There exists $C>0$ independent of $t_0, t_1, s, y(\cdot)$, and $f(\cdot)$, such that for any $t \in (t_0, t_1)$ and $\eta > C\sigma$, we have  
\[
\Vert V(t) \Vert_{X_1} \le C \big( e^{C\sigma |t-s|} \Vert V(s) \Vert_{X_1} +\eta^{-\frac 1{\tilde p}} \Vert e^{\eta |t-\cdot|} f(\cdot)\Vert_{\hat X_{(s, t)}^{\tilde p', q'}}\big).
\]
Moreover, $V(t)\in X_{c, y(t)}^e$, for almost all $t\in (t_0, t_1)$, if $V(s) \in X_{c, y(s)}^e$ and 
\begin{equation} \label{E:linear-2}
(I-\Pi_{c, y(t)}^e) f(t, \cdot)=0, \quad \forall \; a.e. \; t\in (t_0, t_1).
\end{equation}
\end{corollary}

The above estimates indicate that the linear equation \eqref{lv} exhibits at most weak exponential growth due to $|\p_t y|$. 

Based on the regularity of the nonlinearity given in Lemma \ref{ESG}, we also consider the space 
\[
\wt X_{(t_0, t_1)}  \triangleq L^2 \big((t_0, t_1), X_1\big)  + L^2 \big((t_0, t_1), W^{1, \frac 32} \big) + L^2 \big((t_0, t_1), L^{\frac 32} \cap \dot W^{1, \frac 65} \big)
\]
and $\wt X_{(t_0, t_1), loc}$. The next proposition will be a simple consequence of Proposition \ref{P:linear-0}. 

\begin{proposition} \label{P:linear}
Suppose \eqref{E:linear-1} holds and $f \in \wt X_{(t_0, t_1), loc}$. Then for any $s\in (t_0, t_1)$ and initial value $V(s) \in X_1$, \eqref{lv} has a unique solution $V(t) \in X_1$.  Moreover, there exists $C>0$ independent of $t_0, t_1, s, \sigma, y(\cdot)$, and $f(\cdot)$, such that for any $t \in (t_0, t_1)$, and $\eta
\in (C\sigma, 1)$, we have 
\[
\Vert V^e (t) \Vert_{X_1} \le C \big( e^{C\sigma |t-s|} \Vert V^e(s) \Vert_{X_1} +\eta^{-\frac 12} \Vert e^{\eta |t-\cdot|} f^e(\cdot)\Vert_{\wt X_{(s, t)}}\big).
\]
\[
| V^\perp (t) | \le C \big( e^{C\sigma |t-s|} | V^\perp (s) | +\eta^{-\frac 12} | e^{\eta |t-\cdot|} f^\perp (\cdot)|_{L_{(s, t)}^2}\big)
\]
where $V^{e, \perp}$ and $f^{e, \perp}$ are defined in defined in \eqref{E:temp-1} which satisfy \eqref{E:V-perp}. 
\end{proposition}
These two propositions and Corollary \ref{C:linear-1} will be proved in the rest of the section. \\

\noindent {\bf Energy estimates of homogeneous linear equation.} We start with the basic well-posedness and energy estimates of the homogeneous equation of \eqref{lv} based on the uniform positivity of $L^e (y) = L_{c, y}|_{X_{c, y}^e}$. 

\begin{lemma} \label{L:homo}
Assume $f\equiv 0$, then \eqref{lv} defines a bounded solution map 
\[
S(t, s) \in L(X_1, X_1), \; \forall t, s \in [t_0, t_1], 
\]
with initial value given at $t=s$, which satisfies  
\begin{equation} \label{E:S-1}
S(s, s) =I, \quad S(t, t') S(t', s) = S(t, s), \quad S(t, s) \in L(X_{c, y(s)}^e, X_{c, y(t)}^e).  
\end{equation}
Moreover there exists $C>0$ independent of $t_1$, $t_2$, $t$, $s$, and $y(\cdot)$ such that 
\begin{equation} \label{E:S-2}
\langle L_{c, y(t)} S(t, s) V, S(t, s) V \rangle  \le e^{C \sigma |t-s|} \langle L_{c, y(s)} V, V\rangle, \; \; \forall V \in X_{c, y(s)}^e. 
\end{equation}
\end{lemma}

As a consequence, the lemma implies that, under the assumption $f\equiv0$, \eqref{lv} preserve the constraint $V \in X_{c, y}^e$ if it holds initially. Later we will show that this holds for non-homogeneous equation as well. Furthermore the homogeneous equation induces possible exponential growth only due to $ \Vert \p_t y\Vert_{L^\infty}$. The coefficient 1 in front of the above exponential is important for future estimates. 

\begin{proof}
From Lemma \ref{L:A_e}, $A_e(y) =\Pi_{c, y}^e J L_{c, y} \Pi_{c, y}^e$ is a bounded perturbation to $JL_{c, \infty}$ on $X_1$. This, along with the boundedness of $\mathcal{F}$, implies that \eqref{lv} is well-posedness on $X_1$ and thus the solution flow $S(t, s)$ of bounded linear operators is well-defined. 

Since $(I- \Pi_{c, y}^e) X_1 \subset D(JL_{c, \infty})$, for any $V \in X_1$, by direct computation using \eqref{lv}, one finds that $(I-\Pi_{c,y}^e)V$ satisfies
\[
 \partial_{t}\big((I - \Pi_{c,y}^e)V\big)= (I-\Pi_{c,y}^e) D_{y}\Pi_{c,y}^e (\partial_{t}y)\big(\Pi^{e}_{c,y}V - (I-\Pi^{e}_{c,y})V\big) - D_{y}\Pi_{c,y}^e(\partial_{t}y)V.
\]
Following the same procedure as in Subsection \ref{csntw}, we obtain exactly the same equation as \eqref{E:Pi-e-temp-1} which yields  

\begin{equation} \label{E:Ve-c}
(I-\Pi_{c,y(t)}^e) V(t)=0, \; \forall t, \; \text{ if } (I-\Pi_{c,y(s)}^e)  V(s)=0.  
\end{equation} 

Finally we prove inequality \eqref{E:S-2}. Let $V(t)$ be a solution of \eqref{lv} with $V(s) \in X_{c, y(s)}^e \cap D(JL_{c, \infty})$, which yields $V(t) \in X_{c, y(t)}^e \cap D(JL_{c, \infty})$ for all $t$. By direct calculation using $J^* = -J=J^{-1}$, one has

\begin{align*}
\partial_{t}\langle L_{c,y} {V},  V\rangle &=  \langle J^{-1} D_y \wt Q (y) (\partial_{t}y) {V},  {V}\rangle+2\langle  \p_t {V}, L_{c,y} {V}\rangle\\
&= - \langle J D_y \wt Q (y) (\partial_{t}y) {V},  {V}\rangle+2\langle  \mathcal{F}(c, y) (\p_t y, V), L_{c,y} {V}\rangle
\end{align*}

where $\wt Q$ was defined in \eqref{E:tQ}. It follows from the bounds \eqref{E:CF} and \eqref{E:tQ-1} that
\begin{equation} \label{E:p_tVe}
|\partial_{t}\langle L_{c,y} {V},  V\rangle | \le C |\p_t y| \Vert V \Vert_{X_1}^2.
\end{equation} 
Recall from Lemma \ref{decomp2} that the bounded symmetric quadratic form $\langle L_{c, y}  \cdot,  \cdot \rangle$ satisfies $\langle L_{c, y}  V,  V \rangle \ge \epi \Vert  V \Vert^2_{X_1}$ for any $ V \in  X_{c, y}^e$. This uniform lower bound of $L_{c, y}$ on $X_{c, y}^e$, the above estimate, and the Gronwall inequality immediately imply \eqref{E:S-2} when $V(s) \in X_{c, y(s)}^e \cap D(JL_{c, \infty})$. Since $X_{c, y(s)}^e \cap D(JL_{c, \infty})$ is dense in $X_{c, y(s)}^e$, a standard density argument yields \eqref{E:S-2} for general solution $V(t) \in X_{c, y(t)}^e$. The proof of the lemma is complete. 
\end{proof}

\noindent {\bf Space-time estimates of \eqref{lv}.} Given initial data at $t=s \in [t_0, t_1]$, the solution of \eqref{lv} can be written as 
\begin{equation} \label{E:lv-Duamel}
V(t) = S(t, s) V(s) + \int_s^t S(t, \tau) f(\tau) d\tau.  
\end{equation}
Since $f(t)$ is not assumed to be in $X_1$, we first prove the following lemma.

\begin{lemma} \label{L:smallT}
Suppose \eqref{E:linear-1} and \eqref{E:linear-2} hold, $(\tilde p, \tilde q)$ is an admissible pairs, and $f \in \hat X_{(t_0, t_1),loc}^{\tilde p', \tilde q'}$.

Then for any given $s\in (t_0, t_1)$ and initial value $V(s) \in X_{c, y(s)}^e$, \eqref{E:lv-Duamel} has a unique solution $V(t)$ satisfying  $V(t)\in X_{c, y(t)}^e$, for almost all $t\in (t_0, t_1)$. Moreover, for any admissible pair $(p, q)$, there exists $T, C>0$ independent of $f, t_0, t_1$, and $y(\cdot)$ such that, if $t_0 <t_0' \le s \le t_1' < t_1$ satisfy $t_1'-t_0' \le T$, then 
\begin{equation} \label{E:smallT-1} \begin{split}
\Vert V(t) \Vert_{\hat X_{(t_0', t_1')}^{p, q}}\le C\big( \Vert V(s) \Vert_{X_1} + \Vert f \Vert_{\hat X_{(t_0', t_1')}^{\tilde p', \tilde q'}}  \big).
\end{split} \end{equation}
In particular, if $(p, q) =(\infty, 2)$, then it holds, for $t \in (t_0', t_1')$ 
\begin{equation} \label{E:smallT-2} \begin{split}
\langle L_{c, y(t)} V(t), V(t) \rangle^{\frac 12}  \le & e^{C\sigma |t-s|} \langle L_{c, y(s)} V(s), V(s) \rangle^{\frac 12} + C \Vert f\Vert_{\hat X_{(t_0', t_1')}^{\tilde p', \tilde q'}}.
\end{split} \end{equation}
\end{lemma}

\begin{proof}
We prove the lemma in several steps. \\

\noindent $\bullet$ {\it Step 1. Change of variables and dispersive estimates of the constant coefficient homogeneous linear equation.} To make it more convenient to carry out the dispersive estimates, we first apply a similar transformation to diagonalize $JL_{c=0, \infty}$. Let 
\[
 P=\sqrt{-\Delta(2-\Delta)^{-1}}, \quad R=\sqrt{-\Delta(2-\Delta)}, \quad \mathcal{P}= \begin{bmatrix}P&0\\0&1\end{bmatrix}, \quad V = \mathcal{P}Z.
\]
Apparently, 
\begin{equation} \label{CP}
\mathcal{P} \text{ is an isomorphism from } \dot H^1 \text{ to } X_1. 
\end{equation}
From \eqref{lv}, \eqref{E:homo-1}, and \eqref{E:tQ}, it is straight to compute that $Z$ satisfies
\begin{equation} \label{E:lv-1} 
Z_{t} = JHZ+Q\big(y(t), \p_t y(t)\big) Z +  \tilde f(t).
\end{equation}
where 
\begin{align*}
& H=\begin{bmatrix}R& -c\cdot\nabla\\c\cdot\nabla&R\end{bmatrix}, \quad \tilde f(t) = \mathcal{P}^{-1} f(t),\\
&Q(y, z) = \mathcal{P}^{-1} \big( \wt Q(y) + \mathcal{F}(c, y) (z, \cdot) \big)  \mathcal{P}. 
\end{align*}

Let $\mathcal{R} = \begin{pmatrix} R & 0 \\ 0 & R \end{pmatrix}$. From\eqref{E:CF},  \eqref{E:tQ-1}, \eqref{CP}, and our assumptions, we have 
\begin{equation} \label{E:Q-1}
\Vert Q(y, z) Z \Vert_{\dot H^1} \le C ( 1+ |z|)  \Vert Z \Vert_{\dot H^1}, \quad \tilde f \in L^{\tilde p'}_{(t_{0},t_1)}\dot{B}^{1}_{\tilde q',2}
\end{equation}
for some $C>0$ independent of $y$. 

It was proved in \cite{GNT} that for any $q\in [2,\infty]$, one has
\[
\Vert e^{tJ\mathcal{R}}\phi\Vert_{\dot{B}_{q,2}^{r}}\lesssim t^{-3(\frac 12 - \frac 1q)}\Vert \phi\Vert_{\dot{B}^{r}_{q',2}}. 
\]
Furthermore, for any admissible pairs $(p_{j}, q_{j})$, $j=1,2$, it holds 

\begin{align*}
&\Vert e^{tJ\mathcal{R}}\phi\Vert_{L^{p_{1}}\dot{B}_{q_{1},2}^{r}}\lesssim\Vert \phi\Vert_{\dot{H}^{r}},\quad \Vert\int_{0}^{t}e^{(t-\tau)J \mathcal{R}}f(\tau)d\tau\Vert_{L^{p_{1}}\dot{B}^{r}_{q_{1},2}}\lesssim \Vert f\Vert_{L^{p'_{2}}\dot{B}^{r}_{q'_{2},2}}.
\end{align*}

These estimates lead to
\begin{equation}\label{gnt} \begin{split}
&\Vert e^{tJH}\phi\Vert_{\dot{B}_{q,2}^{r}}\lesssim t^{-3(\frac 12 - \frac 1q)}\Vert \phi\Vert_{\dot{B}^{r}_{q',2}}, \quad \Vert e^{tJH}\phi\Vert_{L^{p_{1}}\dot{B}_{q_{1},2}^{r}}\lesssim\Vert \phi\Vert_{\dot{H}^{r}},\\
&\Vert\int_{0}^{t}e^{(t-\tau)J H}g(\tau)d\tau\Vert_{L^{p_{1}}\dot{B}^{r}_{q_{1},2}}\lesssim \Vert g\Vert_{L^{p'_{2}}\dot{B}^{r}_{q'_{2},2}}.
\end{split} \end{equation} 
In fact, since $JH - J\mathcal{R}= c \cdot \nabla$ which commutes with $J \mathcal{R}$, we have 
\[
e^{tJH} Z = (e^{tJ \mathcal{R}} Z) (\cdot +ct) = e^{tJ \mathcal{R}} \big( Z (\cdot +ct)\big). 
\]
The first two of the inequalities in \eqref{gnt} follow immediately due to the translation invariance of the Besov norms. To see the last one in \eqref{gnt}, 
\begin{align*}
&\Vert\int_{0}^{t}e^{(t-\tau)J H}g(\tau)d\tau\Vert_{L^{p_{1}}\dot{B}^{r}_{q_{1},2}} = \Vert\int_{0}^{t} e^{(t-\tau)J \mathcal{R}}\big(g(\tau) (\cdot + c(t-\tau))\big) d\tau\Vert_{L^{p_{1}}\dot{B}^{r }_{q_{1},2}}\\
=& \Vert\int_{0}^{t} e^{(t-\tau)J \mathcal{R}}\big(g(\tau) (\cdot -c \tau) \big)d\tau\Vert_{L^{p_{1}}\dot{B}^{r }_{q_{1},2}} \lesssim \Vert g(t) (\cdot - ct)\Vert_{L^{p'_{2}}\dot{B}^{r}_{q'_{2},2}}= \Vert g\Vert_{L^{p'_{2}}\dot{B}^{r}_{q'_{2},2}}.
\end{align*}

\noindent $\bullet$ { \it Step 2. Space-time estimate.} In the next step, instead of \eqref{E:lv-Duamel}, we will obtain the space-time estimate of solutions of \eqref{E:lv-1} based on \eqref{gnt} with 
\begin{equation} \label{E:lv-esti-1}
\begin{split}
&Z(t)- e^{(t-s) JH} Z(s) = \int_{s}^{t}e^{(t-\tau) JH}\big[Q\big(y(\tau), \p_t y(\tau) \big) Z(\tau) + \tilde f(\tau)\big]d\tau
\\=&
\int_{0}^{t-s}e^{ (t-s -\tau) J H}\big[Q\big(y(s+\tau), \p_t y(s +\tau)\big)Z(s+\tau)+ \tilde f(s+\tau) \big]d\tau.
\end{split}
\end{equation}
By \eqref{gnt}, \eqref{E:tQ-1}, \eqref{E:Q-1}, \eqref{CP}, and $\sigma = |\p_t y|_{L^\infty} \le 1$,   
for admissible pairs $(p, q)$, $(\tilde p, \tilde q)$ and $(\infty, 2)$, and $t_0 < t_0' \le s \le t_1' < t_1$, we have 
\begin{equation}\label{XT1}\begin{split}
&\Vert Z(t) -  e^{(t-s) JH} Z(s) \Vert_{L^{p}_{(t_{0}', t_1')}\dot{B}_{q,2}^{1}} 
\le  C\big( \Vert Z \Vert_{L^{1}_{(t_{0}',t_1')}\dot{H}^{1}} + \Vert \tilde f \Vert_{L^{\tilde p'}_{(t_{0}',t_1')}\dot{B}^{1}_{\tilde q',2}}\big)\\
\le & C\big( (t_1' -t_0') \Vert Z \Vert_{L^{\infty}_{(t_{0}',t_1')}\dot{H}^{1}} + \Vert \tilde f \Vert_{L^{\tilde p'}_{(t_{0}',t_1')}\dot{B}^{1}_{\tilde q',2}}\big).
\end{split} \end{equation}

Consider the standard splitting of $Z(t)$ into 
\[
Z(t) = Z_h (t) + Z_{in}(t), \quad V(t) = \mathcal{P} Z_h(t) + \mathcal{P} Z_{in} (t), 
\]
where $Z_h(t)$ satisfies the corresponding homogeneous equation of \eqref{E:lv-1} (i.e. without $\tilde f$) and $Z_h(s) = Z(s)$, and $Z_{in}(t)$ solves \eqref{E:lv-1} and $Z_{in}(s)=0$.\\

\noindent $\bullet$ {\it Step 3. Non-homogeneous part $Z_{in}$.} On the one hand, applying \eqref{XT1} to $Z_{in}(t)$
with the admissible pair $(p=\infty, q=2)$, we obtain that there exists $T>0$ independent of $t_0$, $t_1$, $t_0'$, $t_1'$, and $y(t)$ such that, if $t_1' -t_0' \le T$, it holds 
\[
\Vert Z_{in} \Vert_{L_{(t_{0}',t_{1}')}^{\infty}\dot{H}^{1}} \le C \Vert \tilde f\Vert_{L^{\tilde p'}_{(t_{0}',t_1')}\dot{B}^{1}_{\tilde q',2}}.
\]
Substituting this back into \eqref{XT1}, we have that for  any $t_1' - t_0' \le T$ and admissible pairs $(p,q)$,
\begin{equation}\label{XTG}
\Vert Z_{in} \Vert_{L^{p}_{(t_{0}',t_1')}\dot{B}_{q,2}^{1}} \le C \Vert \tilde f \Vert_{L^{\tilde p'}_{(t_{0}',t_1')}\dot{B}^{1}_{\tilde q',2}}.
\end{equation}

We claim that $(I-\Pi_{c, y(t)}^e) f(t)=0$ implies  
\begin{equation} \label{E:PZin}
\mathcal{P} Z_{in} (t)= \int_s^t S(t, \tau) f(\tau) d\tau \in X_{c, y(t)}^e, \quad a.e. \;\; t\in (t_0', t_1'). 
\end{equation}
In fact, let 
\begin{align*}
&Y= \big\{ \tilde g\in L^{\tilde p'}_{(t_{0}',t_1')}\dot{B}^{1}_{\tilde q',2}  \mid (I-\Pi_{c, y(t)}^e) \mathcal{P} g(t) =0, \; \forall t\in (t_0', t_1')\big\} \subset L^{\tilde p'}_{(t_{0}',t_1')}\dot{B}^{1}_{\tilde q',2}.
\end{align*}

Since $\zeta_c^\alpha \in H^k \times (H^k \cap H^{-1})$, for any $k \ge 1$ and $\alpha \in \{T, d1, d2, +, -\}$, $I-\Pi_{c, y}^e$ actually applies to $\mathcal{P} g(t) \in B_{\tilde q', 2}^1 \times \dot B_{\tilde q', 2}^1$. Consider the mapping $\Gamma$ 
\[
W(t) = (\Gamma g) (t) \triangleq \int_s^t S(t, \tau) \mathcal{P} \tilde g(\tau) d\tau. 
\]
Inequality \eqref{XTG} and the definition of $\mathcal{P}$ imply that $\Gamma: L_{(t_{0}',t_1')}^{\tilde p'} \dot{B}^{1}_{\tilde q',2} \to L_{(t_{0}',t_1')}^\infty X_1$ is a bounded operator, and thus also bounded when restricted to $Y$. Since Lemma \ref{L:homo} implies 
\[
(I-\Pi_{c, y(t)}^e) (\Gamma g) (t) =0, \; \forall t \in (t_0', t_1') \quad \text{ if } \; g \in Y \cap L_{(t_{0}',t_1')}^\infty X_1
\]
and $Y \cap L_{(t_{0}',t_1')}^\infty X_1$ is dense in $Y$, we obtain that $(I-\Pi_{c, y(t)}^e) \Gamma$ vanishes on $Y$. Therefore $\mathcal{P} Z_{in} (t) \in X_{c, y(t)}^e$ for almost all $t\in (t_0', t_1')$. \\

\noindent $\bullet$ {\it Step 4. Homogeneous part and the completion of the proof of the lemma.} On the other hand, it is clear 
\[
Z_h(t) = \mathcal{P}^{-1} S(t, s) V(s), \quad \mathcal{P} Z_h(t) \in X_{c, y(t)}^e,
\]
which also implies $V(t) =  \mathcal{P} Z_h(t) + \mathcal{P} Z_{in} (t) \in X_{c, y(t)}^e$ for almost all $t \in (t_0', t_1')$. Applying Lemma \ref{L:homo} to the $\Vert\cdot \Vert_{L^\infty \dot H^1}$ on the right side of \eqref{XT1} for $Z_h$, we have 
\[
\Vert Z_h - e^{(t-s) JH} Z(s) \Vert_{L^{p}_{(t_{0}',t_1')}\dot{B}_{q,2}^{1}} \le C (t_1' - t_0')  \Vert V(s) \Vert_{X_1} 
\]
for $t_0' < s < t_1'< t_0'+T$. From \eqref{gnt}, we obtain  
\[
\Vert Z_h \Vert_{L^p_{(t_{0}',t_1')}\dot{B}_{q,2}^{1}} \le C  \Vert V(s) \Vert_{X_1}
\]
and inequality \eqref{E:smallT-1} follows immediately from the above estimates. 

To derive \eqref{E:smallT-2} in the case of $(p=\infty, q=2)$, we apply \eqref{E:S-2} instead, along with \eqref{XTG}, \eqref{E:PZin}, and \eqref{CP} and the uniform positivity of $L_{c, y}$ on $X_{c, y}^e$, to compute, for $t \in (t_0', t_1')$, 
\begin{align*}
& \langle L_{c, y(t)} V(t),  V(t) \rangle = \langle L_{c, y(t)} S(t, s) V(s), S(t, s) V(s) \rangle \\
& \qquad \qquad + 2  \langle L_{c, y(t)} S(t, s) V(s), \mathcal{P} Z_{in}(t) \rangle + \langle L_{c, y(t)}  \mathcal{P} Z_{in}(t), \mathcal{P} Z_{in}(t) \rangle \\
\le & \big( \langle L_{c, y(t)} S(t, s) V(s), S(t, s) V(s) \rangle^{\frac 12} + \langle L_{c, y(t)}  \mathcal{P} Z_{in}(t), \mathcal{P} Z_{in}(t) \rangle^{\frac 12} \big)^2 \\
\le &\big( e^{C \sigma |t-s|} \langle L_{c, y(s)} V(s), V(s)\rangle^{\frac 12} + C \Vert \tilde f \Vert_{L^{\tilde p'}_{(t_{0}',t_1')}\dot{B}^{1}_{\tilde q',2}} \big)^2.
\end{align*}
This implies \eqref{E:smallT-2}. Finally as $T$ is independent of $f$ and $y(t)$, a standard continuation argument extends of solutions on $(t_0, t_1)$ and thus completes the proof of the lemma. 
\end{proof}

In the next step, we iterate the above small time estimates.

\begin{lemma} \label{L:linear-0}
Suppose \eqref{E:linear-1} and \eqref{E:linear-2} hold, $(p, q)$ is an admissible pairs, and $f \in \hat X_{(t_0, t_1)}^{\tilde p', q'}$, where $\tilde p\in [1, p]$. Then there exists $C>0$ independent of $f, t_0, t_1, s, t$, such that for any $\eta > C\sigma$, every solution $V(t)$ to \eqref{E:lv-Duamel} satisfies 
\[
\langle L_{c, y(t)} V(t), V(t) \rangle^{\frac 12}  \le  e^{C\sigma |t-s|} \langle L_{c, y(s)} V(s), V(s) \rangle^{\frac 12} + C\eta^{-\frac 1{\tilde p}} \Vert e^{\eta |t - \cdot|} f \Vert_{\hat X_{(s, t)}^{\tilde p', q'}}
\]
for any $t_0 < t, s < t_1$.
\end{lemma}

\begin{proof} 
We only prove the estimates for $t>s$, the estimates for negative time can be obtained similarly. Suppose $t= s+kT + t'$ where $t'\in [0, T)$, we use \eqref{E:smallT-2} repeatedly to compute 
\begin{align*}
&\langle  L_{c,  y(t)} V(t), V(t) \rangle^{\frac 12}  \le e^{C\sigma t'} \big( \langle L_{c, y(\tau)} V(\tau), V(\tau) \rangle|_{\tau=s+kT}\big)^{\frac 12} + C \Vert f \Vert_{\hat X_{(s+kT, t)}^{p', q'}}\\
\le & e^{C\sigma (t-s)} \langle L_{c, y(s)} V(s), V(s) \rangle^{\frac 12} + C\big(\Vert f \Vert_{\hat X_{(s+kT, t)}^{\tilde p', q'}} \\
& \qquad  \qquad \qquad + \sum_{j=1}^k e^{C\sigma (t-s - jT)} \Vert f \Vert_{\hat X_{(s+(j-1)T, s+jT)}^{\tilde p', q'}} \big) \\
\le & e^{ C\sigma (t-s)} \langle L_{c, y(s)} V(s), V(s) \rangle^{\frac 12} + C( \sum_{j=0}^k e^{\tilde p(C\sigma -\eta) (t-s- jT)} )^{\frac 1{\tilde p}} \Vert e^{\eta (t - \cdot)} f \Vert_{ \hat X_{(s, t)}^{\tilde p', q'}}.
\end{align*}

Summing up the exponentials completes the proof of the lemma.
\end{proof} 

Finally we drop the assumption \eqref{E:linear-2}. \\

\noindent {\bf Proof of Proposition \ref{P:linear-0}.}
We split \eqref{lv} into the $X_{c, y}^e$ component and its complementary component as in \eqref{E:temp-1}. Much as the calculation in the derivation of \eqref{E:Pi-e-temp-1}, we have that $V^\perp$ satisfies \eqref{E:V-perp}.  

The remaining estimate of $V^\perp(t)$ is similar to the above. In fact, for $t>s$, 
\begin{align*} 
|V^\perp(t)| \le & e^{C \sigma (t-s)} |V^\perp (s)| + \int_s^t e^{C\sigma (t-\tau)} |f^\perp (\tau)| d\tau\\
\le &  e^{C \sigma (t-s)} |V^\perp (s)| + \int_s^t  e^{(C \sigma - \eta) (t-\tau )} e^{\eta (t-\tau)} |f^\perp (\tau)| d\tau\big)
\end{align*}
which implies the desired estimate on $V^\perp(t)$. 

Due to the choice of $\mathcal{F}$, it is straight forward to compute that $V^e(t) =V(t)-V^\perp(t)$ satisfies \eqref{lv} with the non-homogeneous term $f(t)$ replaced by $f^e(t)$. Lemma \ref{L:linear-0} implies the estimate on $V^e(t)$ which completes the proof of the proposition. 
\hfill $\square$

Finally, we apply Lemma \ref{L:linear-0} to prove Proposition \ref{P:linear}. \\

\noindent {\bf Proof of Proposition \ref{P:linear}.}
We first decompose $f \in \wt X_{(t_0, t_1), loc}$ into the sum of several terms satisfying the assumptions in Lemma \ref{L:linear-0}. In fact, by the definition of $\wt X$, we can write 
\[
f = \phi + \psi + \gamma
\]
where 
\[
\phi \in L_{loc}^2 W^{1, \frac 32} \subset L_{loc}^2 B_{\frac 32, 2}^1 \subset \hat X_{(t_0, t_1), loc}^{2, \frac 32}, \quad \gamma \in L_{loc}^2 (H^1 \times \dot H^1) \subset  \hat X_{(t_0, t_1), loc}^{2, 2}, 
\]
and 
\[
\psi  \in L_{loc}^2 (L^2\cap L^{\frac 32}), \quad  \nabla \psi \in L_{loc}^2 L^{\frac 65}.
\]
Let $\chi$ be the same smooth cut-off function used in Subsection \ref{SS:X_0}. Clearly 
\[
\chi(D) \psi \in L_{loc}^2 W^{1, \frac 32} \subset L_{loc}^2 B_{\frac 32, 2}^1 \subset \hat X_{(t_0, t_1), loc}^{2, \frac 32}, \; \text{ and } \nabla \big(1-\chi(D)\big) \psi \in L_{loc}^2 L^{\frac 65}. 
\]
Moreover, since the inverse Fourier transform of $\frac {1-\chi(|\xi|)}{|\xi|}$ is in $L^1$, we have 
\[
\big(1-\chi(D)\big) \psi = \frac {1-\chi(D)}{|D|} |D| \psi \in L_{loc}^2 L^{\frac 65}
\]
and thus $\big(1-\chi(D)\big) \psi \in \hat X_{(t_0, t_1), loc}^{2, \frac 65}$. 

The desired estimate follows immediately from applying Proposition \ref{L:linear-0} to each of these terms. \hfill $\square$

\section{Construction of Lipschitz Local Invariant manifolds of $\mathcal{M}$} \label{S:InMa}
Based on the space-time estimates developed in Section \ref{S:linear}, we construct the center-unstable manifold $\CW^{cu}(\CM)$ of $\mathcal{M}$, while the center-stable manifold $\CW^{cs}(\CM)$ can be constructed similarly. The intersection of the center-unstable and the center-stable manifolds yields the center manifold of $\mathcal{M}$.

\subsection{Outline of the construction of the center-unstable manifold of $\mathcal{M}$.} \label{SS:GT}

Our construction roughly follows the procedure in \cite{Ca81}. The codim-d local center-unstable manifold are over the directions of $X_{c, y}^{d1} \oplus X_{c, y}^e \oplus X_{c, y}^{d2} \oplus X_{c, y}^+$ along $\mathcal{M}$. In coordinate system \eqref{E:coord-1} $\CW^{cu}(\CM)$ is represented as the graph of some mapping $h^{cu}$ 
\begin{equation} \label{E:CU-1} \begin{split}
\CW^{cu}(\CM)  = \Phi\big( \big\{ &a^- = h^{cu} (y, a^{d1}, a^{d2}, a^+, V^e) \mid \\
& (y, a^{d1}, a^{d2}, a^+, V^e) \in B^{d_1+d_2+d} (\delta) \oplus \CX^e (\delta)\big\} \big)
\end{split}\end{equation} 
where the above sets are defined in \eqref{E:Bball}. Even though $W^{cu}(\CM)$ is local, by using a standard cut-off technique, we will carry out the construction on $\BR^{d_1+d_2+d} \oplus \CX^e (\delta)$.
Moreover, for technical convenience, we shall work with $h(y, a^{d1}, a^{d2}, a^+, V)$ defined on $\BR^{3+d_1+d_2+d} \times X_1 (\delta)$ to avoid the non-flat bundle. However, only the value of $h$ on $\BR^{d_1+d_2+d} \oplus \CX^e(\delta)$ matters. 

Let 
\[
X^{cu} = \BR^{3+d_1+d_2+d} \times X_1, \; X^{cu} (\delta) = \{ (y, a^{d1}, a^{d2}, a^+, V) \in X^{cu} : \Vert V \Vert_{X_1} < \delta\}. 
\]
The following projection $\tilde \Pi^e$, linear except in $y$, will be used often 
\begin{equation} \label{E:PiW}
\tilde \Pi^e (y, a^{d1}, a^{d2}, a^+, V)= (y, a^{d1}, a^{d2}, a^+, \Pi_{c, y}^e V) \in \BR^{d_1+d_2+d} \oplus \CX^e. 
\end{equation}
We shall modify equations \eqref{eqy'}, \eqref{E:a^pm}, \eqref{E:a^d1}, \eqref{E:a^d2}, and \eqref{E:VE}, along with \eqref{E:em-perp-G}, into a system defined on $X^{cu} \times \BR^d$. As a standard technique in local analysis, we first cut-off the nonlinearities as well as the off-diagonal linear terms in the direction transversal to $\CM$. Take a cut-off function 
\begin{equation} \label{E:gamma} 
\gamma\in C_{0}^{\infty}(\mathbb{R}),\;  \text{ s. t. }
\gamma(x)=1, \; \forall \vert x\vert \leq 1, \; \; \gamma(x)=0, \;  \forall \vert x\vert \geq 3, \; \; \vert \gamma'\vert_{C^0 (\BR)}  \leq 1
\end{equation}
and for $\delta >0$, $a^- \in \BR^d$, and $W= (y, a^{d1}, a^{d2}, a^+, V) \in X^{cu}$, let 
\[
\gamma_\delta (W, a^-) = \gamma\big(3 \delta^{-1}(|a^{d1}| + |a^{d2}| + |a^+| + |a^-|+ \Vert V\Vert_{X_1} )\big).
\]
For $a^- \in \BR^d$ and $W\in X^{cu}$, let 
\begin{align*}
\hat{G}^{\alpha}(W, a^-) = &\gamma_{\delta} (W, a^-) G^\alpha (c, y, w), \qquad  \alpha\in\{+,-, d2\}\\
\hat G^{d1} (W, a^-) = & \gamma_{\delta} (W, a^-) \big( \langle \zeta_c^{d1} (\cdot +y), A_{1e}(y)K_{c, y} w\rangle + M_{12} a^{d2} + G^{d1} (c, y, w)\big)\\
\hat G^T (W, a^-) =&  \gamma_{\delta} (W, a^-) \big( \langle \zeta_c^T(\cdot +y),  A_{Te}(y) K_{c, y} w\rangle + M_{T1} a^{d1} + M_{T2} a^{d2}  \\
&+ G^T (c, y, w)\big)\\
\hat G^e (W, a^-) = &\gamma_{\delta} (W, a^-) \big(  A_{e2}(y)K_{c, y} w+ G^e (c, y, w)\big)
\end{align*}
where functions $\zeta_c^T = (\zeta_{c, 1}^T, \zeta_{c, 2}^T,\zeta_{c, 3}^T)$, $\zeta_c^{d1} = (\zeta_{c, 1}^{d1}, \ldots,\zeta_{c, d_1}^{d1})$, and operators $A_{1e}$, $A_{e2}$, and $A_{Te}$ are given in Lemma \ref{decomp2}, matrices $M_{12}, M_{T1}, M_{T2}$ in \eqref{E:M_ab}, and 
\begin{equation} \label{E:lp-w} \begin{split}
w =&\Lambda(W, a^-) \triangleq Em^\perp \big(y, a^{d1}, a^{d2}, a^+, a^-, \Pi_{c, y}^e V \big)\\
=& K_{c,y}^{-1} \big(( a^{d1} V_{c}^{d1} + a^{d2} V_{c}^{d2} + a^+ V_{c}^+ + a^- V_c^- ) (\cdot + y) + \Pi_{c, y}^e V
\big).
\end{split}\end{equation}

From the definitions of $\hat G^\alpha$, it clearly holds that they are independent of the extra component $(I-\tilde \Pi^e) W$ added to avoid the non-flat bundle $\BR^{d_1+d_2+d} \oplus \CX^e$.  
In particular, $\hat G^e$ satisfies 
\begin{equation} \label{E:wtX}
(I- \Pi_{c, y}^e) \hat G^e =0, \quad \hat G^e \in \wt X= X_1 + W^{1, \frac 32} + (L^{\frac 32} \cap \dot W^{1, \frac 65}). 
\end{equation}
Denote 
\begin{align*} 
& \wt X^{cu} = \BR^{3+d_1+d_2+d} \times \wt X, \; \wt X^{cu} (\delta) = \{ (y, a^{d1}, a^{d2}, a^+, V) \in \wt X^{cu} : \Vert  V \Vert_{\wt X} < \delta\},\\
&\hat G^{cu} (W, a^-) = (\hat G^T, \hat G^{d1}, \hat G^{d2}, \hat G^+, \hat G^e) (W, a^-),  \\
&A^{cu} (y, \tilde y) = diag\big( 0, M_1, M_2, M_+, \Pi_{c, y}^e J L_{c, y} \Pi_{c, y}^e + \mathcal{F} (c, y) (\tilde y, \cdot ) \big).
\end{align*}
We shall consider, for $W= (y, a^{d1}, a^{d2}, a^+, V) \in X^{cu}$ and $a^- \in \BR^d$, 
\begin{subequations} \label{E:cut-GP}
\begin{equation} \label{E:cut-GP-cu}  
\p_t W = A^{cu}\big(y, \hat G^T (W, a^-) \big) W + \hat G^{cu}(W, a^-)
\end{equation} 
\begin{equation} \label{E:cut-GP-s}
\p_t a^- = M_- a^- + \hat G^- (W, a^-)
\end{equation}
\end{subequations}
which, for  $\Vert w \Vert_{X_1} \le \delta$, coincides with the system consisting of equations \eqref{eqy'}, \eqref{E:a^pm}, \eqref{E:a^d1}, \eqref{E:a^d2}, and \eqref{E:VE}, along with \eqref{E:em-perp-G}, the representation of \eqref{GP} in the local coordinate system near $\CM$. 

As the off-diagonal linear blocks in $JL_{c, y}$ are incorporated into $\hat G^{cu}$, the latter does not have small Lipschitz constants, which is often a necessity in constructing local invariant manifolds. Accordingly, we 
introduce metrics involving a scale constant $Q>1$ 
\begin{equation} \label{E:Q-norm} \begin{split}
& \Vert (y, a^{d1}, a^{d2}, a^+, V) \Vert_{X_1, Q} \triangleq |y|+ Q |a^{d1}| + Q^3 |a^{d2}| + |a^+| + Q^2\Vert V \Vert_{X_1}, \\
& \Vert (y, a^{d1}, a^{d2}, a^+, V) \Vert_{\wt X, Q} \triangleq |y|+ Q |a^{d1}| + Q^3 |a^{d2}| + |a^+| + Q^2\Vert V \Vert_{\wt X}
\end{split}\end{equation} 
to make Lipschitz constants of $\hat G^{cu, -}$ small (Lemma \ref{L:G-Lip}).  

We shall construct the local center-unstable manifold $\CW^{cu}(\CM)$ as the graph $\{ a^- = h^{cu} (W)\}$ of some $h: X^{cu}(\delta) \to \BR^d$. Since $\CW^{cu}(\CM)$ is expect to be translation invariant, we will only consider translation invariant mappings $h: X^{cu} (\delta) \to \BR^d$, which satisfy, for any $z \in \BR^3$, 
\begin{equation} \label{E:trans-inv-h}
h \big(y + z, a^{d1}, a^{d2}, a^+, V (\cdot +z) \big) = h (y, a^{d1}, a^{d2}, a^+, V). 
\end{equation}

Fix constants $Q, \delta, \mu$ such that 
\begin{equation} \label{E:parameter-1}
\delta<1 ,\quad Q>1, \quad \mu<\frac 15, 
\end{equation}
whose additional assumptions will be given later. Let 
\begin{equation} \label{E:Gamma-mu} \begin{split} 
\Gamma_{\mu, \delta} = \{&h : X^{cu} (\delta)  \to \BR^d  \mid  h(y, 0,0,0,0) =0, \;  \Vert h\Vert_{C^0} \le \delta/15, \\
& h \text{ satisfies \eqref{E:trans-inv-h}, and } Lip_{\Vert \cdot \Vert_{X_1, Q}} \le \mu.
\}.
\end{split} \end{equation}
Here $h(y, 0,0,0,0) =0$ is required as $\CW^{cu}(\CM)$ should contain $\CM$. 
Clearly $\Gamma_{\mu, \delta}$ equipped with $\Vert \cdot\Vert_{C^0}$ is a complete metric space. 

We will define a transform on $\Gamma_{\mu, \delta}$ based on \eqref{GP}. For any $h \in \Gamma_{\mu, \delta}$ and $\bar W\in X^{cu}(\delta)$, consider the solution $W(t) = (y, a^{d1}, a^{d2}, a^+, V)(t) \in X^{cu}$ of  
\begin{equation}\label{lp}
\p_t W = A^{cu} \big(y, \hat G^T(W, h(W)\big) \big)W + \hat G^{cu} \big(W, h(W) \big), \qquad W(0)= \bar W. 
\end{equation}

\begin{remark}\label{R:C-0}
Even though $h$ is defined only on $X^{cu}(\delta)$, due to the cut-off function $\gamma_{\delta}$, 
for any $h \in \Gamma_{\mu, \delta}$, $\alpha \in \{T, d1, d2, \pm, e\}$, it holds $\hat G^{\alpha} \big(W, h(W)\big) =0$ whenever $W \in X^{cu}\backslash X^{cu}(\delta)$. Consequently,  the right side of \eqref{lp} is well-defined for all $W\in X^{cu}$. 
 
\end{remark}

We define $\widetilde{h}(\bar W)$ as 
\begin{equation}\label{newh}
\widetilde{h} (\bar W)
= \bar a^- =\int_{-\infty}^{0}e^{-t M_-}\hat{G}^{-} \big(W(t), h(W(t))\big) dt.
\end{equation}
We denote this transformation $h \to \wt h$ as 
\[
\mathcal{T} (h) = \wt h.
\]

In order to construct the center-unstable manifold, in the following subsections, under suitable assumptions on $Q$, $\delta$, and $\mu$ we will show $\wt h \in \Gamma_{\mu, \delta}$ is well-defined and that $\CT$ is a contraction on $\Gamma_{\mu, \delta}$. The graph of the unique fixed point, restricted to the set 
\[
B^{d_1+d_2+d} (\delta) \oplus \CX^e (\delta) =  \{(y, a^{d1}, a^{d2}, a^+, V) \in X^{cu}(\delta) \mid |(a^{d1}, a^{d2}, a^+)| < \delta, \ V \in X_{c, y}^e\}
\]
would be the desired center-unstable manifold $\CW^{cu}(\CM)$. To end this subsection, we give the following lemma to show that working on systems \eqref{E:cut-GP} or \eqref{lp} on the expanded domain $X^{cu}$, only to avoid the non-flat bundle $\BR^{d_1+d_2+d} \oplus \CX^e$, does not change the local invariant manifolds.

\begin{lemma} \label{L:reduce} 
The following statements hold. 
\begin{enumerate} 
\item Suppose $W(t)$ satisfies \eqref{E:cut-GP-cu} on $[t_1, t_2]$ for some $a^- \in C^0([t_1, t_2], \BR^d)$ and $\tilde \Pi^e W(t_0)=W(t_0)$ for some $t_0 \in [t_1, t_2]$, then $\tilde \Pi^e W(t) =W(t)$ for all $t \in [t_1, t_2]$. 
\item Assume $h_j \in \Gamma_{\mu, \delta}$, $j=1,2$, satisfy $h_1(W)=h_2(W)$ for all $W \in X^{cu} (\delta)$ with $\tilde \Pi^e W= W$. Then $\tilde h_j$, $j=1,2$, defined in \eqref{newh} satisfy the same property. 
\end{enumerate}
\end{lemma}

\begin{proof} 
For the first statement, we observe that a direct consequence of \eqref{E:wtX},  \eqref{E:V-perp}, and our assumption is $(I-\tilde \Pi^e) W(t) =0$, for all $t\le [t_1, t_2]$, which implies $W(t) = \tilde \Pi^e W(t)$. The second statement of the lemma is just a simple corollary of part (1) and the definition of $\wt h$.  
\end{proof}

\subsection{Apriori Estimates}

Following the construction outlined in Subsection \ref{SS:GT}, in order to prove that $\wt h \in \Gamma_\mu$ is well-defined for any given $h \in \Gamma_{\mu, \delta}$, we start with the following preliminary estimates. 

\begin{lemma} \label{L:G-Lip}
$\hat G^{cu, -}: X^{cu} \times \BR^d \to \wt X^{cu} \times \BR^d$ are smooth, $\hat G^{cu, -} (y,0, 0, 0, 0, 0) =0$. Moreover, there exists $C>0$ independent of $Q$ and $\delta$ such that 
\begin{align*}
\Vert D \hat G^{cu}\Vert_{L_Q (X^{cu} \times \BR^d, \wt X^{cu})}   \le C(Q^{-1} +\delta Q^3), \quad \Vert D\hat G^- \Vert_{L(X^{cu} \times \BR^d, \BR^d)} \le C\delta  
\end{align*}
where $\Vert \cdot \Vert_{L_Q (X^{cu} \times \BR^d, \wt X^{cu})}$ denote the operator norm  when evaluated in $\Vert \cdot \Vert_{X_1, Q}$ and $\Vert \cdot \Vert_{\wt X, Q}$.
\end{lemma}

\begin{proof}
From the definitions of $\hat G^{cu, -}$ and \eqref{E:G^pm-0} and \eqref{E:G^pm-2} which in turn are derived from Lemma \ref{ESG}, the smoothness of $\hat G^{cu, -}$ and $\hat G^{cu, -} (y,0, 0, 0, 0, 0) =0$ follows immediately. Moreover, it is straight forward to obtain the following estimates. Firstly, for $l, k \ge 0$, 
\begin{equation}  \label{E:hat-G-esti-1} \begin{split}
\Vert D_{(a^{d1}, a^{d2}, a^+, a^-, V)}^k &D_y^l (\hat G^{cu} , \hat G^{-} ) (W, a^-) \Vert  \le C_{k , l} \delta^{1-k}
\end{split} \end{equation}
for some $C_{k, l}>0$. When we exclude the off-diagonal terms in $\hat G^{cu}$, the estimates may be improved  to
\begin{equation}  \label{E:hat-G-esti-2} \begin{split}
& |D_{(a^{d1}, a^{d2},  a^+, a^-, V)}^k D_y^l \hat G^{d2, +, -}|  + | D_{(a^+, a^-)}^{k} D_y^l \hat G^T|\\
 & \qquad + \Vert D_{(a^{d1}, a^+, a^-, V)}^{k} D_y^l \hat G^e \Vert  + |D_{(a^{d1}, a^+, a^-)}^{k} D_y^l \hat G^{d1}|  
\le  C_{k , l} \delta^{2-k}, 
 \end{split} \end{equation}
for some $C_{k, l}>0$. In the above $\hat G^e$ is always evaluated in the $\Vert\cdot \Vert_{\wt X}$ norm. The desired estimates on $D \hat G^{cu, -}$ follows from straight forward calculations based on the above inequalities.  
\end{proof}

The following lemma is a simple corollary of Proposition \ref{P:linear-0}. 

\begin{lemma} \label{L:A-cu}
There exists $C>0$ such that, for any $y \in C^1\big((-\infty, 0], \BR\big)$ satisfying $|\p_t y|_{L^\infty} \le \sigma$, $f \in L^2\big((-\infty, 0], \wt X^{cu})$, $Q>1$, $\eta \in (C \sigma, 1)$, and $W(t)$, $t\in (-\infty, 0]$, solving 
\[
\p_t W = A^{cu}(y, \p_t y) W + f, 
\]
we have 
\[
\Vert W(t) \Vert_{X_1, Q}^2 \le C \eta^{-2d_1} e^{-\eta t}  \Vert W(0)\Vert_{X_1, Q}^2 +C \eta^{-2d_1-1} \int_t^0 e^{\eta(\tau -t)} \Vert f(\tau) \Vert_{\wt X, Q}^2 d\tau.  
\]
\end{lemma}

\begin{proof} 
Since $A^{cu}$ takes a diagonal form, we may consider each component individually. For the $a^{d1, d2}$ and $y$ components, in addition to applying \eqref{E:exp-M} we also use 
\[
|e^{tM_1}| + |e^{tM_2}| \le C(1+ |t|)^{d_1} \le C\eta^{-d_1} e^{\frac \eta3 |t|}
\]
and the following estimate based on the Cauchy-Schwartz inequality
\[
| \int_t^0 e^{\frac \eta3 (\tau-t)} g(\tau) d\tau|^2 \le C \eta^{-1} \int_t^0 e^{\eta (\tau-t)} |g(\tau)|^2 d\tau  
\]
to obtain the desired inequality. The estimate on the $V$ component is a direct consequence of Proposition \ref{P:linear-0} and the estimate of the $a^+$ and $y$ component trivially follows from \eqref{E:exp-M}. \end{proof} 

\begin{proposition} \label{P1}
Let $\eta \in (0, 1)$, $T \in [-\infty, 0)$, $h \in \Gamma_{\mu, \delta}$, and $\tilde a_2^-(\cdot) \in C^0 \big( [T, 0], \BR^d)$.
Suppose $W_1(t) = (y_1, a_1^{d1}, a_1^{d2}, a_1^+, V_1)(t) \in X^{cu}$ is a solution to \eqref{lp} and $W_2 = (y_2, a_2^{d1}, a_2^{d2}, a_2^+, V_2)(t) \in X^{cu}$ solves 
\begin{equation} \label{lp-M}
\p_t W = A^{cu}\big(y, \hat G^T(W, \tilde a_2^-)\big) W + \hat G^{cu} (W, \tilde a_2^-),
\end{equation}
with initial values $\bar W_j=(\bar y_j, \bar a_j^{d1}, \bar a_j^{d2}, \bar a_j^+, \bar V_j) \in X^{cu}(\delta)$, $j=1,2$. Then these solutions exist for all $t\in [T, 0]$ and there exists $C>0$ independent of $\mu, T, \eta, Q$, and $\delta$, such that if  \eqref{E:parameter-1} is satisfied and 
\begin{equation} \label{E:parameter-2}
C \eta^{-(1 + d_1)} ( Q^{-1}  +  Q^3 \delta)< 1
\end{equation}
then $W_j(t) \in X^{cu}(C\delta)$ 
for all $t \in [T, 0]$ and 
\begin{align*}
\Vert ( W_2 - W_1)(t) \Vert_{X_1, Q}^2 \le & C\eta^{-2d_1} \Big (e^{-2\eta t} \Vert ( W_2 - W_1)(0) \Vert_{X_1, Q}^2  \\
&+ \eta^{-1} Q^6 \delta^2 \int_t^0 e^{2\eta (\tau-t)} |\big(\tilde a_2^- - h(W_2) \big) (\tau) |^2 d\tau \Big).
\end{align*}
\end{proposition}

\begin{proof} 
To analyze solutions to \eqref{lp} and \eqref{lp-M}, we first note   
\[
\hat G^T (W, a^-) =0,\quad \hat G^e (W, a^-) =0, \; \text{ if } \; \Vert V \Vert_{X_1} \ge \delta.
\]

Therefore if $\Vert V_j \Vert_{X_1} \ge \delta$,  \eqref{E:p_tVe}, \eqref{lp}, and decomposition \eqref{E:V-perp} yield 
\[
\p_t y_j =0, \quad \p_t \langle L_{c, y_j} \Pi_{c, y_j}^e V_j, \Pi_{c, y_j}^e V_j \rangle =0, \quad \p_t (I-\Pi_{c, y_j}^e) V_j =0, 
\]
which along with the initial condition and Lemma \ref{decomp2} yield
\begin{equation} \label{E:delta-1} 
\Vert V_j (t)\Vert_{X_1} \le C \delta, \quad \forall t\in [T, 0].
\end{equation} 

To estimate the difference, let 
\begin{align*}
&B(t) = \mathcal{F} (c, y_2) \big(\hat G^T(W_2, \tilde a_2^-),  \cdot\big) - \mathcal{F} (c, y_1) \big(\hat G^T(W_1, h(W_1)),  \cdot\big)
+ A_e (y_2) - A_e(y_1)\\
& B^{cu}(t) \triangleq  A^{cu}\big(y_2, \hat G^T( W_1, \tilde a_2^-) \big) - A^{cu} \big( y_1, \hat G^T(W_1, h(W_1)) \big) \\
&\qquad =diag\big(0, 0, 0, 0, B(t) \big).
\end{align*} 
The cut-off in the definition of $\hat G^T$, \eqref{lp}, and \eqref{E:hat-G-esti-1} imply 
\[
|\p_t y_j| = |\hat G^T| \le C \delta.
\]
From Lemma \ref{decomp2}, Lemma \ref{L:A_e}, \eqref{E:hat-G-esti-1}, \eqref{E:delta-1}, and  \eqref{E:delta-1}, we can estimate
\begin{align*}
&\Vert B^{cu} (t) W_2 (t) \Vert_{X_1, Q} = Q^2 \Vert B (t) V_2 (t) \Vert_{X_1}\\
\le & C \delta Q^2 \big(|y_2 - y_1| \big(1+ |\hat G^T (W_1, h(W_1))|\big) + |\hat G^T (W_2, \tilde a_2^-) - \hat G^T (W_1, h(W_1))|\big) \\
\le &C\delta Q^2 (\Vert W_2 - W_1 \Vert_{X^{cu}} +   | \tilde a_2^- -h(W_1)|).
\end{align*}

From the definitions of $W_j$, $j=1,2$, and decomposition \eqref{E:V-perp}, we have 
\begin{align*}
\p_t (W_2 - W_1) = &A^{cu} \big(y_1, \hat G^T(W_1, h(W_1)) \big)(W_2 - W_1) + B^{cu} W_2\\
& + \hat G^{cu} (W_2, \tilde a_2^- ) - \hat G^{cu} \big(W_1, h(W_1)\big).
\end{align*}

Applying Lemma \ref{L:A-cu} we obtain 
\begin{align*}
\Vert (W_2 - W_1)(t) \Vert_{X_1, Q}^2 \le &C \eta^{-2d_1} e^{-\eta t} \Vert (W_2 - W_1)(0) \Vert_{X_1, Q}^2 \\
&+ \eta^{-2d_1-1} \int_t^0 e^{\eta (\tau -t)} \big( \Vert (B^{cu} W_2) (\tau) \Vert_{\wt X, Q}^2 \\
&+ \Vert  \big(\hat G^{cu} (W_2, \tilde a_2^- ) - \hat G^{cu} \big(W_1, h(W_1)\big)\big)(\tau)  \Vert_{\wt X, Q}^2   \Big) d\tau.  
\end{align*}
The above estimate on $B^{cu}$ and Lemma \ref{L:G-Lip} including \eqref{E:hat-G-esti-2} 

imply 
\begin{align*}
\Vert (W_2 - &W_1)(t) \Vert_{X_1, Q}^2 \le C  \eta^{-2d_1} e^{-\eta t} \Vert (W_2 - W_1)(0) \Vert_{X_1, Q}^2 + \eta^{-2d_1-1}\int_t^0 e^{\eta (\tau -t)} \\
&\times \big( (Q^{-1} + \delta Q^3)^2 \Vert (W_2^e -W_1^e)(\tau)\Vert_{X_1, Q}^2 +  \delta^2 Q^6 | \big(\tilde a_2^- -h(W_1)\big) (\tau)|^2\big) d\tau.
\end{align*}
Since $h$ is Lipschitz with Lipschitz constant $\mu <1$, we obtain

\begin{align*}
e^{\eta t} \Vert  (&W_2 -W_1)(t) \Vert_{X_1, Q}^2 \le C   \eta^{-2d_1} \Vert (W_2 - W_1)(0) \Vert_{X_1, Q}^2 + \eta^{-2d_1-1}\int_t^0 e^{\eta \tau} \\
& \times \big( (Q^{-1} + \delta Q^3)^2 \Vert (W_2 -W_1)(\tau)\Vert_{X_1, Q}^2 + \delta^2 Q^6 | \big(\tilde a_2^- -h(W_2)\big) (\tau)|^2\big) d\tau.
\end{align*}

The estimate on $W_2 -W_1$ follows from the Gronwall inequality. 
\end{proof}

\begin{remark} \label{R:bump}
It is worth pointing out that the $\mathcal{F}$ term in the equation of $\p_t V$ can not be cut off as it ensures $V \in X_{c, y}^e$ if this holds initially. In the proof of the above proposition, this term was under control since it vanishes when $\Vert V \Vert_{X_1} = C \delta$ which implies $\p_t y=0$. Seemingly this argument heavily depends on the lack of growth of $e^{t A_e(y)}$ for any fixed $y$. If $e^{tA_e(y)}$ indeed induces some weak exponential growth backward in $t$, instead of the cut-off applied to the $V$ equation, a standard trick is to add a bump function to modify the $V$ equation so that it is actually slightly inflowing/decaying backward in $t$  for $\Vert V \Vert_{X_1} \ge C \delta$. The same estimates could be obtained subsequently.  
\end{remark}

\subsection{Lipschitz center-unstable manifold} \label{SS:LipWcs}

In this subsection, we will show that the transformation outlined in Subsection \ref{SS:GT} is well-defined and is a contraction on $\Gamma_{\mu, \delta}$ for appropriate $\mu, Q$, and $\delta$, which would imply the existence of a fixed point and thus a center-unstable manifold. For any $h \in \Gamma_{\mu, \delta}$, recall we attempted to define a new mapping $\wt h=\CT(h)$ whose value $\wt h(W) =\bar a^-$ at $W= (\bar y, \bar a^{d1}, \bar a^{d2}, \bar a^+, \bar V)$ is given by \eqref{newh}. 

\begin{lemma} \label{L:contraction-1} 
Fix $\eta \in (0, 1) \cap (0, \lambda)$. There exists $C>0$ independent of $Q, \mu, \delta$ and $\eta$, such that if \eqref{E:parameter-1},  \eqref{E:parameter-2}, and 
\begin{equation} \label{E:parameter-3}
C (\lambda-\eta)^{-1} \eta^{-(d_1+1)} Q^3  \delta^2 <1, \quad  C( \lambda -\eta)^{-1} \eta^{-d_1} \delta < \mu, 
\end{equation}
are satisfied, then $\mathcal{T}$ is a contraction on $\Gamma_{\mu, \delta}$. 
\end{lemma}

\begin{proof} 
We first prove that $\wt h \in \Gamma_{\mu, \delta}$ . Since $h(y, 0) =0$ and $\hat G^{cu, -}(y, 0)=0$, if its initial data $\bar W=(\bar y, 0)$,  then the solution to \eqref{lp} apparently is $W(t)=(\bar y, 0)$. Therefore $\wt h(\bar y, 0) =0$. From \eqref{E:hat-G-esti-2} and \eqref{E:exp-M}, it is easy to estimate 
\[
\Vert \wt h \Vert_{C^0} \le C \lambda^{-1} \delta^2 \le \delta/15,
\]
where \eqref{E:parameter-3} is used. 

From \eqref{E:trans-inv-G}, \eqref{E:trans-inv-h}, and that $\Vert \cdot \Vert_{X_1}$ is translation invariant, the cut-off is not affected by any spatial translation. Therefore $\big(y(t) + z, a^{d1}(t), a^{d2}(t), a^+(t), V(t, \cdot + z)\big)$ is a solution to \eqref{lp} for any $z\in \BR^3$ and any solution $(y, a^{d1}, a^{d2}, a^+, V)(t)$ to \eqref{lp}. Therefore the definition of $\wt h$ implies that it also satisfies \eqref{E:trans-inv-h}. 

Finally, we show the Lipschitz property for $\tilde h$. For any $\bar W_j \in X^{cu}(\delta)$, let $W_j(t) $, $t\le 0$, be the corresponding solutions to \eqref{lp} and $\bar a_j^-$ be given as in \eqref{newh}, for $j=1,2$. Applying Proposition \ref{P1} to these two solutions, we have, for any $t\le 0$, 
\begin{align*} 
& \Vert ( W_2 -  W_1)(t) \Vert_{X_1, Q} \le C \eta^{-d_1} e^{ - \eta t}  \Vert  \bar W_2 - \bar W_1\Vert_{X_1, Q}.
\end{align*}
Therefore, we obtain from \eqref{newh}, \eqref{E:exp-M}, and \eqref{E:hat-G-esti-2} that 
\begin{align*}
|\bar a_2^- - \bar a_1^-|
\le & C \delta \int_{-\infty}^0 e^{\lambda \tau} \Vert ( W_2 - W_1) (\tau) \Vert_{X_1, Q} d\tau 
\le  \frac {C \eta^{-d_1} \delta}{\lambda -\eta}  \Vert \bar W_2 - \bar W_1 \Vert_{X_1, Q}. 
\end{align*}
The desired Lipschitz property of $\wt h$ follows immediately from \eqref{E:parameter-3} and thus $\wt h \in \Gamma_{\mu, \delta}$.

To see the transformation $h \to \wt h$ is a contraction, given any $h_1, h_2 \in \Gamma_{\mu, \delta}$ and initial value $W \in X^{cu}(\delta)$, let $W_j(t) $, $t \le 0$, $j=1,2$, be the solutions to \eqref{lp} associated to $h_j$, with the initial value $W$. In applying Proposition \ref{P1}, we notice the corresponding 
\[
|(a_2 - h_1(W))(t)| \le \Vert h_2 - h_1 \Vert_{C^0}, \quad (W_1 - W_2)(0)=0,
\] 
and thus,  for any $t\le 0$,
\begin{align*} 
\Vert (W_2 - W_1) (t) \Vert_{X_1, Q}  \le & C \eta^{-d_1-1} Q^3 \delta e^{-\eta t}  \Vert h_2 - h_1 \Vert_{C^0}.
\end{align*}
Therefore \eqref{newh}, \eqref{E:exp-M}, and \eqref{E:hat-G-esti-2} again imply 
\[
|\bar a_2^- - \bar a_1^-| \le  C \eta^{-(1+d_1)} Q^3 \delta^2 \int_{-\infty}^0 e^{(\lambda - \eta)\tau} d\tau \Vert h_2 -h_1 \Vert_{C^0}. 
\]
Therefore \eqref{E:parameter-3} implies the contraction property. 
\end{proof}

Under conditions \eqref{E:parameter-1},  \eqref{E:parameter-2}, and \eqref{E:parameter-3}, which can apparently be satisfied by choosing $\mu, \delta, Q$, and $\eta$ carefully, Lemma \ref{L:contraction-1} implies 
\[
\exists | \, h^{cu} \in \Gamma_{\mu, \delta}, \; \text{ s. t. } \; \mathcal{T}(h^{cu}) = h^{cu}. 
\]

We are only concerned with $h^{cu}$ restricted to $\BR^{d_1+d_2 +d} \oplus \CX^e (\delta)$. Let 
\begin{align*}
W^{cu}& = graph(h^{cu})=  \big\{ (y, a^{d1}, a^{d2}, a^+,  a^-, V^e) \mid \\
&a^- = h^{cu} (y, a^{d1}, a^{d2}, a^+, V^e), \ (y, a^{d1}, a^{d2}, a^+, V^e) \in \BR^{d_1+d_2+d}\times \CX^e (\delta)\big\}
\end{align*}
and an even small submanifold for \eqref{tfeq} and \eqref{GP}
\begin{equation} \label{E:CW-cu}\begin{split}
\CW^{cu}(\CM) = \Phi\big( \big\{(y, a^{d1}, &a^{d2}, a^+,  a^-, V^e) \in W^{cu} \mid \\
&|a^{d1}|, |a^{d2}|, |a^+|, \Vert V^e\Vert_{X_1} <\delta/15
\big\} \big). 
\end{split}\end{equation}
$W^{cu}$ is a Lipschitz manifold due to the Lipschitz property of $h^{cu} \in \Gamma_{\mu, \delta}$. 

For any $\bar a^- = h^{cu} (\bar W)$, $\bar W= ( \bar y, \bar a^{d1}, \bar a^{d2}, \bar a^+, \bar V^e)\in \BR^{d_1+d_2 +d} \oplus \CX^e (\delta)$
corresponding to a point on $W^{cu}$, let $W(t) =(y, a^{d1}, a^{d2}, a^+, V^e)(t)$ be the corresponding solution of \eqref{lp} where $V^e(t)  \in X_{c, y(t)}^e (C\delta)$ due to Lemma \ref{L:reduce} and Proposition \ref{P1}. Let $a^-(t) = h^{cu} \big(W(t)\big)$ whenever $t$ satisfies $V^e(t)  \in X_{c, y(t)}^e (\delta)$. We can extend $a^-(t)$ to be a bounded function for $t \in \BR$. 
For any $|t_0|\ll1$ such that $\tilde W=W(t_0) \in X^{cu}(\delta)$, since \eqref{lp} is autonomous, $W(t+t_0)$ is its solution with initial value $\tilde W$. Since $\mathcal{T}(h^{cu}) = h^{cu}$, we can compute $\tilde a^-= a^-(t_0) = h^{cu}\big(\tilde W\big)$ satisfies 
\begin{align*}
\tilde a^- = &\int_{-\infty}^0 e^{-\tau M_-} \hat G^-\big((W, a^-) (t_0+\tau) \big) d\tau
= \int_{-\infty}^{t_0} e^{-(\tau - t_0) M_-} \hat G^-\big((W, a^-)(\tau) \big) d\tau \\
=& e^{t_0 M_-} \bar a^- + \int_0^{t_0}  e^{(t_0-\tau) M_-} \hat G^-\big((W, a^-) (\tau) \big) d\tau.
\end{align*}
As in Remark \ref{R:C-0}, the above equality does not depend on the extension of $a^-(t)$ as $\hat G^- (W, a^-) =0$ whenever $W \in X^{cu} \backslash X^{cu}(\delta)$. 
Therefore $(W, a^-)(t) = \big( W(t), h^{cu} (W(t))\big)$ is a solution to \eqref{E:cut-GP}. 
Along with the translation invariance \eqref{E:trans-inv-h} of $h^{cu}$, we have proved the local invariance of $W^{cu}$ under \eqref{E:cut-GP}. Since $h^{cu}$ is translation invariant, we obtain 

\begin{lemma} \label{L:invariance-1}
$W^{cu}$ is locally invariant under \eqref{E:cut-GP}, i.e. if $w(t)$ is a solution to \eqref{E:cut-GP} and $w(0) \in W^{cu}$, then exists $\epsilon>0$ such that $w(t) \in W^{cu}$ for all $t \in (-\epsilon, \epsilon)$. Moreover $W^{cu}$ satisfies, for any $z \in \BR^3$, 
\[
w(\cdot +z) \in W^{cu} \; \text{ if } \;  w\in W^{cu}. 
\] 
\end{lemma}

Solutions starting on $W^{cu}$ might leave $W^{cu}$ through its boundary $\overline{h^{cu}\big(X^{cu} (\delta)\big)} \backslash h^{cu}\big(X^{cu}(\delta)\big)$. 

Since \eqref{E:cut-GP} coincides with the original system \eqref{eqy'}, \eqref{E:a^pm}, \eqref{E:a^d1}, \eqref{E:a^d2}, and \eqref{E:VE} when 
\[
|a^{d1}| + |a^{d2}| + |a^+| + |a^-|+\Vert V^e \Vert_{X_1} \le  \delta/3, 
\]
$\CW^{cu}$ is a locally invariant manifold of \eqref{tfeq} and \eqref{GP}. Namely

\begin{proposition} \label{P:inv-1}
If  $U(t) = \Phi\big( w(t)\big)$ solves \eqref{tfeq}, satisfies $U(0)  \in \CW^{cu}$, and $w(t) \in B^{d_1+d_2+2d}(\frac \delta{15}) \oplus \CX^e (\frac \delta{15})$ 
for all $t\in [-T, T]$, $T>0$,  
then $U(t) \in \CW^{cu}$, $t\in [-T, T]$.
\end{proposition}

\subsection{Local dynamics related to the center-unstable manifold} \label{SS:dyna}

We start with the local stability of the center-unstable manifold, which means that if a solution to \eqref{tfeq} stays in a $\delta_0$-neighborhood of $\mathcal{M}$ over a time interval, then its distance to $\CW^{cu}$ shrinks  exponentially. Since \eqref{tfeq} is equivalent to \eqref{E:cut-GP} for $U$ near $\mathcal{M}$, we only need to work with \eqref{E:cut-GP}. More precisely,

\begin{lemma} \label{L:StaWcu}
There exists $C>0$ independent of $Q, \mu, \delta$ and $\eta$, such that if \eqref{E:parameter-1},  \eqref{E:parameter-2}, \eqref{E:parameter-3}, and 
\begin{equation} \label{E:parameter-4} 
C\big( \eta^{-(2d_1+1)} Q^6\delta^2 + \delta^2 \eta^{-1} +  \eta^{-2(d_1+1)} Q^6\delta^4 (\lambda -2\eta)^{-1} \big) < \eta
\end{equation}
are satisfied, then for any $T>0$ and solution $(W, a^-) (t) = (y, a^{d1}, a^{d2}, a^+, a^-, V^e)(t) \in \BR^{d_1 + d_2+2d} \oplus \CX^e(\delta)
$, $t \in [0, T]$, to \eqref{E:cut-GP}, we have  
\begin{align*}
&|a^-(t) - h^{cu} \big(W(t)\big) | 
\le  C e^{-(\lambda -2\eta) t}  |a^-(0) - h^{cu} \big(W(0)\big) |
\end{align*}
for any $t \in [0, T]$.
\end{lemma} 

\begin{proof} 
Let 
\begin{align*}
& W_1(t) = W(t), 
\quad \tilde a_1^- (t) = a^-(t), \quad \Delta^- (t) =  a^-(t) - h^{cu} \big(W(t)
\big). 
\end{align*}
Fix $t \in (0, T]$ and let $\big( (W_2, \tilde a_2^-)(\tau)\big)$ 
be the solution to \eqref{E:cut-GP} with initial value $\big(W(t), h^{cu}(W(t))\big)$ 
at $\tau=t$. The invariance of $W^{cu}$ under \eqref{E:cut-GP} implies that, for all $\tau\le t$,
\[\tilde a_2^- (\tau) =  h^{cu} \big( W_2
(\tau)\big), \quad \Delta^- (t) = \tilde a_1^-(t) - \tilde a_2^-(t),
\]
where note the latter holds at $\tau =t$ only. Denote 
\begin{align*}
l(\tau) = &  \Vert (W_2-W_1) (\tau) \Vert_{X_1, Q}.
\end{align*}
Lemma \ref{L:reduce} implies $\tilde \Pi^e W_j (\tau) = W_j(\tau)$ for $\tau \le t$, $j=1,2$, since it holds at $\tau =t$. From Proposition \ref{P1}, we have, for any $\tau \le t$, 
\begin{equation} \label{E:temp-3.1}
l(\tau)^2 \le C \eta^{-(2d_1+1)} Q^6\delta^2 \int_\tau^t e^{2\eta (\tau' -\tau)} |\Delta^- (\tau' )|^2 d\tau'.
\end{equation} 

Using the variation of parameter formula, we have   
\begin{align*} 
&\Delta^- (t) =  (\tilde a_1^- - \tilde a_2^-) (t) \\
=& e^{t M_-} (\tilde a_1^-- \tilde a_2^- ) (0) + \int_0^t e^{(t-\tau) M_-} \Big( \hat G^- \big( (W_1, \tilde a_1^-)(\tau),  \big) - \hat G^-\big((W_2, \tilde a_2^- ) (\tau)\big) \Big) d\tau 
\end{align*}
It follows from \eqref{E:exp-M} and \eqref{E:hat-G-esti-2} that  
\[
|\Delta^- (t)| \le C e^{- \lambda t} | (\tilde a_1^-- \tilde a_2^- ) (0)| + C\delta \int_0^t e^{- \lambda (t-\tau)} \big( l(\tau) + |(\tilde a_1^- - \tilde a_2^-) (\tau)|\big)  d\tau.
\]
Since 
\[
|(\tilde a_1^- - \tilde a_2^-) (\tau)| \le |\Delta^- (\tau)| + \mu l(\tau), 
\]
we obtain 
\[
|\Delta^- (t)| \le C e^{- \lambda t}\big( l(0) + | \Delta^- (0)|\big) + C\delta \int_0^t e^{- \lambda (t-\tau)} \big( l(\tau) + |\Delta^- (\tau)|\big)  d\tau.
\]
We use \eqref{E:temp-3.1} to proceed 
\begin{align*} 
|\Delta^- (t)|^2 \le& C e^{- 2(\lambda -\eta) t}| \Delta^- (0)|^2 + C\eta^{-(2d_1+1)} Q^6\delta^2  e^{- 2\lambda t}  \int_0^t e^{2\eta \tau} |\Delta^- (\tau)|^2 d\tau \\
& + C\delta^2 \big(\int_0^t e^{-\lambda (t-\tau)} |\Delta^- (\tau)|  d\tau\big)^2 + C\eta^{-(2d_1+1)} Q^6\delta^4 \\
& \times \Big( \int_0^t  e^{-\lambda( t-\tau)}  \big(\int_\tau^t e^{2\eta (\tau' -\tau)} |\Delta^- (\tau')|^2 d\tau'  \big)^{\frac 12} d\tau\Big)^2 \\
\triangleq &C e^{- 2(\lambda - \eta) t}| \Delta^- (0)|^2 + I_1 + I_2 + I_3.
\end{align*}
The above integrals are estimated by Cauchy-Schwartz inequality. Firstly, 
\begin{align*}
I_2 \le & C\delta^2 \int_0^t e^{-2\eta (t-\tau)} d\tau \int_0^t e^{-2(\lambda -\eta) (t-\tau)} |\Delta^- (\tau)|^2  d\tau \\
\le & C\delta^2 \eta^{-1} \int_0^t e^{-2(\lambda -\eta) (t-\tau)} |\Delta^- (\tau)|^2  d\tau.
\end{align*}
Secondly, 
\begin{align*}
I_3 \le & C\eta^{-(2d_1+1)} Q^6\delta^4 e^{-2 \lambda t}  \int_0^t e^{2\eta \tau} d\tau \int_0^t \int_\tau^t e^{2(\lambda -\eta) \tau + 2\eta(\tau' -\tau)} |\Delta^- (\tau')|^2  d\tau' d\tau \\
\le & C \eta^{-2(d_1+1)} Q^6\delta^4 e^{-2(\lambda -\eta) t} \int_0^t |\Delta^- (\tau')|^2 \int_0^{\tau'} e^{2\eta \tau' + 2(\lambda -2 \eta) \tau} d\tau d\tau' \\
\le &  C \eta^{-2(d_1+1)} Q^6\delta^4 (\lambda -2\eta)^{-1}\int_0^t e^{-2(\lambda -\eta) (t-\tau)} |\Delta^- (\tau)|^2 d\tau. 
\end{align*}
Finally, it is also easy to see 
\begin{align*}
I_1 \le & C\eta^{-(2d_1+1)} Q^6\delta^2 \int_0^t e^{-2(\lambda -\eta) (t-\tau)} |\Delta^- (\tau)|^2 d\tau.
\end{align*}
Therefore we obtain 
\[
|\Delta^- (t)|^2 \le C e^{- 2(\lambda -\eta) t}| \Delta^- (0)|^2 + \eta \int_0^t e^{-2(\lambda -\eta) (t-\tau)} |\Delta^- (\tau)|^2 d\tau
\]
where assumption \eqref{E:parameter-4} was used. The desired estimates follows immediately from the Gronwall inequality and the proof is complete. 
\end{proof}

\begin{remark} \label{R:Lip-CU}
The proof of the local asymptotic stability of the center-unstable manifold could have been much simpler if $h^{cu}$ had been smooth, which will be proved in the next section. In that case, one could obtain the decay estimate using certain property derived by differentiating the invariance equation of $h^{cu}$. In this subsection, even though we went a greater length to obtain the result, it has the benefit to show that the local asymptotic stability still holds even if $h^{cu}$ is only Lipschitz, which is the case when $G$ is only Lipschitz. 
\end{remark}

A direct corollary of Lemma \ref{L:StaWcu} is that the following condition for a point to belong to $W^{cu}$. 

\begin{lemma} \label{L:character-cu}
There exists $C>0$ such that if $\eta \in (C \delta, 1)$ and $Q, \mu, \delta$ satisfy \eqref{E:parameter-1}, \eqref{E:parameter-2}, \eqref{E:parameter-3}, and \eqref{E:parameter-4}, then a solution of \eqref{E:cut-GP} $(W, a^-) (t) = (y, a^{d1}, a^{d2}, a^+, a^-, V^e)(t) \in W^{cu}$ if $(W, a^-) (t) \in \BR^{d_1 + d_2+2d} \oplus \CX^e(\delta)
$ for all $t \le 0$ and satisfies $\sup_{t\le 0} |a^-(t)| <\infty$.
\end{lemma}

Since \eqref{GP}, or equivalently \eqref{tfeq}, is equivalent to \eqref{E:cut-GP} in a neighborhood of $\CM$, we have 

\begin{corollary} \label{C:character-cu}
If $U(t) = \Phi\big(w(t) \big)$ is a solution of \eqref{tfeq} satisfying $w(t) \in B^{d_1+d_2+2d}(\frac \delta{15}) \oplus \CX^e(\frac \delta{15})$ for all $t\le 0$, then $U(t) \in \CW^{cu}$, $t \le 0$. 
\end{corollary}

\begin{remark} \label{R:character-cu} 
Note that the assumption in the above lemma is satisfied if a solution stays in a neighborhood of $\CM$ for all $t \le 0$, which is the case of neighboring traveling waves.  
\end{remark}

More precisely, consider another travel wave $U_{\wt c} = (u_{\wt c}, v_{\wt c}) = \psi(w_{\wt c})$ with traveling velocity $\wt c \in \BR^3$. Here we recall $\psi$ and $\psi^{-1}$ is given in \eqref{E:psi} and \eqref{E:psi-inv}, respectively. 

\begin{lemma} \label{L:nTW}
There exists $\delta_0>0$ such that 
\begin{enumerate}
\item if $\wt c \wedge c =0$, $\Vert w_{\wt c} - w_c\Vert_{X_1} < \delta_0$, then $U_{\wt c} \in \CW^{cu}$;
\item or without $\wt c \wedge c=0$, but instead with the additional $\Vert x \nabla (w_{\wt c} - w_c \big)\Vert_{X_1} < \delta_0$ and that the angel between $c$ and $\wt c$ is bounded by $\delta_0$, then $u_{\wt c} \in \CW^{cu}$. 
\end{enumerate}
\end{lemma}

\begin{proof}
The solution to \eqref{w1w2} corresponding to $U_{\wt c}$ is given by $w_{\wt c}(x - \wt c t)$.  If $\wt c$ and $c$ are parallel, then $ w_c (\cdot - \wt c t) \in \psi^{-1}(\CM)$ and $\Vert w_{\wt c} (\cdot - \wt c t) - w_c (\cdot - \wt c t) \Vert_{X_1} < \delta_0$ for all $t \le 0$, then $U_{\wt c} \in \CW^{cu}$ by Corollary \ref{C:character-cu}.  

In the general case, there exists a near identity orthogonal matrix $O_{3\times 3}$ such that $O \wt c= |\wt c||c|^{-1} c$. It is easy to verify that $w_c\big(O(x - \wt ct)\big)$ is a traveling wave of \eqref{w1w2} with traveling velocity $|\wt c||c|^{-1} c$. $|O-I | \ll1$ yields that $w_c \big(Ox)$ is close to $w_c$ satisfying the assumption of case (1), therefore $U_c (Ox) \in \CW^{cu}$. Our assumptions imply $\Vert w_{\wt c} - w(O\cdot)\Vert_{X_1} \ll1$ and thus Corollary \ref{C:character-cu} implies $U_{\wt c} \in \CW^{cu}$. 
\end{proof}

\subsection{Construction of Local Center-Stable Manifolds}
Basically by reversing the time in the previous procedure, we can construct a local Lipschitz center-stable manifold $W^{cs}$ of $\CM$. It is given by the graph of a function $h^{cs} : B^{d_1+d_2+d} (\delta) \oplus \CX^e (\delta) \to \BR^d$, 
\begin{align*}
\CW^{cs}(\CM)  = \Phi\big( \big\{ &a^+ = h^{cs} (y, a^{d1}, a^{d2}, a^-, V^e) \mid \\
& (y, a^{d1}, a^{d2}, a^-, V^e) \in B^{d_1+d_2+d} (\delta) \oplus \CX^e (\delta)\big\} \big). 
\end{align*}

We briefly outline the steps here. Let $X^{cs} =\BR^{3+d_1+d_2+d} \times X_1$ be same as $X^{cu}$ and equipped with the same $\Vert \cdot \Vert_{X_1, Q}$ metric as in \eqref{E:Q-norm}. The set $\Gamma_{\mu, \delta}^{cs}$ of mappings $h: X^{cs} (\delta) \to \BR^d$ also takes the same form as $\Gamma^{cu}_{\mu, \delta}$. 

On $X^{cs} \times \BR^d$, we rewrite \eqref{E:cut-GP} as
\begin{subequations} \label{E:cut-GP-1}
\begin{equation} \label{E:cut-GP-cs}  
\p_t W = A^{cs}\big(y, \hat G^T (W, a^+) \big) W + \hat G^{cs}(W, a^+)
\end{equation} 
\begin{equation} \label{E:cut-GP-u}
\p_t a^+ = M_+ a^+ + \hat G^+ (W, a^+)
\end{equation}
\end{subequations}
where $\hat G^{cs} = (\hat G^T, \hat G^{d1}, \hat G^{d2}, \hat G^-, \hat G^e)$ are as defined in Subsection \ref{SS:GT} and 
\[
A^{cs} (y, \tilde y) = diag\big( 0, M_1, M_2, M_-, \Pi_{c, y}^e J L_{c, y} \Pi_{c, y}^e + \mathcal{F} (c, y) (\tilde y, \cdot ) \big).
\]

For any $h$ and $\bar W \in X^{cs}(\delta)$, let $W(t)$, $t\ge 0$, be the solution to 
\[
\p_t W = A^{cs} \big(y, \hat G^T(W, h(W)\big) W + \hat G^{cs} \big(W, h(W) \big), \qquad W(0)= \bar W. 
\]
and define $\widetilde{h}(\bar W)$ as 
\[
\widetilde{h} (\bar W) = \bar a^+ =- \int_0^\infty e^{-\tau M_+}\hat{G}^{+} \big(W(t), h(W(t))\big) dt
\]
and $\CT^{cs} (h) =\wt h$. Following exactly the same procedure, one proves that this defines a contraction mapping on $\Gamma_{\mu, \delta}^{cs}$, the graph of whose fixed point, restricted to $\BR^{d_1+d_2+d} (\frac \delta{15}) \oplus \CX^e (\frac \delta{15})$, leads to the locally invariant Lipschitz center stable manifold of $\CM$. 

\begin{proposition} \label{P:CS}
There exist $ \delta>0$ such that there exists $h^{cs} \in \Gamma_{\mu, \delta}^{cs}$ and  
\begin{enumerate}
\item the center-stable manifold 
\begin{align*}
&\CW^{cs}= \Phi\big(\{ (W, a^+) \in W^{cs} \mid W \in B^{d_1+d_2+d} (\delta/15) \oplus \CX^e (\delta/{15})\}\big)
\end{align*}
is locally invariant under \eqref{tfeq}, where  
\[
W^{cs} = graph(h^{cs})=\big\{ a^+ = h^{cs} (W
) \mid 
W \in \BR^{d_1+d_2+d}\oplus \CX^e (\delta)\big\}, 
\]
locally invariant under \eqref{E:cut-GP-1}. 
\item There exists $C>0$ independent of $Q, \mu, \delta$ and $\eta\in (0, 1)$, such that if \eqref{E:parameter-1},  \eqref{E:parameter-2}, \eqref{E:parameter-3}, and \eqref{E:parameter-4} are satisfied, then for any $T>0$ and any solution $U(t) = \Phi\big((W, a^+)(t)\big)$ to \eqref{E:cut-GP-1} with $(W, a^+)(t) \in \BR^{d_1 + d_2+2d} \oplus \CX^e(\delta/15)$, $t \in [0, T]$, we have  
\begin{align*}
&|a^+(t) - h^{cs} \big(W(t) 
\big) | 
\ge  C e^{(\lambda -2\eta) t}  |a^+(0) - h^{cs} \big(W(0) 
\big) |
\end{align*}
for any $t \in [0, T]$. 
\item A solution of \eqref{tfeq} $\Phi\big((W, a^+)(t)\big) \in \CW^{cs}$ for all $t \ge 0$ if $(W, a^+)(t) \in B^{d_1 + d_2+2d} (\delta/15) \oplus \CX^e(\delta/15)$ for all $t \ge 0$.
\end{enumerate}
\end{proposition} 

The estimate in part (2) on the growth of $a^+(t)$, $t>0$, for any solution follows directly from the decay estimate of $a^+(t)$, $t<0$, which is parallel to Lemma \ref{L:StaWcu} for $W^{cu}$ in the opposite time evolution direction.  

\begin{remark} \label{R:nTW-cs}
The local invariance of $\CW^{cs}$ is in the same sense as in Proposition \ref{P:inv-1}. 

Like $\CW^{cu}$, $\CW^{cs}$ is translation invariant in the sense as in Lemma \ref{L:invariance-1}, and $\CW^{cs}$ is Lipschitz. 
As in Lemma \ref{L:nTW}, all neighboring traveling waves belong to $W^{cs}$ under the same assumptions. 
\end{remark}

\begin{remark} \label{R:instability}
The above statement (2) implies that, if the initial value is not on the center-stable manifold, then the solution would eventually leave the $\frac \delta{15}$-neighborhood of $\CM$, and thus $\CM$ is orbitally unstable. 
\end{remark}

\subsection{Local Center Manifolds} 

A center manifold $\CW^c = \Phi(W^c)$ is given by the intersection of a center-unstable and a center-stable manifold, and thus it is also locally invariant and extends in the directions of the center subspace $X_{c, y}^T \oplus X_{c, y}^{d1} \oplus X_{c, y}^{d2} \oplus X_{c, y}^e$ at any $y$. For $0< \delta \ll1$, $(y, a^{d1}, a^{d2}, a^+, a^-, V^e) \in \BR^{d_1 + d_2+2d} \oplus \CX^e(\delta)$ belongs to $W^c$ if and only if 
\begin{equation}\label{fpcm}
{a}^{-}=h^{cu}(y,a^{d1}, a^{d2}, a^{+},V^{e}), \quad {a}^{+}=h^{cs}(y,a^{d1}, a^{d2}, a^{+},V^{e}).
\end{equation}
The 

Lipschitz property  implies that \eqref{fpcm} is equivalent to that $(a^+, a^-)$ is the fixed point of a contraction $(h^{cu}, h^{cs})$ with Lipschitz constant $\mu$. 
Therefore we obtain 

\begin{proposition} \label{P:CM}
There exists $C>0$ independent of $Q, \mu, \delta$ and $\eta\in (0, 1)$, such that if \eqref{E:parameter-1},  \eqref{E:parameter-2}, \eqref{E:parameter-3}, and \eqref{E:parameter-4}
are satisfied, then 
there exists $h^{c}: \BR^{d_1 + d_2} \oplus \CX^e(\delta) \to \BR^{2d}$ such that 
\begin{enumerate}
\item the center manifold 
\begin{align*}
\CW^{c}= \Phi\big(\{ &(y, a^{d1}, a^{d2}, a^+, a^-, V^e) \in W^{c} \mid \\
&(y, a^{d1}, a^{d2}, V^e) \in B^{d_1+d_2} (\delta/15) \oplus \CX^e (\delta/{15})\}\big)
\end{align*}
is locally invariant under \eqref{tfeq}, where  
\begin{align*}
W^c= graph(h^{c})= &graph(h^{cu}) \cap graph(h^{cs}) = \big\{  (y, a^{d1}, a^{d2}, a^+, a^-, V^e)\\
& \in \BR^{d_1+d_2+2d} \oplus \CX^e(\delta) \mid  (a^+, a^-) = h^{c} (y, a^{d1}, a^{d2}, V^e) \},
\end{align*}
locally invariant under \eqref{E:cut-GP-1} (and equivalently \eqref{E:cut-GP-1}). 
\item $h^c$ satisfies \eqref{E:trans-inv-h}, $h^c(y, 0, 0, 0)=0$, and  
\begin{align*}
&|h^c (y_2, a_2^{d1}, a_2^{d2}, V_2^e) - h^c (y_1, a_1^{d1}, a_1^{d2}, V_1^e)|\\
\le & \frac \mu{1-\mu} \big( |y_2 -y_1| + Q|a_2^{d1} - a_1^{d1}| + Q^3 |a_2^{d2} - a_1^{d2}| + Q^2 \Vert V_2^e - V_1^e \Vert_{X_1} \big).
\end{align*}
\item a solution $\Phi\big((y, a^{d1}, a^{d2}, V, a^+, a^-)(t)\big) \in \CW^{c}$ if $(y, a^{d1}, a^{d2}, a^+, a^-, V)(t) \in B^{d_1 + d_2+2d} (\delta/15) \oplus \CX^e(\delta/15)$ forall $t \in \BR$.
\item There exists $\delta>0$ such that any traveling wave solution satisfying assumptions in Lemma \ref{L:nTW} belongs to $\CW^c$. 
\end{enumerate}
\end{proposition}

The following lemma states that, as a submanifold, the center manifold attracts orbits on the center-unstable and center-stable manifolds.

\begin{lemma} \label{L:center-sta}
There exists $C>0$ independent of $Q, \mu, \delta$ and $\eta\in (0, 1)$, such that if \eqref{E:parameter-1},  \eqref{E:parameter-2}, \eqref{E:parameter-3}, and \eqref{E:parameter-4}
are satisfied, then for any $T>0$ the following hold. 
\begin{enumerate} 
\item Let $U(t) = \Phi\big((W, a^+, a^-)(t) \big)\in \CW^{cs}$ be a solution to \eqref{tfeq}, where $W=(y, a^{d1}, a^{d2}, V^e)$, satisfying  
$ (W, a^+, a^-)(t) \in B^{d_1 + d_2+2d} (\delta/15) \oplus \CX^e(\delta/15)$, $t \in [0, T]$, then we have  
\begin{align*}
&|(a^+, a^-)(t) - h^{c} \big(W(t) 
\big) | 
\le  C e^{-(\lambda -2\eta) t}  |(a^+, a^-)(0) - h^{c} \big(W(0) 
\big) | 
\end{align*}
for any $t \in [0, T]$. 
\item Let $U(t) = \Phi\big((W, a^+, a^-)(t) \big)\in \CW^{cu}$ be a solution to \eqref{tfeq}, where $W=(y, a^{d1}, a^{d2}, V^e)$, satisfying  
$ (W, a^+, a^-)(t) \in B^{d_1 + d_2+2d} (\delta/15) \oplus \CX^e(\delta/15)$, $t \in [-T, 0]$, then we have  
\begin{align*}
&|(a^+, a^-)(t) - h^{c} \big(W(t) 
\big) | 
\le  C e^{(\lambda -2\eta) t}  |(a^+, a^-)(0) - h^{c} \big(W(0) 
\big) | 
\end{align*}
for any $t \in [-T, 0]$. 
\end{enumerate}
\end{lemma}

\begin{proof}
Let us denote the components of $h^c$ by $h^c = (h_+^c, h_-^c)$. From \eqref{fpcm}, we have 
\begin{equation} \label{E:temp-5}
h_-^c (W) = h^{cu} \big(W, h_+^c (W)\big), \quad h_+^c (W) = h^{cu} \big(W, h_-^c (W)\big)
\end{equation}
We shall only prove part (1) of the lemma, where $a^+ = h^{cs} (W, a^-)$, as part (2) is verbatim. One may compute 
\begin{align*}
|a^- - h_-^c(W)| =& |a^- - h^{cu} \big(W, h_+^c (W)\big)| \le |a^- - h^{cu}(W, a^+)| + \mu |a^+ - h_+^c(W)| 
\end{align*}
and 
\[
|a^+ - h_+^c(W)| = |h^{cs} (W, a^-) - h^{cs}\big(W, h_-^c (W)\big) | \le \mu |a^- - h_-^c(W)|.
\]
Therefore
\[
|(a^+, a^-)(t) - h^{c} \big(W(t) \big) | \le (1-\mu)^{-1} |a^- - h^{cu}(W, a^+)|
\]
which along with Lemma \ref{L:StaWcu} implies the desired estimates. 
\end{proof}

\section{Smoothness of the center-unstable manifold} \label{S:smooth}

We will prove the smoothness of the local center-unstable/center-stable/center manifolds roughly following the approach in \cite{CLS92}. 

\begin{proposition} \label{P:smoothness} 
For any $k>0$, there exists $C>0$ such that if $\eta \in (C\delta, 1)$ and $Q, \mu, \delta$ satisfy \eqref{E:parameter-1}, \eqref{E:parameter-2}, \eqref{E:parameter-3}, \eqref{E:parameter-4},  \eqref{E:parameter-5}, and \eqref{E:parameter-6}, then $h^{cu}, h^{cs}, h^c \in C^k$ and $Dh^{cu}(y, 0, 0, 0, 0)$, $Dh^{cs}(y, 0, 0, 0, 0)$, and $Dh^{c}(y, 0, 0, 0, 0)$ are equal to 0. 
\end{proposition}

Unlike in  \cite{CLS92}, however, $A^{cu}$ in \eqref{E:cut-GP} depends on the unknowns and extra care has to be taken. Without loss of generality, we will work on $h^{cu} \in \Gamma_{\mu, \delta}$, which is defined on $X^{cu} (\delta)$, and the proof of $h^{cs}$ is verbatim. The smoothness of the center manifold, as the intersection of the center-unstable and center-stable manifolds,  follows subsequently.

\subsection{Outline of the framework of the smoothness proof} \label{SS:pre-smooth}

We first introduce some notations to simplify the presentations. Consider \eqref{lp} with $h=h^{cu}$. For $t \le 0$, let 
\[
\Psi(t, W)= (y, a^{d1}, a^{d2}, a^+, V)(t), \quad W = (y, a^{d1}, a^{d2}, a^+, V) \in X^{cu} (\delta), 
\]
be the solution with initial value $W$. We have from Lemma \ref{L:reduce} that 
\begin{equation} \label{E:reduce-1}
\tilde \Pi^e \Psi (t, W)  = \Psi(t, W), \quad \forall t\le 0 \; \; \text{ if } \; \tilde \Pi^e W = W. 
\end{equation}
Moreover, assuming \eqref{E:parameter-1} and \eqref{E:parameter-2}, Proposition \ref{P1} implies, for all $t \le 0$,  
\begin{equation} \label{E:Lip-Psi}
Lip_{\Vert \cdot \Vert_{X_1, Q}} \Psi(t, \cdot) \le 
C \eta^{-d_1} e^{-\eta t},
\quad \Psi(t, W) \in X^{cu} (C\delta), \; \forall W \in X^{cu}(\delta).
\end{equation}
As the fixed point of the transformation $\CT$, $h^{cu}$  satisfies
\begin{equation} \label{E:fix-p} 
h^{cu} (W) = \int_{-\infty}^0 e^{-tM_-} \hat G^-\Big(\Psi(t, W), h^{cu}\big(\Psi(t, W)\big) \Big) dt.  
\end{equation}
Since \eqref{lp} is autonomous, a time translation of \eqref{E:fix-p} implies, for $t\le 0$, 
\begin{equation} \label{E:fix-p-1} 
h^{cu}\big(\Psi(t, W)\big) = \int_{-\infty}^t e^{(t-\tau) M_-} \hat G^-\Big(\Psi(\tau, W), h^{cu}\big(\Psi(\tau, W)\big) \Big) d\tau.  
\end{equation}
By differentiating \eqref{E:fix-p} formally,  we obtain, for any $\wt W \in X^{cu}$, 
\begin{align*}
Dh^{cu} (W) \wt W =&  \int_{-\infty}^0 e^{-tM_-} \Big(D_{a^-}  \hat G^-\big(\Psi(t, W), h^{cu}\big(\Psi(t, W)\big) \big)Dh^{cu} \big(\Psi(t, W)\big) \\
&+ D_W \hat G^-\big(\Psi(t, W), h^{cu}\big(\Psi(t, W)\big) \big)\Big) D\Psi(t, W) \wt W dt. 
\end{align*}
Here $D\Psi$ also depends on $Dh^{cu}$ as it solves the following system of equation derived by differentiating \eqref{lp}
\begin{equation} \label{E:DPsi}
\p_t D\Psi = A^{cu} \big(y(t), \hat G^T \big) D\Psi + \CG_1 (\Psi) D\Psi + \wt \CG_1 (\Psi) Dh^{cu} D\Psi,  
\end{equation} 
where 
$\Psi$ and $D\Psi$ are evaluated at $(t, W)$, $\hat G^{cu}$ at $( \Psi, h^{cu})$, $h^{cu}$ and $D^{cu} h$ at $\Psi$. In the above $\CG_1 \in C^k \big(X^{cu},  L(X^{cu})\big)$ and $\wt \CG_1 \in C^k\big(X^{cu}, L(\BR^d, X^{cu})\big)$ are given by 
\begin{equation} \label{E:TCG-1} \begin{split}
\wt \CG_1 (W) \wt a^- =& D_{a^-} \big(A^{cu} \big(y, G^T(W, a^-) \big) \big)|_{a^- = h^{cu} (W)} (\wt a^-) W+ D_{a^-} \hat G^{cu} (\wt a^-)\\
=& \Big(0, 0, 0, 0, \CF\big(c,y\big) \big(D_{a^-} \hat G^T (\wt a^-), V \big) \Big)+ D_{a^-} \hat G^{cu} (\wt a^-),\\
\end{split} \end{equation}
\begin{equation} \label{E:CG-1} \begin{split}
\CG_1 (W) \wt W = & D_W \big(A^{cu} \big(y, G^T(W, a^-) \big) \big)|_{a^- = h^{cu} (W)} (\wt W) W+ D_W \hat G^{cu} (\wt W)\\
=& \Big(0, 0, 0, 0, \big( D_y A^e(y) \wt y\big) V+  \CF\big( c, y \big) \big( D_W \hat G^T (\wt W), V \big) \\ 
&\qquad \qquad \qquad + \big( D_y \CF(c, y) (\wt y)\big) ( \hat G^T, V)\Big)  + D_W \hat G^{cu} (\wt W) 
\end{split} \end{equation}
where $W =(y, a^{d1}, a^{d2}, a^-, V)$, $\wt W =(\wt y, \wt a^{d1}, \wt a^{d2},\wt a^-,\wt V) \in X^{cu}$ and $\hat G^{cu}$ are evaluated at $\big(W, h^{cu}(W)\big)$. 

Motivated by the above formally derived the equations, we define a linear transformation $\CT_1$ on 
\[
Y_1= C^0\big(X^{cu} (\delta), L(X^{cu}, \BR^d) \big) 
\]
as, for any $\CH \in Y_1$, $W\in X^{cu} (\delta)$, and $\wt W \in X^{cu}$, 
\begin{equation} \label{E:CT1} \begin{split}
(\CT_1 \CH) (W) \wt W =&   \int_{-\infty}^0 e^{-tM_-} \Big(D_W \hat G^-\big(\Psi, h^{cu} (\Psi) \big)\\
&+D_{a^-}  \hat G^-\big(\Psi, h^{cu}(\Psi) \big)\CH \big(\Psi\big) \Big) \Psi_1 (t) \wt W dt
\end{split} \end{equation}
where $\Psi$ is evaluated at $(t, W)$. Operator $\Psi_1 (t) \in L(X^{cu})$ satisfies $\Psi_1(0) =I$ and 
\begin{equation} \label{E:Psi-1}
\p_t \Psi_1 = A^{cu} \big(y(t), \hat G^T \big) \Psi_1 + \CG_1 (\Psi) \Psi_1 + \wt \CG_1 (\Psi) \CH (\Psi) \Psi_1,  
\end{equation}
where $\CG$ and $\CG_1$ are given in \eqref{E:CG-1}, $\hat G^{cu}$ is evaluated at $\big( \Psi, h^{cu} (\Psi)\big)$, and $\CH$ at $\Psi(t, W)$. Just as in Remark \ref{R:C-0}, the right side of \eqref{E:Psi-1} and the integrand in \eqref{E:CT1} are well-defined. Since \eqref{lp} is autonomous, when $W$ is shifted to $\Psi(t_0, W)$, the principle fundamental solution to the associated \eqref{E:Psi-1} becomes $\Psi_1(t+t_0) \Psi_1(t_0)^{-1}$. Therefore we obtain 
\begin{equation} \label{E:CT1-shift} \begin{split}
(\CT_1 \CH) \big( \Psi(t_0, W)\big) \Psi_1(t_0) & \wt W =   \int_{-\infty}^{t_0} e^{(t_0-t)M_-} \Big(D_W \hat G^-\big(\Psi, h^{cu} (\Psi) \big)\\
&+D_{a^-}  \hat G^-\big(\Psi, h^{cu}(\Psi) \big)\CH \big(\Psi\big) \Big) \Psi_1 (t) \wt W dt, 
\end{split} \end{equation}
where $\Psi$ is still evaluated at $(t, W)$ and $\Psi_1$ defined for $W$. 

If $h^{cu} \in C^1$, then $Dh^{cu}$ must be the fixed point of $\CT_1$. Therefore, our strategy to prove $h^{cu} \in C^1$ is to show 1.) $\CT_1$ is a well-defined contraction and 2.) the fixed point of $\CT_1$ is indeed $Dh^{cu}$ (Subsection \ref{SS:C1}). In the proof of the $C^k$ and higher order $C^k$ smoothness of $h^{cu}$ we shall need the following spaces $Y_k$, $k\ge 1$, of symmetric $k$-linear  transformations depending smoothly on the base points, 
\[
Y_k=\Big( C^0\big(X^{cu} (\delta), L(\otimes_{sym}^k (X^{cu}), \BR^d) \big) \Big), \quad k\ge 0.
\]
We equip $Y_k$ with the norm
\begin{equation} \label{E:norm-Yk} \begin{split}
&\Vert \CH \Vert_{Y_k} = \sup \{ \Vert \CH(W) \Vert_{L_Q^k} : W \in X^{cu} (\delta)\}, \\ 
&\Vert \CH(W) \Vert_{L_Q^k} = \sup\{ \frac {|\CH(W) (\wt W_1, \ldots, \wt W_k)|}{\Vert \wt W_1 \Vert_Q \ldots \Vert \wt W_k \Vert_Q}:  \wt W_1, \ldots, \wt W_k \in X^{cu} \backslash \{0\}\}.
\end{split} \end{equation}
We also use the $\Vert \cdot \Vert_{L_Q^k}$ norm of multilinear transformations in $L(\otimes_{sym}^k (X^{cu}), X^{cu})$ where $\Vert \cdot \Vert_{X_1, Q}$ is used in both the domain and the range.  

Formally differentiate \eqref{E:fix-p} twice, we see $D^2 h^{cu}$ is a fixed point of the following affine transformation $\CT_2$ on the space $Y_2$ of symmetric $k$-linear (with $k=2$) transformations depending continuously on the base points, 
\[
Y_k=\Big( C^0\big(X^{cu} (\delta), L(\otimes_{sym}^k (X^{cu}), \BR^d) \big) \Big)
\]
Here for any $\CH \in Y_2$, $W\in X^{cu} (\delta)$, and $\wt W_1, \wt W_2 \in X^{cu}$, 
\begin{equation} \label{E:CT2} \begin{split}
(\CT_2 \CH)& (W) (\wt W_1, \wt W_2)  =   \int_{-\infty}^0 e^{-tM_-} \Big( \Big(D_{a^-}  \hat G^- \CH (\Psi) + D_{WW} \hat G^- \Big) \\
& \big(  D\Psi \wt W_1,  D\Psi \wt W_2\big) + D_{a^-a^-} \hat G^- ( Dh^{cu} D\Psi \wt W_1, Dh^{cu}D \Psi \wt W_2) \\
& + 2 D_{Wa^-} \hat G^- (D\Psi \wt W_1, Dh^{cu} D\Psi \wt W_2) \\
&+ ( D_{a^-} \hat G^- Dh^{cu}+ D_W \hat G^-) \Psi_2(t) (\wt W_1, \wt W_2) \Big) dt, 
\end{split} \end{equation}
where $\Psi$ and $D\Psi$ are evaluated at $(t, W)$, $h^{cu}$ and $D h^{cu}$ at $\Psi$,  $\hat G^-$ and $D\hat G^-$ at $(\Psi, h^{cu})$, and the symmetric bilinear transformation $\Psi_2(t) \in L(\otimes_{sym}^2 X^{cu}, X^{cu})$ satisfies $\Psi_2(0)=0$ and 
\begin{equation} \label{E:Psi-2} \begin{split}
\p_t \Psi_2 = &\Big( A^{cu} \big(y(t), \hat G^T \big) +  \CG_1 (\Psi) + \wt \CG_1 (\Psi) D h^{cu} \Big) \Psi_2 \\
&+ \wt \CG_1 (\Psi) \CH (\Psi) (D\Psi, D\Psi) + \CG_2 (\Psi, D\Psi, Dh^{cu}). 
\end{split} \end{equation}
Here $\CG_2  (\Psi, D\Psi, Dh^{cu}) \in L(\otimes_{sym}^2 X^{cu}, X^{cu})$ is given by 
\begin{align*}
\CG_2  (\Psi, D\Psi, Dh^{cu}) (\wt W_1, \wt W_2) = &D_W \big( A^{cu} (y, G^T) +\CG_1(\Psi)\big) (\wt W_1) \big( D\Psi (\wt W_2)\big) \\
&+ D_W \big( \wt \CG_1(\Psi) \big)(\wt W_1) Dh^{cu} D\Psi(\wt W_2). 
\end{align*}

In order to prove $h^{cu}\in C^2$, we shall 1.) show that $\CT_2$ is a well-defined affine contraction and 2.) its fixed point is $D^2 h^{cu}$. 

The general $C^k$ smoothness of $h^{cu}$  (Subsection \ref{SS:Ck}) follows much as $h^{cu} \in C^2$ by 1.) differentiating \eqref{E:fix-p} repeatedly to obtain an affine operator on the space $Y_k$ of multilinear transformations, and 2.) proving its fixed point is indeed $D^k h^{cu}$. 

\begin{remark} 
A possible alternative adopted approach to prove the $C^k$, $k\ge 1$, smoothness of $h^{cu}$ is to prove that iteration sequences of the transformation $\CT$ defined in \eqref{lp} and \eqref{newh} actually converge in $C^k$ topology. That proof usually required the $C^{k, 1}$ bound on nonlinearity. Even though $\hat G$ is indeed smooth in our problem, this proof in Section \ref{S:smooth} shows that the $C^k$ smoothness holds as long as $\hat G \in C^k$. 
\end{remark}

\subsection{$C^1$ smoothness of $h^{cu}$} \label{SS:C1} 

We first prove the following estimate on equation \eqref{E:Psi-1} where $\Psi= \Psi(t, W)$, $W \in X^{cu}(\delta)$. 

\begin{lemma} \label{L:LPsi} 
There exists $C>0$ such that, if $\eta \in (C\delta, 1)$ and $\mu, \delta, Q$ satisfy \eqref{E:parameter-1} and  \eqref{E:parameter-2}, then for any $ B \in C^0\big([T, 0], L (X^{cu}, \BR^d)\big)$ with  $\Vert B\Vert_{ C^0_t L_Q^1} \le 1$, any solution $\wt W(t) \in X^{cu}$ of 
\[
\p_t \wt W = \big(A^{cu} \big(y(t), \hat G^T \big)  + \CG_1 (\Psi)  + \wt \CG_1 (\Psi) B \big) \wt W + f(t),  
\]
satsfies 
\[
\Vert \wt W(t) \Vert_{X_1, Q}^2 \le C \eta^{-2d_1} e^{-2\eta t}  \Vert W(0)\Vert_{X_1, Q}^2 +C \eta^{-2d_1-1} \int_t^0 e^{2\eta(\tau -t)} \Vert f(\tau) \Vert_{\wt X, Q}^2 d\tau.  
\]
\end{lemma}

\begin{proof}
From Lemma \ref{L:A_e}, \eqref{E:CF-def}, \eqref{E:hat-G-esti-1}, and \eqref{E:hat-G-esti-2}, we have, for any $W= (y, a^{d1}, a^{d2}, a^+, V)$, $\wt W= (y, a^{d1}, a^{d2}, a^+, V) \in X^{cu}$, and 
\begin{align}
& \Vert \wt \CG_1 (W) \wt a^- \Vert_{\wt X, Q} \le C\delta (1+ \Vert V \Vert _{X_1}) |\wt a^-|, \label{E:CG-1-e-1}\\
&\Vert \CG_1(W) \wt W\Vert_{\wt X, Q} \le C ( \Vert V \Vert_{X_1} + Q^{-1} + Q^3 \delta) \Vert \wt W \Vert_{X_1, Q}. \label{E:CG-1-e-2}
\end{align}
Lemma \ref{L:A-cu}, \eqref{E:Lip-Psi}, and the above inequalities imply that, for any $t\in [T, 0]$, 
\begin{align*}
\Vert \wt W (t) \Vert_{X_1, Q}^2 \le & C \eta^{-2d_1} e^{-\eta t} \Big(\Vert \wt W(0) \Vert_{X_1, Q}^2 \\
&+\eta^{-1} \int_t^0 e^{\eta \tau}  \big((Q^{-1} + Q^3 \delta)^2\Vert \wt W (\tau) \Vert_{X_1, Q}^2 + \Vert f (\tau)\Vert_{\wt X, Q}^2 \big) d\tau \Big). 
\end{align*}
The lemma follows from the Gronwall inequality. 
 \end{proof}

Recall the Lipschitz constant $\mu$ in the definition of $\Gamma_{\mu, \delta}$, which naturally should be an upper bound of $\Vert D h^{cu} \Vert_{L_Q^1}$. 

\begin{lemma} \label{L:CT-1}
There exists $C>0$ such that, if $\eta \in (C\delta, 1)$ and $\mu, \delta, Q$ satisfy \eqref{E:parameter-1},   \eqref{E:parameter-2}, \eqref{E:parameter-3}, and \eqref{E:parameter-4}, then $\CT_1$ defines a mapping on the closed $\mu$-ball $Y_1(\mu) = \{ \CH \in Y_1 : \Vert \CH\Vert_{Y_1} \le \mu\}$ with Lipschitz constant $C \delta \eta^{-d_1} (\lambda -2\eta)^{-1}$. 
\end{lemma}

\begin{proof} 
Let $\CH \in Y_1(\mu)$. Lemma \ref{L:LPsi} implies that, for any $W \in X^{cu} (\delta)$, $t\le 0$, the $\Psi_1(t)$ defined in  \eqref{E:Psi-1} satisfies 
\begin{equation} \label{E:Psi-1-1} 
\Vert \Psi_1 (t) \Vert_{L_Q (X^{cu})} \le C \eta^{-d_1} e^{-\eta t}. 
\end{equation}
Therefore \eqref{E:exp-M}, definition \eqref{E:CT1} of $\CH_1$, and \eqref{E:hat-G-esti-1} imply 
\begin{equation} \label{E:Psi-1-2} 
\Vert \CT_1(\CH) \Vert_{Y_1} \le C \delta\eta^{-d_1} \int_{-\infty}^0 e^{(\lambda-\eta)t}dt= C \delta\eta^{-d_1} (\lambda-\eta)^{-1} \le \mu
\end{equation}
due to \eqref{E:parameter-3}. To prove $\CT_1 (\CH) \in Y_1$, it remains to show $\CH (W)$ is continuous in $W$. In fact, the above estimate implies that $\CT_1^{(n)} (\CH) \to \CT_1(\CH)$ uniformly, where 
\begin{align*} 
\big(\CT_1^{(n)} ( \CH)\big) (W) \wt W =&   \int_{-n}^0 e^{-tM_-} \Big(D_W \hat G^-\big(\Psi, h^{cu} (\Psi) \big)\\
&+D_{a^-}  \hat G^-\big(\Psi, h^{cu}(\Psi) \big)\CH \big(\Psi\big) \Big) \Psi_1 (t) \wt W dt.
\end{align*} 
From the continuity of $D\hat G^{cu, -}$, it is easy to verify that $\big(\CT_1^{(n)} (\CH)\big) (W)$ is $C^0$ in $W$. Therefore $\CT_1(\CH)$ is also continuous and thus $\CT_1(\CH) \in Y_1 (\mu)$.  

In the following we estimate the Lipschitz constant of $\CT_1$. Let $\CH_j\in Y_1(\mu)$ and $\Psi_{1, j}(t)$ be defined in  \eqref{E:Psi-1} for $\CH_j$, $j=1,2$, which satisfy 
\begin{align*}
\p_t (\Psi_{1,2} - \Psi_{1,1}) = &\big( A^{cu} (y, G^T) - \CG_1(\Psi) - \wt \CG_1(\Psi) \CH_1 \big) (\Psi_{1,2} - \Psi_{1,1})\\
& + \wt \CG_1(\Psi) (\CH_2 - \CH_1)  (\Psi) \Psi_{1,2}
\end{align*}
and $(\Psi_{1,2} - \Psi_{1,1})(0)=0$. Using Lemma \ref{L:LPsi} and \eqref{E:TCG-1}, we obtain 
\[
\Vert (\Psi_{1,2} - \Psi_{1,1}) (t) \Vert_{L_Q (X^{cu})} \le C \eta^{-2d_1-\frac 12}\delta |t|^{\frac 12} e^{-\eta t} \Vert \CH_2 - \CH_1\Vert_{Y_1}.
\]
From the definition of $\CT_1$, we have, for any $W \in X^{cu}(\delta)$,  
\begin{align*}
\big(\CT_1(\CH_1) - \CT_1(\CH_2)\big) (W) =& \int_{-\infty}^0 e^{-tM_-} \big(D_{a^-} \hat G^- (\CH_2 - \CH_1) \Psi_{1,2} (t) \\
&+(D_W \hat G^- + D_{a^-} G^- \CH_1) (\Psi_{1,2} - \Psi_{1,1}) (t)  \big)  dt,   
\end{align*}
where $D\hat G$ is evaluated at $\big(\Psi, h^{cu} (\Psi)\big)$, $\CH_j$ at $\Psi$, and $\Psi$ at $(t, W)$. Using \eqref{E:exp-M}, \eqref{E:hat-G-esti-2}, and the above estimates on $\Psi_{1,j}$ and $\Psi_{1,2} - \Psi_{1,1}$, it follows that 
\begin{align*}
\Vert \CT_1(\CH_1) - \CT_1(\CH_2) \Vert_{Y_1} \le & C\delta \int_{-\infty}^0 e^{(\lambda -\eta) t} \eta^{-d_1}  (1 + \eta^{-d_1-\frac 12}\delta |t|^{\frac 12}) dt \Vert \CH_2 - \CH_1\Vert_{Y_1}\\
\le & C\delta \eta^{-d_1} (\lambda -2\eta)^{-1} \Vert \CH_2 - \CH_1\Vert_{Y_1}.
\end{align*}
The proof of the lemma is complete. 
\end{proof} 

Assume 
\begin{equation} \label{E:parameter-5}
C \delta \eta^{-d_1} (\lambda -2\eta)^{-1} <1,
\end{equation}
then $\CT_1$ is a contraction mapping on $Y_1(\mu)$. Let $\CH^{cu} \in Y_1(\mu)$ be the unique fixed point of $\CT_1$. In the rest of this subsection, we will prove 

\begin{lemma}  \label{L:C1}
There exists $C>0$ such that if $\eta \in (C\delta, 1)$ and $Q, \mu, \delta$ satisfy \eqref{E:parameter-1}, \eqref{E:parameter-2}, \eqref{E:parameter-3}, \eqref{E:parameter-4}, and \eqref{E:parameter-5}, then $h^{cu} \in C^1 (X^{cu}, \BR^d)$ and $D h^{cu} (W)= \CH^{cu} (W)$ for any $W \in X^{cu}$. 
\end{lemma}

\begin{proof}
Since $\CH(W)$ is continuous in $W$, it suffices to show $Dh^{cu} (W_0) \wt W = D\CH(W_0) \wt W$ at any fixed $W_0 \in X^{cu}(\delta)$ and $\wt W \in X^{cu}\backslash \{0\}$. 
Let $\Psi_1(t)$ be defined as in \eqref{E:Psi-1} associated to $\CH^{cu}$ and $W_0$ and 
\begin{align*}
&R_\Psi (t) = \Psi(t, W_0 +\wt W) - \Psi(t, W_0) - \Psi_1 (t) \wt W, \\
&R_h (t) = h^{cu} \big( \Psi(t, W_0 + \wt W) \big) - h^{cu} \big( \Psi(t, W_0)\big) - \CH^{cu} \big(\Psi(t, W_0)\big) \Psi_1(t) \wt W. 
\end{align*} 
According to \eqref{E:Lip-Psi}, \eqref{E:Psi-1-1}, $\Vert \CH^{cu} \Vert_{Y_1} \le \mu$, and the Lipschitz property of $h^{cu}$, $R_\psi$ and $R_h$ satisfy the rough estimates 
\begin{equation} \label{E:R-basic-1}
\Vert R_\Psi (t) \Vert_{X_1, Q}  + |R_h(t)|\le C \eta^{-d_1} e^{-\eta t} \Vert \wt W \Vert_{X_1, Q}, 
\end{equation}
for $t\le 0$. Our goal is to show  $\Vert R_\Psi(0) \Vert_{X_1, Q} \slash \Vert \wt W \Vert_{X_1, Q}  \to 0$ as $\Vert \wt W \Vert_{X_1, Q} \to 0$. 

To analyze $R_h$ and $R_\Psi$, denote \begin{align*} 
&W(s,t) = (1-s) \Psi(t, W_0) + s \Psi(t, W_0 + \wt W), \\ 
&a^- (s,t) = (1-s) h^{cu} \big(\Psi(t, W_0)\big)  + s h^{cu} \big( \Psi(t, W_0 + \wt W) \big).
\end{align*}
and for $\alpha \in \{T, d1, d2, \pm, V, cu\}$, let  
\begin{align*}
R^\alpha (t) =& \hat G^\alpha \big( W(1, t), a^-(1, t)\big) - \big[ \hat G^\alpha  + D_W \hat G^\alpha \big( W(1, t) - W(0, t)\big) \\
&+ D_{a^-}\hat G^\alpha \big( a^-(1,t) - a^-(0, t)  \big)  \big]
\end{align*}
where $\hat G^\alpha$ and $D\hat G^\alpha$ in the brackets $[\ldots]$ are evaluated at 
$\big( W(0, t), a^-(0, t)\big) =\Big( \Psi(t, W_0), h^{cu}\big(\Psi(t, W_0)\big)\Big)$. From \eqref{E:Psi-1-1}, we have 
\begin{equation}  \label{E:R_cu} 
\Vert R^{cu}(t)\Vert_{\wt X, Q} + |R^-(t)| \le 
r(t) \Vert \wt W \Vert_{X_1, Q}
\end{equation}
where $r(t)>0$ satisfies  
\begin{equation} \label{E:gen-bound}  
r(t) \le C 
\eta^{-d_1} e^{-\eta t}, \quad \Vert r \Vert_{C^0([t_1, t_2], \BR)} \to 0 \; \text{ as } \; \Vert \wt W \Vert_{X_1, Q} \to 0
\end{equation} 
for any $t_1 \le t_2 \le 0$ \footnote{Here we only need some uniform continuity of $D\hat G$, instead of $\hat G \in C^2$ or $C^{1,1}$.} 

From \eqref{E:fix-p} and $\CT_1(\CH^{cu}) = \CH^{cu}$, we have 
\[
R_h(0) = \int_{-\infty}^0 e^{-tM_-}\big( R^- (t) + D_W \hat G^- R_\Psi  (t)  + D_{a^-}\hat G^- R_h(t) \big) dt 
\]
Moreover, using  \eqref{E:fix-p} and \eqref{E:CT1-shift} instead, we obtain 
\begin{equation} \label{E:R_h}
R_h(t) = \int_{-\infty}^0 e^{-\tau M_-} ( R^-  + D_W \hat G^- R_\Psi   + D_{a^-}\hat G^- R_h)|_{\tau +t} d\tau,  \quad t\le 0, 
\end{equation} 
where again the above $D \hat G^-$ are evaluated at $\Big( \Psi(t+\tau, W_0), h^{cu}\big(\Psi(t+\tau, W_0)\big)\Big)$. 

From \eqref{lp}, $R_\Psi(t)$ satisfies $R_\Psi(0)=0$ and 
\begin{align*} 
\p_t R_\Psi& =  \CA_0^{cu} (t) R_\Psi + \CA_0^- (t) R_h + R^{cu} + D_W \hat G^{cu} R_\Psi  + D_{a^-}\hat G^{cu} R_h+ \int_0^1 (\CA_s^{cu} \\
&- \CA_0^{cu}) (t) \big( W(1,t) - W(0, t) \big)  
+ (\CA_s^- - \CA_0^-) (t) \big( a^-(1,t) - a^-(0, t)    
\big) 
ds 
\end{align*}
where $D\hat G^{cu}$ is evaluated at $\big( W(0, t), a^-(0, t)\big)$, operators $\CA_s^{cu}(t) \in L(X^{cu})$ and $\CA_s^-(t) \in L(\BR^d, X^{cu})$ are given by 
\begin{align*}
\CA_s^{cu}(t) \wt W =& A^{cu} \Big(y(s, t), G^T\big(W(s,t), a^-(s,t) \big)\Big) \wt W \\
&+ D_W \Big(A^{cu} \big(y, G^T (W, a^-)\big)\Big)|_{\big(W(s,t), a^-(s,t)\big)} (\wt W) W(s,t)  \\
\CA_s^-(t) \wt a^- = &D_{a^-} \Big(A^{cu} \big(y, G^T (W, a^-)\big)\Big)|_{\big(W(s, t), a^-(s, t)\big)} (\wt a^-) W(s,t)
\end{align*}
with $W(s,t)$ and $a^-(s,t)$ defined in the above and $y(s,t)$ being the $y$ component of $W(s,t)$ (so the $D_W$ also acts on the $y$ component in $A^{cu}$). 
Note $D A^{cu}$ acts only on the $V$ component of $\wt W$. From Lemma \ref{L:A_e},  \eqref{E:Lip-Psi}, \eqref{E:TCG-1}, \eqref{E:CG-1}, \eqref{E:CG-1-e-1}, \eqref{E:CG-1-e-2}, and \eqref{E:R_cu}, it is straight forward to obtain 
\begin{align*}
&\Vert \p_t R_\Psi -(A^{cu} + \CG_1) R_\psi  \Vert_{\wt X, Q}=\Vert \p_t R_\Psi - \CA_0^{cu} R_\psi - D_W \hat G^{cu} R_\Psi \Vert_{\wt X, Q} \\
\le &C \delta |R_h| + \Vert R^{cu}\Vert_{\wt X, Q} + r(t) \Vert \wt W\Vert_{X_1, Q} 
\le 
C \delta |R_h| + r(t) \Vert \wt W \Vert_{X_1, Q}
\end{align*}
where where $A^{cu}$ and $\CG_1$ are evaluated based on $\Psi(t, W_0)$ and $r(t)$ satisfies \eqref{E:gen-bound}. Lemma \ref{L:LPsi} implies  
\begin{equation} \label{E:R-Psi}
\Vert R_\Psi (t) \Vert_{X_1, Q}^2 \le C \delta^2 \eta^{-2d_1-1} \int_t^0 e^{2\eta(\tau -t)} |R_h(\tau)|^2 d\tau + r_1(t) \Vert \wt W\Vert_{X_1, Q}^2
\end{equation}
where $A^{cu}$ and $\CG_1$ are evaluated based on $\Psi(t, W_0)$ and $r_1(t)$ satisfies 
\begin{equation} \label{E:gen-bound-1} 
r_1(t) \le C 
\eta^{-4d_1 -2} (1+|t|) e^{-2\eta t}, \; \; \Vert r_1 \Vert_{C^0([t_1, t_2], \BR)} \to 0 \; \text{ as } \; \Vert \wt W \Vert_{X_1, Q} \to 0
\end{equation} 
for any $t_1 \le t_2 \le 0$. 

Finally, let  
\[
\wt R_h = \sup_{t \le 0} e^{2\eta t} \frac {|R_h(t)|}{ \Vert \wt W\Vert_{X_1, Q}}, \quad \wt R_\Psi = \sup_{t \le 0} e^{2\eta t} \frac {\Vert R_\Psi(t) \Vert_{X_1, Q}}{ \Vert \wt W\Vert_{X_1, Q}}.
\]
Inequality \eqref{E:R-basic-1} implies $\wt R_h, \wt R_\Psi < \infty$. We will prove $\wt R_h, \wt R_\Psi \to 0$ as $ \Vert \wt W\Vert_{X_1, Q} \to 0$. In fact, \eqref{E:R-Psi} and \eqref{E:R_h} along with \eqref{E:R_cu} and Lemma \ref{L:G-Lip} imply 
\[
\wt R_\Psi \le C \delta \eta^{-d_1-1}  \wt R_h + \sup_{t\le 0} r_1(t)^{\frac 12} e^{2 \eta t} 
\]
and
\begin{align*}
\wt R_h \le & C \int_{-\infty}^0\delta e^{(\lambda -2\eta) \tau} (\wt R_\Psi + \wt R_h) 
+ e^{\lambda \tau + 2 \eta t} r (t+\tau) d\tau\\ 
\le & C (\lambda-2\eta)^{-1} \big( 
\delta(\wt R_\Psi +  \wt R_h) + \sup_{\tau\le 0} r(\tau) e^{2\eta \tau} \big). 
\end{align*}
Therefore 
\[
\wt R_\Psi  + \wt R_h \le C \big(\sup_{t\le 0} r_1(t)^{\frac 12} e^{2 \eta t} + \sup_{\tau\le 0} r(\tau) e^{2\eta \tau} \big).
\]
From \eqref{E:gen-bound} and \eqref{E:gen-bound-1}, we obtain that $\wt R_h, \wt R_\Psi \to 0$ as $\Vert \wt W\Vert_{X_1, Q} \to 0$. Consequently $Dh^{cu} (W_0) = \CH(W_0)$ and $D \Psi(t, W_0) = \Psi_1(t, W_0)$. 
\end{proof}

Finally we prove that, at any traveling wave, the center-unstable manifold is tangent to the center-unstable subspace. 

\begin{lemma} \label{L:tangency}
There exists $C>0$ such that if $\eta \in (C\delta, 1)$ and $Q, \mu, \delta$ satisfy \eqref{E:parameter-1}, \eqref{E:parameter-2}, \eqref{E:parameter-3}, \eqref{E:parameter-4}, and \eqref{E:parameter-5}, then $Dh^{cu} (y, 0, 0, 0, 0) =0$ at any $y \in \BR^3$. 
\end{lemma}

\begin{proof} 
In the proof of this lemma, we adopt the notation $(y, 0) = (y, 0, 0, 0, 0) \in X^{cu}$. Observe that \eqref{lp} and the definition of $\hat G^{cu}$ implies $\Psi\big(t, (y, 0)\big) =(y, 0)$ for all $t\le 0$. 
For any $\CH \in Y_1$, \eqref{newh}, the fact $D \hat G^-(y, 0)=0$, and the above observation implies $\CT_1 (\CH) (y, 0) =0$. The conclusion of the lemma follows immediately. 
\end{proof}

\subsection{Higher order smoothness of $h^{cu}$} \label{SS:Ck} 

In this subsection, we shall prove 

\begin{proposition} \label{P:Ck} 
For any $k\ge 1$, there exists $C>0$ such that if $\eta \in (C\delta, 1)$ and $Q, \mu, \delta$ satisfy \eqref{E:parameter-1}, \eqref{E:parameter-2}, \eqref{E:parameter-3}, \eqref{E:parameter-4},  \eqref{E:parameter-5}, and 
\begin{equation} \label{E:parameter-6} 
C \delta \eta^{-kd_1} \big(\lambda- k\eta\big)^{-1} \le 1
\end{equation}
then $h^{cu} \in C^k$ and 
\[
\Vert D^k h^{cu} \Vert_{Y_k} + \sup_{t\le 0} e^{k\eta t} \Vert D^k \Psi(t, \cdot)\Vert_{C^0 L_Q^k}  < \infty.
\] 
\end{proposition} 

Here $D^k \Psi$ denote the differentiation with respect to $W$ only. In particular $\Vert Dh^{cu} \Vert_{Y_1} \le \mu$ and $\Vert D^kh^{cu} \Vert_{Y_k}$ may depend on $\delta$ for $k>1$. In the rest of this subsection, $C$ as usual denotes a generic upper bound independent of $t$, $W\in X^{cu}(\delta)$, and $\delta, Q, \mu$, while $\wt C$ independent of $t$ and $W\in X^{cu}(\delta)$, but may depend on $\delta, Q, \mu$. 

Formally differentiating \eqref{lp} and \eqref{newh} $k$ times implies that $D^k h^{cu}$ should be a fixed point of the following affine transformation $\CT_k$ on the space $Y_k$. Here for  $k\ge 2$, any $\CH \in Y_k$, $W\in X^{cu} (\delta)$, and $\wt W_1, \ldots, \wt W_k \in X^{cu}$, 
\begin{equation} \label{E:CTk} \begin{split}
& (\CT_k \CH) (W)  (\wt W_1, \ldots, \wt W_k)  \\
=&   \int_{-\infty}^0 e^{-tM_-} \Big( \big( ( D_{a^-} \hat G^- Dh^{cu}+ D_W \hat G^-) \Psi_k(t) +\CL_k(t)  \big) (\wt W_1, \ldots, \wt W_k) \\
& \qquad +D_{a^-}  \hat G^- \CH (\Psi) ( D\Psi \wt W_1, \ldots, D\Psi \wt W_k)\Big)  dt, 
\end{split} \end{equation}
where $D^l \Psi$ is evaluated at $(t, W)$, $D^l h^{cu}$ at $\Psi$, $D^l \hat G^l$ at $(\Psi, h^{cu})$, the symmetric multilinear mapping $\CL_k (t) \in Y_k$ is an algebraic combination involving $D^l \Psi$ and $D^l h^{cu}$, $D^k \hat G^-$, and $D^l \hat G^-$, $0\le l\le k-1$, and the symmetric multilinear $\Psi_k(t) \in L(\otimes_{sym}^k X^{cu}, X^{cu})$ satisfies $\Psi_k(0)=0$ and 
\begin{equation} \label{E:Psi-k} \begin{split}
\p_t \Psi_k = &\Big( A^{cu} \big(y(t), \hat G^T \big) +  \CG_1 (\Psi) + \wt \CG_1 (\Psi) D h^{cu} (\Psi) \Big) \Psi_k \\
&+ \wt \CG_1 (\Psi) \CH (\Psi) (D\Psi, \ldots, D\Psi) + \CG_k (t). 
\end{split} \end{equation}
Here $\CG_k(t) \in C^0\big( X^{cu}(\delta), L(\otimes_{sym}^k X^{cu}, X^{cu})\big)$ is again an algebraic combination involving $D^l \Psi$ and $D^l h^{cu}$, $D^k \hat G^{cu}$, and $D^l \hat G^{cu}$, $0\le l\le k-1$. These terms $\CG_k$ and $\CL_k$ are the lower order term in the higher order differentiation of compositions of mappings. The explicit forms of $\CT_2$, $\CG_2$, and $\CL_2$ can be found in \eqref{E:CT1} and \eqref{E:Psi-k}. 

The proof of Proposition \ref{P:Ck} is inductive in $k$. The case of $k=1$ has been proved in Subsection \ref{SS:C1}. Assume it holds for $1\le l< k$, we will prove it for $k$. As outlined in Subsection \ref{SS:pre-smooth}, we shall prove by showing that $D^k h^{cu}$ is given by the fixed point of the contraction $\CT_k$. 
Based on the usual formula of higher order derivatives of compositions of mappings, the induction assumptions imply 
\begin{equation} \label{E:temp-4} 
\sup_{t\le 0} e^{k\eta t} \Vert \CL_k (t) \Vert_{Y_k} + \sup_{t\le 0} e^{k\eta t} \Vert \CG_k (t) \Vert_{C^0L_Q^k} < \infty.  
\end{equation} 
In the following proof we will skip some details which are similar to those in Subsection \ref{SS:C1}.  

For $k \ge 2$, as $\CT_k$ is an affine transformation on $Y_k$, we first consider its homogeneous part $\CT_k \in L(Y_k)$
\begin{equation} \label{E:tCTk} \begin{split}
& (\wt \CT_k \CH) (W)  (\wt W_1, \ldots, \wt W_k)  \\
=&   \int_{-\infty}^0 e^{-tM_-}  \big( ( D_{a^-} \hat G^- Dh^{cu}+ D_W \hat G^-)\wt \Psi_k(t) (\wt W_1, \ldots, \wt W_k) \\
& \qquad +D_{a^-}  \hat G^- \CH (\Psi) ( D\Psi \wt W_1, \ldots, D\Psi \wt W_k)\big) dt, 
\end{split} \end{equation}
with the same convention of the notations and 
\begin{equation} \label{E:tPsi-k} \begin{split}
\p_t \wt \Psi_k = &\Big( A^{cu} \big(y(t), \hat G^T \big) +  \CG_1 (\Psi) + \wt \CG_1 (\Psi) D h^{cu} (\Psi) \Big) \wt \Psi_k \\
&+ \wt \CG_1 (\Psi) \CH(\Psi) (D\Psi, \ldots, D\Psi). 
\end{split} \end{equation}

\begin{lemma} \label{L:LContraction}
Let $k \ge 2$. There exists $C>0$ such that if $\eta \in (C\delta, 1)$ and $Q, \mu, \delta$ satisfy \eqref{E:parameter-1}, \eqref{E:parameter-2}, \eqref{E:parameter-3}, \eqref{E:parameter-4}, and \eqref{E:parameter-5}, then 
\[
\Vert \wt \CT_k \Vert_{L(Y_k)} \le C \delta \eta^{-kd_1} (\lambda- k\eta)^{-1}. 
\]
\end{lemma}

\begin{proof} 
Lemma \ref{L:LPsi} and \eqref{E:TCG-1} imply, for $t \le 0$,  
\begin{equation} \label{E:Psi-k-e}
\Vert \wt \Psi_k (t) \Vert_{L_Q^k} \le C \delta \eta^{-(k+1)d_1-1} e^{-k\eta t} \Vert \CH\Vert_{Y_k}. 
\end{equation}
Substituting it into \eqref{E:tCTk} yields the lemma. 
\end{proof}

\begin{lemma} \label{L:kContraction}
Let $k \ge 2$ and assume Proposition \ref{P:Ck} holds for each $l$, $0 \le l \le k$. There exists $C>0$ such that if $\eta \in (C\delta, 1)$ and $Q, \mu, \delta$ satisfy \eqref{E:parameter-1}, \eqref{E:parameter-2}, \eqref{E:parameter-3}, \eqref{E:parameter-4}, \eqref{E:parameter-5}, and \eqref{E:parameter-6}, then $\CT_k$ is a contraction on $Y_k$. Moreover, for any $\CH\in Y_k$ and $W \in X^{cu}(\delta)$, the $\Psi_k(t)$ defined in \eqref{E:Psi-k} satisfies 
\[
\sup_{t\le 0, W \in X^{cu}(\delta)} e^{k\eta t} \Vert \Psi_k (t) \Vert_{L_Q^k} <\infty. 
\]
\end{lemma}

\begin{proof}
Firstly, the $\Vert \cdot \Vert_{Y_k}$ bound of $\CT_k(0)$ can be easily obtained using \eqref{E:temp-4}, which along with Lemma \ref{L:LContraction} implies the $\Vert \cdot \Vert_{Y_k}$ bound of $\CT_k(\CH)$ for any $\CH \in Y_k$. The continuity of $\CT_k (\CH)(W)$ with respect to $W$ follows from the same argument as in the proof of Lemma \ref{L:CT-1}. Therefore $\CT_k(\CH) \in Y_k$ and thus \eqref{E:parameter-6} and Lemma \ref{L:LContraction} imply that $\CT_k$ is a contraction. 
\end{proof}

Recall that the Lipschitz property of $h^{cu}$ was used in the proof of $h^{cu} \in C^1$ in Subsection \ref{SS:C1}. Similarly, before we proceed to prove $D^k h^{cu}$ is equal to the fixed point of $\CT_k$ and thus $h^{cu} \in C^k$, we first take a step back to prove $Lip \ D^{k-1} h^{cu} < \infty$ using the above lemma. 

\begin{lemma} \label{L:C1,1}
Let $k \ge 2$ and assume Proposition \ref{P:Ck} holds for each $l$, $0 \le l \le k-1$. There exists $C>0$ such that if $\eta \in (C\delta, 1)$ and $Q, \mu, \delta$ satisfy \eqref{E:parameter-1}, \eqref{E:parameter-2}, \eqref{E:parameter-3}, \eqref{E:parameter-4}, \eqref{E:parameter-5}, and \eqref{E:parameter-6}, then \
\[
Lip \ D^{k-1} h^{cu} + \sup_{t\le 0} e^{k\eta t} Lip\ D^{k-1}\Psi(t, \cdot) <\infty. 
\]
\end{lemma} 

\begin{proof} 
From the induction assumption, $D^{k-1} h^{cu} \in Y_{k-1}$ and thus $\CT_{k-1} (D^{k-1} h^{cu} ) = D^{k-1} h^{cu}$. To prove the lemma, we shall show that, for some $C_{k-1}$ which might depend on $\delta, Q,\mu$, the closed subset 
\[
\wt Y_k =\{ \CH \in Y_{k-1}: Lip \ \CH \le C_{k-1}\}
\]
of $Y_{k-1}$ is invariant under $\CT_{k-1}$, which implies $D^{k-1} h^{cu} \in \wt Y_{k-1}$ and thus Lipschitz. 

Since $\CL_{k-1}$ and $\CG_{k-1}$, appearing in \eqref{E:CTk} and \eqref{E:Psi-k}, involve only $D^l \Psi$ and $D^l h^{cu}$, $D^{k-1} \hat G^-$, and $D^l \hat G^-$, $0\le l\le k-2$, the induction assumptions imply, for $t\le 0$, 
\begin{equation} \label{E:Lip-CG}
\Vert D_W \CL_{k-1} (t) \Vert_{Y_{k}} +  \Vert D_W \CG_{k-1} (t) \Vert_{C^0 L_Q^{k}} \le \wt C e^{-k\eta t}. 
\end{equation}
Due to the slightly different forms, one has to proceed separately in the cases of $k=2$ and $k>2$, even though the estimates in prove these cases are essentially the same. 

\vspace{.08in}\noindent {\it Case 1: $k-1\ge 2$.} 
Let $\CH\in \wt Y_{k-1}$ and $W_j \in X^{cu}(\delta)$, $j=0,1$,  let $ \Psi_{k-1}^j (t)=D^{k-1} \Psi(t, W_j)$, which are also the solutions to \eqref{E:Psi-k} where $\Psi$ is  evaluated at $(t, W_j)$. From Lemma \ref{L:LPsi}, \eqref{E:Psi-k-e}, \eqref{E:Lip-CG}, the induction assumptions, and the Lipschitz bound on $\CH$, it is straight forward to obtain the desired Lipschitz estimate on $D^{k-1} \Psi(t, \cdot)$ 
\[ 
\Vert \Psi_{k-1}^1-\Psi_{k-1}^0 \Vert_{L_Q^{k-1}} \le (\wt C + C\delta \eta^{-1-(k+1)d_1} C_{k-1} )  e^{-k\eta t} \Vert W_1 - W_0\Vert_{X_1, Q}.  
\]
Therefore, using \eqref{E:hat-G-esti-2}, \eqref{E:Lip-CG}, \eqref{E:parameter-2}, and the induction assumptions, we can estimate  \eqref{E:CTk} as 
\begin{align*}
&\Vert \CT_{k-1} (\CH) (W_0) - \CT_{k-1} (\CH) (W_1)\Vert_{L_Q^{k-1}} \\
\le & \int_{-\infty}^0 e^{(\lambda -k\eta) t} \big( C\delta \eta^{-k d_1} ( 1+ \eta^{-d_1-1} \delta)C_{k-1} + \wt C) dt \ \Vert W_1 - W_0\Vert_{X_1, Q}\\
\le & (\lambda -k\eta)^{-1} \big( C\delta \eta^{-k d_1} C_{k-1} + \wt C)  \ \Vert W_1 - W_0\Vert_{X_1, Q}.
\end{align*}
From \eqref{E:parameter-6}, there exists $C_{k-1}>0$ such that $\CT_{k-1} (\CH) \in \wt Y_k$ for any $\CH \in \wt Y_k$. 

\vspace{.08in}\noindent {\it Case 2: $k-1= 1$.} 
In this case, one considers \eqref{E:CT1} and \eqref{E:Psi-1} instead. The estimates are similar we omit the details. 
\end{proof}

Assume \eqref{E:parameter-6} and let $\CH_k \in Y_k$ be the fixed point of $\CT_k$. We will prove 

\begin{lemma} \label{L:Ck} 
There exists $C>0$ such that if $\eta \in (C\delta, 1)$ and $Q, \mu, \delta$ satisfy \eqref{E:parameter-1}, \eqref{E:parameter-2}, \eqref{E:parameter-3}, \eqref{E:parameter-4},  \eqref{E:parameter-5}, and \eqref{E:parameter-6}, then $D^k h^{cu} = \CH_k$. 
\end{lemma} 

\begin{proof}
As in the proof of Lemma \ref{L:C1}, for any fixed $W_0 \in X^{cu}(\delta)$ and $\wt W \in X^{cu}\backslash \{0\}$, let $\Psi_k(t) \in L(\otimes_{sym}^k X^{cu}, X^{cu})$ be defined as in \eqref{E:Psi-k} associated to $\CH_k$ and $W_0$ and 
\begin{align*}
R_\Psi (t) = & D^{k-1} \Psi(t, W_0 +\wt W) -D^{k-1} \Psi(t, W_0) - \Psi_k (t) (\wt W, \ldots )\\
R_h (t) =& \Big(D^{k-1} h^{cu} \big( \Psi(t, W_0 + \wt W) \big) - D^{k-1} h^{cu} \big( \Psi(t, W_0)\big)\Big) \big(D\Psi (\cdot), \ldots, D\Psi(\cdot)\big) \\
&- \CH_k \big(\Psi(t, W_0)\big) \big(D\Psi (\wt W), D\Psi (\cdot), \ldots, D\Psi(\cdot) \big),
\end{align*} 
where all above $D\Psi$ are evaluated at $(t, W_0)$. Note in the above $\Psi_k (t) (\wt W, \ldots )  \in L(\otimes_{sym}^{k-1} X^{cu}, X^{cu})$ and 
\[
\CH_k \big(\Psi(t, W_0)\big) \big(D\Psi (\wt W), D\Psi (\cdot), \ldots, D\Psi(\cdot)\big) \in L(\otimes_{sym}^{k-1} X^{cu}, \BR^d), 
\]
consistent with the other terms. According to \eqref{E:Lip-Psi} and Lemma \ref{L:C1,1}, $R_\psi$ and $R_h$ satisfy the rough estimates 
\begin{equation} \label{E:R-basic-k}
\Vert R_\Psi (t) \Vert_{L_Q^{k-1}}  + |R_h(t)|_{L_Q^{k-1}}\le \wt Ce^{-k \eta t}  \Vert \wt W \Vert_{X_1, Q}, 
\end{equation}
for $t\le 0$. Our goal is to show  $\Vert R_{\Psi, h}(0) \Vert_{X_1, Q} \slash \Vert \wt W \Vert_{X_1, Q}  \to 0$ as $\Vert \wt W \Vert_{X_1, Q} \to 0$. 

Using $\CT_{k-1} (D^{k-1} h^{cu}) = D^{k-1} h^{cu}$, Lemma \ref{L:C1,1}, and the induction assumptions, much as the derivation of \eqref{E:R_h} and \eqref{E:R-Psi}, we obtain 
\begin{equation}   \label{E:R_h-k} \begin{split}
R_h(t) = \int_{-\infty}^0 e^{-\tau M_-} \big(& (D_W \hat G^- + D_{a^-} \hat G^- D h^{cu}) R_\Psi \\
&+ D_{a^-} \hat G^- R_h 
+ R_1\big)|_{t+\tau} d\tau  
\end{split} \end{equation}
\begin{align*} 
\p_t R_\Psi = &  \Big( A^{cu} \big(y(t), \hat G^T \big) +  \CG_1 (\Psi) + \wt \CG_1 (\Psi) D h^{cu} (\Psi) \Big) R_\Psi + \wt \CG_1 (\Psi) R_h 
+ R_2 (t). 
\end{align*}
where $D \hat G^-$ and $D\hat G^{cu}$ are evaluated at $(\Psi, h^{cu})$, $h^{cu}$ at $\Psi$, and $\Psi$ at $(t, W_0)$, followed by the shift in the integral of $R_h$. Here the norms $r_j (t) = \Vert R_j(t) \Vert_{L_Q^{k-1}}$, $j=1,2$, of the remainder terms $R_1(t)$ and $R_2(t)$ satisfy 
\begin{equation} \label{E:gen-bound-2}
r_1(t) + r_2(t) \le \wt C e^{-k\eta t} \Vert \wt W \Vert_{X_1, Q} , \quad \lim_{\Vert \wt W \Vert_{X_1, Q} \to 0} \frac {\Vert r_1 + r_2\Vert_{C^0([t_1, t_2])}} {\Vert \wt W \Vert_{X_1, Q}} =0
\end{equation}
for any $t_1\le t_2 \le 0$. Lemma \ref{L:LPsi} and \eqref{E:Lip-Psi} imply  
\begin{equation} \label{E:R-Psi-k}
\Vert R_\Psi(t)\Vert_{L_Q^{k-1}}^2 \le C \delta^2 \eta^{-2d_1-1} \int_t^0 e^{2\eta( \tau - t)} \Vert R_h(\tau) \Vert_{L_Q^{k-1}}^2 d\tau + r_3 (t)
\end{equation}
where 
\begin{equation} \label{E:gen-bound-3} 
r_3(t) \le \wt C e^{-2k\eta t}  \Vert \wt W \Vert_{X_1, Q}^2 , \quad \lim_{\Vert \wt W \Vert_{X_1, Q} \to 0} \frac {\Vert r_3\Vert_{C^0([t_1, t_2])}} {\Vert \wt W \Vert_{X_1, Q}^2} =0
\end{equation}

Finally, let  
\[
\wt R_h = \sup_{t \le 0} e^{(k+1)\eta t} \frac {|R_h(t)|}{ \Vert \wt W\Vert_{X_1, Q}}, \quad \wt R_\Psi = \sup_{t \le 0} e^{(k+1)\eta t} \frac {\Vert R_\Psi(t) \Vert_{L_Q^{k-1}}}{ \Vert \wt W\Vert_{X_1, Q}}.
\]
Inequality \eqref{E:R-basic-k} implies $\wt R_h, \wt R_\Psi < \infty$. Inequalities \eqref{E:R-Psi-k} and \eqref{E:R_h-k} along with \eqref{E:gen-bound-2}, \eqref{E:gen-bound-3}, and Lemma \ref{L:G-Lip} imply 
\[
\wt R_\Psi \le C \delta \eta^{-d_1-1}  \wt R_h + \sup_{t\le 0} r_3(t)^{\frac 12} e^{(k+1) \eta t} 
\]
and
\begin{align*}
\wt R_h 
\le & C (\lambda-(k+1)\eta)^{-1} \big( 
\delta(\wt R_\Psi +  \wt R_h) + \sup_{\tau\le 0} r_1(\tau) e^{(k+1)\eta \tau} \big). 
\end{align*}
Therefore 
\[
\wt R_\Psi  + \wt R_h \le C \big(\sup_{t\le 0} r_3(t)^{\frac 12} e^{(k+1) \eta t} + \sup_{\tau\le 0} r_1(\tau) e^{(k+1)\eta \tau} \big).
\]
From \eqref{E:gen-bound-2} and \eqref{E:gen-bound-3}, we obtain that $\wt R_h, \wt R_\Psi \to 0$ as $\Vert \wt W\Vert_{X_1, Q} \to 0$. 
\end{proof}

In the last step of the above proof, we may define $\wt R_h$ and $\wt R_\Psi$ by using a weight $e^{a \eta t}$ with any $a>k$ and thus we do not have assume $\lambda > (k+1)\eta$ additionally.

\section{A non-degeneracy case} \label{S:non-deg}

In this section, we consider a traveling wave $U_c=u_c + iv_c$ satisfying the following non-degeneracy conditions. Recall $L_{c, y}$ and $L_y$ defined in \eqref{lcy}, its Morse index $n^-(L_c)$ in \eqref{E:n^-}, and the dimensions $d1, d2, d$ in Lemma \ref{decomp2}. Assume  
\begin{enumerate} 
\item [{\bf (H1)}] $\ker L_{c} = span\{ \p_{x_j} U_c \mid j=1,2,3\}$;
\item [{\bf (H2)}] $d=n^-(L_c)$. 
\end{enumerate}

\begin{remark} \label{R:non-deg}
Assumption {(H1)} is a linearized elliptic problem. Usually {\bf (H2)} is not easy to verify directly. A special situation is when $n^-(L_c)=1$, which is often  the case when $U_c$ is derived from the Mountain Pass or a constrained minimization process with 1 constraint. In this case, according to Theorem 2.3 and Proposition 2.2 in \cite{LZ17}, {\bf (H2)} is satisfied if $\langle L_c V, V\rangle >0$ for all $V \in \ker(JL_c)^2 \backslash \ker (JL_c)$.  More specifically, it was proved in \cite{LWZ} that, if $c_0 \in \BR^3$ and $U_{a c_0} (x)$ is a family of traveling waves depending on $a$ smoothly, then $\frac d{da} P(U_{ac_0}) <0$ along with $n^-(L_c)=1$ implies {\bf (H1)}. 
\end{remark}

Under these hypotheses, among the subspaces in the decomposition given in Lemma \ref{decomp2}, statement (2) there implies $X_{c, y}^{d1}= X_{c, y}^{d2}=\{0\}$ and thus, in the same notations, we have the following decomposition. 

\begin{lemma} \label{L:decomp3}
Assume {\bf (H), (H1-2)}, and  \eqref{E:TW-decay}, then for any $y \in \BR^3$, it holds that  
\begin{enumerate} 
\item $X= X_{c, y}^T \oplus X_{c, y}^e \oplus X_{c, y}^+ \oplus X_{c, y}^-$;
\item $JL_{c, y}$ and $L_{c, y}$ take the forms
\[
L_{c, y} \longleftrightarrow \begin{bmatrix} 0  & 0 & 0 & 0 \\ 0 &  L^e (y) & 0 & 0 \\ 0 &  0 & 0 & L_{+-}(y) \\ 0 & 0 &  L_{+-}(y)^* & 0\end{bmatrix}, 
\]
\[
JL_{c, y} \longleftrightarrow \begin{bmatrix} 0 &    A_{Te}(y) &0 &0\\  0 &  A_{e} (y)& 0 & 0\\ 0  &0 &  A_+ (y)& 0 \\ 0 &0 & 0 &  A_-(y) \end{bmatrix}.
\]
\end{enumerate}
\end{lemma}

Here the above blocks satisfy the same properties as in Lemma \ref{decomp2}. 

In this non-degenerate case, we shall carefully consider the energy-momentum functional $E+ c \cdot P$ invariant under \eqref{GP} and \eqref{tfeq}, where $E$ and $P$ are defined in \eqref{E:energy} and \eqref{E:momentum}. Let $\wt E(y, a^+, a^-, V^e)$ be defined as 
\[
\wt E_c = (E + c \cdot P)\circ \Phi \in C^\infty (\BR^{2d} \oplus \CX^e, \BR),
\]
where the coordinate mapping $\Phi$ is defined in \eqref{E:coord-1}, whose domain can also be extended to $\BR^{3+2d} \times X_1$. The smoothness of $\wt E$ follows from Lemma 2.2 and 2,3 of \cite{LWZ} along with Lemma \ref{decomp2}. Using \eqref{KY} and \eqref{lcy} it is straight forward to obtain the leading order expansion of $\wt E$ at $(y, 0, 0 , 0) = \Phi^{-1} \big(U_c (\cdot +y)\big)$ 
\begin{equation}\label{E:tE1} \begin{split}
\wt E_c\big(y, a^+, a^-, V^e \big) = \langle L_e(y)& V^e, V^e\rangle + 2 \langle \wt L_{+-}a^-, a^+\rangle\\
&+ O\big( (|a^+| + |a^-|+ \Vert V^e\Vert_{X_1})^3\big)
\end{split} \end{equation}
when $|a^+|$, $|a^-|$, and $\Vert V^e\Vert_{X_1}$ are small. Here $L_e(y)$ is given in Lemma \ref{decomp2}, uniformly positive, and translation invariant, i.e. 
\[
\langle L_e(0) V^e, V^e \rangle = \langle L_e(y) V^e(\cdot +y), V^e(\cdot +y) \rangle. 
\]
The $d \times d$ matrix $\wt L_{+-}$ is defined by 
\[
\langle \wt L_{+-}a^-, a^+\rangle = \langle L_{+-}(y) a^- \xi_{c}^- (\cdot +y), a^+ \xi_{c}^- (\cdot +y) \rangle.  
\]
where $\xi_c^- = (\xi_{c, 1}^-, \ldots, \xi_{c, d}^-)$ and $L_{+-}(y)$ are given in Lemma \ref{decomp2}. $\wt L_{+-}$ is independent of $y$ since $L_{c, y}$ and thus $L_{+-}(y)$ are translation invariant. 

Let $\CW^{cu, cs, c}$, $h^{cu, cs}$, $h^c = (h_+^c, h_-^c)$ be given in Section \ref{S:InMa} , whose smoothness are established in Section \ref{S:smooth}, and the parameters $Q, \mu, \delta, \eta$ satisfy \eqref{E:parameter-1}, \eqref{E:parameter-2}, \eqref{E:parameter-3}, \eqref{E:parameter-4}, and \eqref{E:parameter-6}. For any $\big(y, a^+= h^{cs} (y, a^-, V^e), a^-, V^e\big) \in \CW^{cs}$, since $Dh^{cs}(y, 0, 0, 0)=0$, we have 
\begin{equation} \label{E:tEcs}
|\wt E_c \big(y, h^{cs} (y, a^-, V^e), a^-, V^e\big) - \langle L_e(y) V^e, V^e\rangle| \le C_0 (|a^-|+ \Vert V^e\Vert_{X_1})^3
\end{equation}
for some $C_0>0$. Based on the expansion \eqref{E:tE1}, we can prove the exponential stability of $\CW^c$ inside $\CW^{cs}$. 

\begin{lemma} \label{L:StaWc1} 
There exits $C>1$ such that if $\eta \in (C\delta, 1)$ and $Q, \mu, \delta$ satisfy \eqref{E:parameter-1}, \eqref{E:parameter-2}, \eqref{E:parameter-3}, \eqref{E:parameter-4}, and \eqref{E:parameter-6}, then for any initial value $\bar W= \big(\bar y, \bar a^+= h^{cs} (\bar y, \bar a^-, \bar V^e), \bar a^-, \bar V^e\big) \in \CW^{cs}$ with $|\bar a^-|+\Vert \bar V^e\Vert_{X_1} < C^{-2}\delta$, its corresponding solution $W(t) = (y, a^+, a^-, V^e)(t)$ satisfy, for all $t \ge 0$,  
\begin{align}
&|a^-(t)| + \Vert V^e(t) \Vert_{X_1} < \delta/15, \quad a^+ (t) = h^{cs} \big((y, a^-, V^e)(t) \big), \notag \\ 
&|a^-(t) - h_-^c \big((y, V^e) (t) \big)| \le C e^{-(\lambda - 2\eta) t} |a^-(0) - h_-^c \big((y, V^e) (0) \big)| \label{E:StaEcs}.
\end{align}
\end{lemma}

\begin{proof}
The assumptions on $\bar W$, the conservation of $\wt E$, and \eqref{E:tEcs} imply 
\[
|\wt E\big(W(t)\big)| = |\wt E(\bar W)| \le C^{-2} \delta^2.  
\]
Let 
\[
T = \sup \{t>0 : |a^-(t')| + \Vert V^e(t') \Vert_{X_1} < \delta/15, \; \forall  t' \in [0, t)\} >0. 
\]
On $[0, T]$, Proposition \ref{P:CS} implies \eqref{E:StaEcs} holds, which along with $Dh^c(y, 0) =0$ implies 
\[
|a^-(t)| \le C \Vert V^e(t)\Vert_{X_1}^2 + C^{-1} e^{-(\lambda -2\eta) t} \delta, \quad t\in [0, T].
\]
Applying \eqref{E:tEcs} again, we obtain
\[
\Vert V^e(t) \Vert_{X_1}^2 \le C \big(\wt E\big( W(t)\big) + |a^-(t)|^3\big) 
\]
and thus $\Vert V^e(t) \Vert_{X_1}^2 \le C^{-1} \delta^2$. This along with the above inequality on $a^-(t)$ implies the $T=\infty$. From Propositions \ref{P:CS} and \ref{P:CM} and Lemma \ref{L:center-sta}, the rest of the lemma follows. 
\end{proof}

Following exactly the same arguments, we also obtain the exponential stability of $\CW^c$ backward in time inside $\CW^{cu}$. 

\begin{lemma} \label{L:StaWc2} 
There exits $C>1$ such that if $\eta \in (C\delta, 1)$ and $Q, \mu, \delta$ satisfy \eqref{E:parameter-1}, \eqref{E:parameter-2}, \eqref{E:parameter-3}, \eqref{E:parameter-4}, and \eqref{E:parameter-6}, then for any initial value $\bar W= \big(\bar y, \bar a^+, \bar a^-= h^{cu} (\bar y, \bar a^+, \bar V^e), \bar V^e\big) \in \CW^{cu}$ with $|\bar a^+|+\Vert \bar V^e\Vert_{X_1} < C^{-2}\delta$, its corresponding solution $W(t) = (y, a^+, a^-, V^e)(t)$ satisfy, for all $t \le 0$,  
\begin{align}
& |a^+(t)| + \Vert V^e(t) \Vert_{X_1} < \delta/15, \quad a^- (t) = h^{cu} \big((y, a^+, V^e)(t) \big), \notag \\ 
&|a^+(t) - h_+^c \big((y, V^e) (t) \big)| \le C e^{(\lambda - 2\eta) t} |a^+(0) - h_+^c \big((y, V^e) (0) \big)| \label{E:StaEcu}.
\end{align}
\end{lemma}

Consequently, we also obtain the stability of $\CM$ in $\CW^c$. 

\begin{proposition} \label{P:StaWc} 
There exist $C>1$ and $\delta>0$ such that, for any initial value $\bar W= \big(\bar y, (\bar a^+, \bar a^-)= h^{c} (\bar y, \bar V^e), \bar V^e\big) \in \CW^{c}$ with $\Vert \bar V^e\Vert_{X_1} < C^{-2}\delta$, its corresponding solution $W(t) = (y, a^+, a^-, V^e)(t)$ satisfy, for all $t \in \BR$, 
\[
\Vert V^e(t) \Vert_{X_1} < \delta/15, \quad (a^+, a^-)(t) = h^c \big( (y, V^e)(t) \big). 
\] 
\end{proposition}

Combine the above results and Corollary \ref{C:character-cu}, Propositions \ref{P:CS} and \ref{P:CM}, we obtain the following characterization of $\CW^{cu}$, $\CW^{cs}$, and $\CW^c$. 

\begin{proposition} \label{P:character}
There exist $C>1$ and $\delta>0$ such that the following hold. Let $U(t) = \Phi\big(W(t)\big)$, where $W(t)= (y, a^+, a^-, V^e)(t)$, be a solutions to \eqref{tfeq} with initial value 
\[
\bar W= \big(\bar y, \bar a^+, \bar a^-, \bar V^e\big)\in B^{2d} (C^{-2} \delta) \oplus \CX^e(C^{-2} \delta),
\]
then 
\begin{enumerate}
\item $\bar W \in \CW^{cu}$ and thus $W(t) \in \CW^{cu}$ for all $t\le 0$,  if and only if $W(t) \in B^{2d} (\delta/15) \oplus \CX^e(\delta/15)$ for all $t\le 0$. 
\item $\bar W \in \CW^{cs}$ and thus $W(t) \in \CW^{cu}$ for all $t\ge 0$,  if and only if $W(t) \in B^{2d} (\delta/15) \oplus \CX^e(\delta/15)$ for all $t\ge 0$. 
\item $\bar W \in \CW^{c}$ and thus $W(t) \in \CW^{c}$ for all $t\in \BR$,  if and only if $W(t) \in B^{2d} (\delta/15) \oplus \CX^e(\delta/15)$ for all $t\in \BR$. 
\end{enumerate}
\end{proposition}

\begin{remark}
Note that when we construct the local invariant manifolds, we cut off the nonlinearities to focus on the local dynamics. Different choice of the cut-off could yield different local invariant manifolds. Therefore local center-stable, center-unstable, and center manifolds are usually not unique. However, under the non-degeneracy conditions {\bf (H1-2)}, we obtain the above characterization of the local invariant manifolds which is independent of the cut-off. Therefore the local manifolds are unique in this case. 
\end{remark}

\begin{appendix} 
\section{}

In the Appendix, we give some estimates of the nonlinear term $G$ in \eqref{w1w2}. One may compute in \eqref{E:w2}
\begin{equation} \label{E:G_2} \begin{split}
G_2(c, y, w) =&- \left( |U|^2 - |U_c(\cdot +y)|^2 - 2 U_c (\cdot +y) \cdot (K_{c, y} w) \right) u_c (\cdot+ y) \\
& - \left( |U|^2 - |U_c(\cdot +y)|^2 \right) \left( w_{1}-\chi(D)\left(v_{c}(\cdot+y)w_{2}\right)\right) \\
&- \frac 12 \left( (1-|U|^2) \chi(D) (w_2^2) + \Delta \chi(D) (w_2^2)   \right),    
\end{split}\end{equation}
where $U$ is given in \eqref{UnearM} as 
\begin{equation} \label{E:UnearM}
U = \psi\left(w_{c}\left(\cdot +y\right)+w\right) =  U_{c}\left(\cdot +y\right)+ K_{y, c} w - \left( \frac 12 \chi(D) (w_2^2), 0 \right)^T. 
\end{equation}

Substituting this into $G_1$ we obtain 
\[
G_1=G_1(c, y, \p_t y, w) = G_{11} (c, y, w) + G_{12} (c, y, \p_t y, w) 
\] 
where 
\begin{equation} \label{E:G_11} \begin{split}
G_{11} = & \left( |U|^2 - |U_c(\cdot +y)|^2 - 2 U_c (\cdot +y) \cdot (K_{c, y} w) \right) v_c (\cdot+ y) \\
& + \left( |U|^2 - |U_c(\cdot +y)|^2 \right) w_{2}  -\frac 12 \chi(D) (w_{2}\nabla w_{2}\cdot c)
\end{split} \end{equation}
and by substituting \eqref{E:w2} into $G_1$ 
\begin{equation} \label{E:G_12} \begin{split}
G_{12} = G_{12} (c, y, \tilde y, w) =&  \chi(D)\Big(w_2 \big(-(L_{c,y}K_{c,y}w)_1 - \tilde y \cdot \nabla v_c(\cdot+y)  \\
&+G_2(c,y,w)\big)\Big). 
\end{split} \end{equation}
Here $(L_{c,y}K_{c,y}w)_1$ denotes the first component of $L_{c,y}K_{c,y}w$. For fixed $c$ and $y$, $G$ is a polynomial of $w$ and $\tilde y$. More precisely, it is the sum of some multi-linear transformations on $w$ and $\tilde y$ of degree between 2 and $6$.

\begin{lemma}\label{ESG}
Fix $c$. It holds that 
\[
G(c, \cdot, \cdot, \cdot) \in C^\infty ( \BR^3 \times \BR^3 \times X_1, W^{1, \frac 32}) +  C^\infty ( \BR^3 \times X_1, L^{\frac 32} \cap \dot W^{1, \frac 65}), 
\]
and 
\[
G(c, y, 0) =0, \quad D_w G(c, y, 0) =0. 
\]
\end{lemma}

In particular, the only term $G_{12}$ containing $\tilde y$ belongs to $C^\infty ( \BR^3 \times \BR^3 \times X_1, W^{1, \frac 32})$. More refined estimates on $G$ can be found in \eqref{E:G_2-esti-1}, \eqref{E:G_2-esti-2}, \eqref{E:G_11-esti-1}, \eqref{E:G_11-esti-2}, \eqref{E:G_12-esti-1}, and \eqref{E:G_12-esti-2}, where the generic constant $C$ in those inequalities are independent of $y$. 

\begin{proof}
Due to the polynomial form of $G$ in $w \in X_1$ and $\tilde y \in \BR^3$, we only need to estimate the boundedness of each monomial, i.e. multi-linear transformation. To handle the terms with $\chi(D)$, we will  repeatedly use 
\begin{equation} \label{E:chi}
\Vert |\nabla |^s \chi(D) f \Vert_{L^p \cap L^\infty} \le C_{s, p} \Vert f\Vert_{L^p}, \quad \forall k \ge 0, \; 1\le p \le \infty. 
\end{equation}

We start with the consideration on $\vert U\vert ^{2}-\vert U_{c}(\cdot + y)\vert ^{2}$. Let 
\[
\rho=\frac 12 \chi(D)\left(w_{2}^{2}+2w_{2}w_{2c}(\cdot+y)\right) \Longrightarrow D_y^k \rho = \chi(D)\left(w_{2}D^k w_{2c}(\cdot+y)\right)
\] 
for any $k\ge 1$ and 
\[
\nabla \rho = \chi(D)\left(w_{2}\nabla w_2+\nabla w_{2}w_{2c}(\cdot+y)+w_{2}\nabla w_{2c}(\cdot+y)\right). 
\]

Using $w_{2c} = v_c \in \dot H^1$, \eqref{E:chi} implies, for any $s+ k \ge 1$  
\begin{equation} \label{E:rho-E} 
\Vert \nabla^s D_y^k \rho \Vert_{L^{\frac 32} \cap L^\infty} \le C_{s,k} \Vert w_2\Vert_{\dot H^1} (\Vert w_2\Vert_{\dot H^1}+1)
\end{equation}
where $C_{s, k}$ is independent of $y$. Here we used the property $D^k w_{2c} = D^k v_c \in L^2\cap L^\infty$ for all $k\ge 1$ due to equation \eqref{twGP} and also the embedding $\dot W^{1, \frac 32} (\BR^3) \to L^3(\BR^3)$ on $\rho$. The second term in $\rho$ actually has a better spatial decay estimate by using the Hardy's inequality and \eqref{E:TW-decay}. Namely, 
\[
\Vert w_{2}w_{2c}(\cdot+y)\Vert_{L^2} = \Vert w_{2} (\cdot -y) v_{c}\Vert_{L^2} \le C \Vert \frac {w_{2} (\cdot -y)}{|x|} \Vert_{L^2} \le C \Vert w_2\Vert_{\dot H^1}.
 \]
On some occasions, we also need to consider $\big(I-\chi(D)\big)(f_1f_2)$, for $f_{1,2} \in \dot H^1$. Since $\nabla (f_1f_2) \in L^{\frac 32}$, we have 
\begin{align*}
&\Vert \big(I-\chi(D)\big)(f_1f_2) \Vert_{L^{\frac 32}} =\Vert \frac {I-\chi(D)}{|D|}|D|(f_1f_2) \Vert_{L^{\frac 32}} \\
\le& C \Vert \nabla (f_1f_2) \Vert_{L^{\frac 32}} \le C \Vert f_1 \Vert_{\dot H^1} \Vert f_2 \Vert_{\dot H^1}
\end{align*}
where we used that the inverse Fourier transform of $\frac {1-\chi(\xi)}{|\xi|}$ is in $L^1$. It implies 
\begin{equation} \label{E:temp-A-1} 
\Vert \big(I-\chi(D)\big)(f_1f_2) \Vert_{W^{1, \frac 32}}  \le C \Vert f_1 \Vert_{\dot H^1} \Vert f_2 \Vert_{\dot H^1}. 
\end{equation}
One may compute that 
\begin{align*}
\vert U\vert ^{2}-\vert U_{c}\vert ^{2}= & w^{2}_{1}+2u_{c} (\cdot +y) (w_{1} -\rho) +\rho^{2}-2w_{1}\rho+2v_{c}(\cdot +y)w_{2}+w_{2}^{2}\\
=&  w^{2}_{1}+ 2 w_1+2\big(u_{c} (\cdot +y)-1\big) (w_{1} -\rho) +\rho^{2}-2w_{1}\rho\\
&+2\big(I-\chi(D)\big)\big(v_{c}(\cdot +y)w_{2} + w_2^2\big). 
\end{align*}
By using the above inequalities and \eqref{E:TW-decay}, we obtain, through straight forward calculations, for $k\ge 1$,  
\begin{equation} \label{E:temp-A-2} \begin{split}
&  \Vert \vert U\vert ^{2}-\vert U_{c}\vert^{2} - 2 w_1\Vert_{W^{1,\frac 32}}
+\Vert D_y^k (\vert U\vert ^{2}-\vert U_{c}\vert ^{2}) \Vert_{W^{1,\frac 32} \cap \dot H^1}
\\
\le& C \Vert w\Vert_{X_1} (\Vert w\Vert_{X_1}^3+1). 
\end{split}\end{equation}
Similarly, 
\begin{align*}
&\vert U\vert ^{2}-\vert U_{c}\vert ^{2} -2 U_c (\cdot +y) \cdot (K_{c, y} w) \\
=& w^{2}_{1} +\big(1- \chi(D)\big) (w_2^2) +\big(1-u_{c} (\cdot +y)\big) \chi(D) (w_2^2) +\rho^{2}-2w_{1}\rho
\end{align*}
and along with the above inequalities, it implies for $k\ge 1$
\begin{equation} \label{E:temp-A-3} \begin{split}
& \Vert D_y^k \big( \vert U\vert ^{2}-\vert U_{c}\vert ^{2} -2 U_c (\cdot +y) \cdot (K_{c, y} w) \big)\Vert_{W^{1,\frac 65}\cap \dot H^1}\\
& \quad +\Vert  \vert U\vert ^{2}-\vert U_{c}\vert ^{2} -2 U_c (\cdot +y) \cdot (K_{c, y} w)\Vert_{W^{1,\frac 32}} 
\le C \Vert w\Vert_{X_1}^2 (\Vert w\Vert_{X_1}^2+1). 
\end{split}\end{equation}
Substituting \eqref{E:temp-A-2} and \eqref{E:temp-A-3} into \eqref{E:G_2} and using \eqref{E:chi}, we obtain through straight forward calculations
\begin{equation} \label{E:G_2-esti-1}
\Vert G_2 \Vert_{L^{\frac 32} \cap L^2}  + \Vert \nabla G_2 \Vert_{L^{\frac 32} + L^{\frac 65}} \le C \Vert w\Vert_{X_1}^2 (\Vert w\Vert_{X_1}^3+1). 
\end{equation}
Here we have to estimate $\nabla G_2$ in $L^{\frac 32} + L^{\frac 65}$ since 
\[
\Vert \nabla \Delta \chi(D) (w_2^2)\Vert_{L^{\frac 32}} \le C \Vert w_2\Vert_{X_1}^2
\]
does not seem to have better decay and 
\[
\Vert \big(\vert U\vert ^{2}-\vert U_{c}\vert ^{2} \big) \nabla w_1\Vert_{L^{\frac 65}} \le C \Vert w\Vert_{X_1}^2 (\Vert w\Vert_{X_1}^3+1).
\]
does not seem to have better regularity. Similarly, for any $k\ge 1$, 
\begin{equation} \label{E:G_2-esti-2}
\Vert D_y^k G_2 \Vert_{W^{1, \frac 65} \cap \dot W^{1, \frac 32}} \le C \Vert w\Vert_{X_1}^2 (\Vert w\Vert_{X_1}^3+1). 
\end{equation}
The estimates for $G_{11}$ are 
\begin{equation} \label{E:G_11-esti-1}
\Vert G_{11} \Vert_{L^{\frac 32} \cap L^2} + \Vert \nabla G_{11} \Vert_{L^{\frac 32}\cap L^\infty + L^1 \cap L^{\frac 65}} \le C \Vert w\Vert_{X_1}^2 (\Vert w\Vert_{X_1}^3+1). 
\end{equation}
Again, we have to estimate $\nabla G_2$ in this norm as 
\[
\Vert \nabla \chi(D) (w_{2}\nabla w_{2}\cdot c) \Vert_{L^{\frac 32} \cap L^\infty}  \le C \Vert w\Vert_{X_1}^2 
\]
does not seem to have better decay and 
\[
\Vert \nabla \left( \left( |U|^2 - |U_c(\cdot +y)|^2 \right) w_{2} \right) \Vert_{L^1\cap L^{\frac 65}} \le C \Vert w\Vert_{X_1}^2 (\Vert w\Vert_{X_1}^3+1)
\]
does not seem to have better regularity. Differentiating in $y$ implies that 
\begin{equation} \label{E:G_11-esti-2}
\Vert D_y^k G_{11} \Vert_{W^{1,\frac 65} \cap \dot W^{1, \frac 32}} \le C \Vert w\Vert_{X_1}^2 (\Vert w\Vert_{X_1}^3+1), \quad k \ge 1. 
\end{equation}
Next we consider $G_{12}$. Recall from \eqref{lcy}, for any $f=(f_1, f_2) \in X_1$, 
\[
(L_{c, y} f)_1 = (2-\Delta) f_1 + \big( 3(u_c^2-1) + v_c^2\big)(\cdot +y) f_1 - c\cdot \nabla f_2 + 2 (u_cv_c) (\cdot +y)f_2.
\]
Using the Hardy's inequality, and the fact $K_{c, y}$ being an isomorphism, we obtain 
\[
\Vert (L_{c, y} K_{c, y} w)_1 + \Delta w_1 \Vert_{L^2} \le C \Vert w \Vert_{X_1}. 
\]
From $w_2 \Delta w_1 = \nabla \cdot (w_2 \nabla w_1) - \nabla w_2 \cdot \nabla w_1$ and \eqref{E:chi}, we have, for any $s \ge 0$, 
\begin{align*}
\Vert |\nabla|^s \chi(D) (w_2 \Delta w_1) \Vert_{L^{\frac 32} \cap L^\infty} \le C \Vert w_2\Vert_{\dot H^1} \Vert w_1\Vert_{\dot H^1}. 
\end{align*}
Therefore, \eqref{E:G_12}, \eqref{E:G_2-esti-1}, and the above inequalities yield 
\begin{equation} \label{E:G_12-esti-1}
\Vert |\nabla|^s G_{12} \Vert_{L^{\frac 32} \cap L^\infty} \le C\Vert w_2\Vert_{\dot H^1} \big( |\tilde y|  + \Vert w_1\Vert_{\dot H^1} + \Vert w\Vert_{X_1}^2 (\Vert w\Vert_{X_1}^3+1) \big).
\end{equation} 
Differentiating in $y$, we have, for $k\ge 1$,
\begin{align*}
D_y^k (L_{c, y}& K_{c, y} w)_1 = 2 D^k (u_cv_c) (\cdot +y) w_2\\
&+D^k \big( 3(u_c^2-1) + v_c^2\big)(\cdot +y) w_1 - (2-\Delta) \chi(D) \big( (D^k v_c)(\cdot +y) w_2\big) \\
 &- \sum_{k_1 + k_2 =k} D^{k_1} \big( 3(u_c^2-1) + v_c^2\big)(\cdot +y)\chi(D) \big( (D^{k_2} v_c)(\cdot +y) w_2\big). 
\end{align*}
By using \eqref{E:chi} and \eqref{E:G_2-esti-2}, we obtain, for any $k\ge 1$, 
\begin{equation} \label{E:G_12-esti-2}
\Vert |\nabla|^s D_y^k G_{12} \Vert_{L^{\frac 32} \cap L^\infty} \le C\Vert w_2\Vert_{\dot H^1} \big( |\tilde y|  + \Vert w_1\Vert_{\dot H^1} + \Vert w\Vert_{X_1}^2 (\Vert w\Vert_{X_1}^3+1) \big). 
\end{equation} 

Finally, we note that $G_{11}$, $G_{12}$, and $G_2$ are polynomials of $w$ and $\tilde y$ consisting of monomials of degree between 2 and 6 with coefficients depending on $U_c(\cdot +y)$.  Therefore, one may regrouping those monomials so that some of them belong to $W^{1, \frac 32}$ while others to $\dot W^{1, \frac 65}$. Moreover it is easy to obtain the estimates on $D_w^l D_y^k G$ and the the proof is complete. 
\end{proof}

\end{appendix}

\bibliography{gp}{}
\bibliographystyle{plain}

\end{document}